\documentclass[12pt,twoside]{memoir}

\usepackage[top=1.5in, bottom=1.5in, textwidth=5.5in, 
            hmarginratio=3:2]{geometry}

\usepackage{makeidx}
\makeindex

\usepackage{mycommands}
\usepackage{mystyle}

\renewcommand*{\backref}[1]{}
\renewcommand*{\backrefalt}[4]{%
    \ifcase #1 (Not cited.)%
    \or        (Cited on page~#2.)%
    \else      (Cited on pages~#2.)%
    \fi}

\chapterstyle{hline}

\newcommand{\UnderlinedTocSection}[1]{%
  \addtocontents{toc}{\protect\addvspace{1pc}%
         \protect\contentsline {chapter}%
        {\protect\Large\scshape\mdseries #1}{}{}%
        \hfill\par\protect\addvspace{-18pt}%
        \noindent\hrulefill\hspace*{24pc}\par}}

\hypersetup{
  pdfauthor=Cheyne Homberger,
  pdftitle=Dissertation-Homberger,
  }

\renewcommand{\scalebox}[2]{#2}
\begin{document}
\thispagestyle{empty}

\pretitle{ \begin{center}}
\title{\vspace{-6pc} \hrulefill \\ {\scshape \LARGE Patterns in Permutations and
    Involutions\\[6pt]}}
\posttitle{\par {\Large \scshape A Structural and Enumerative Approach}
\\[-4pt] \hrulefill \vfill \end{center}}
\author{{\small by} \\[4pc] Cheyne Homberger}
\predate{\vfill\begin{center}\normalsize}

\date{\vspace{4pc} A Dissertation Presented to the Graduate School of the University
of Florida in Partial Fulfillment of the Requirements for the Degree of \\
Doctor of Philosophy \\[1pc]
University of Florida \\[1pc]
2014}
\maketitle
\thispagestyle{empty}

\cleardoublepage 

\thispagestyle{empty}

\begin{center} \itshape
  \vspace*{.4\textheight}
  To Carol and Fred Gropper, my grandparents
\end{center}
\cleardoublepage

\frontmatter

\cleardoublepage

\chapter*{Acknowledgements}
\addcontentsline{toc}{chapter}{Acknowledgements}
\acknowledge{%
    
  First and foremost I'd like to thank my advisor, Mikl\'os B\'ona, for his
  guidance and encouragement throughout the research process, and for his
  patience and understanding during my meandering course through
  graduate school. 
  I also thank Vince Vatter, whose long discussions, advice, and friendship have
  helped make my time here more productive and more enjoyable.  Thanks also
  to Michael Albert for his support and suggestions during our collaborations,
  and to the remainder of my supervisory committee: Andrew Vince, Meera Sitharam,
  and Kevin Keating, each of whom have helped me to become a more well-rounded
  researcher.

  This dissertation is a product of the combined support of those around me, each
  of whom have left a profound impact both on this work and on my time in
  graduate school.  The graduate student community, with its many seminars and
  happy hours, has made the last five years more fun than it should have been.  I
  am grateful for all of my friends and colleagues, both for their support during
  the busy times and for their distractions during the slow.

  My time at the University of Florida has been marked with frequent
  diversions ---  organizing seminars and serving on administrative committees
  has kept me busy and interested, and I am thankful for all of the coworkers and
  friends I've met along the way.  Special thanks to Margaret, Connie, and
  the rest of the math department staff for helping me to find more travel
  funding than any graduate student deserves.  Finally, I am thankful for my
  students, who taught me to never stop looking for a simpler way to present a
  problem.

  I am grateful for my wonderfully supportive family, who have always always
  encouraged me in every endeavor, fostered every interest, and listened to me
  long before I had anything to say.  Finally, thank you to my best friend and
  favorite travel partner Elizabeth, for her support, editing skills, and
  understanding during the last five years, and for pushing me to be better in
  every way. 
}

\cleardoublepage

\tableofcontents

\cleardoublepage

\listoftables

\cleardoublepage

\listoffigures

\newpage
% ====================================================================== %

\cleardoublepage

\chapter*{Abstract}
\addcontentsline{toc}{chapter}{Abstract}
\begin{abstract}
  This dissertation presents a multifaceted look into the structural
  decomposition of permutation classes. The theory of permutation patterns is a
  rich and varied field, and is a prime example of how an accessible and
  intuitive definition leads to increasingly deep and significant line of
  research. The use of geometric structural reasoning, coupled with analytic
  and probabilistic techniques, provides a concrete framework from which to
  develop new enumerative techniques and forms the underlying foundation to
  this study. 

  This work is divided into five chapters. The first chapter introduces these
  techniques through working examples, both motivating the use of structural
  decomposition and showcasing the utility of their combination with
  analytic and probabilistic methods. The remaining chapters apply these
  concepts to separate aspects of permutation classes, deriving new
  enumerative, statistical, and structural results. These chapters are largely
  independent, but build from the same foundation to construct an overarching
  theme of building structure upon disorder.

  The main results of this study are as follows. Chapter~\ref{chap:expat}
  investigates the average number of occurrences of patterns with permutation
  classes, and proves that the total number of 231-patterns is the same in the
  classes of 132- and 123-avoiding permutations. Chapter~\ref{chap:involutions}
  applies structural decomposition to enumerate pattern avoiding involutions.
  Chapter~\ref{chap:polyclass} uses the theory of grid classes to develop an
  algorithm to enumerate the so-called polynomial permutation classes, and
  applies this to the biological problem of genetic evolutionary distance.
  Finally, we end in Chapter~\ref{chap:fixpat} with an exploration of
  pattern-packing, and determine the probability distribution for the number
  of distinct large patterns contained in a permutation. 
\end{abstract}

% ====================================================================== %
% \addtocontents{toc}{\protect\addvspace{10pt}
%   \noindent{\large \scshape Chapters}\protect\hfill\par\noindent\hrulefill}{}
% \addtocontents{toc}{\scshape \Large \noindent Chapters\hfill\par}{}
%\addcontentsline{toc}{part}{\scshape Chapters}

% \addtocontents{toc}{\contentsline{chapter}{\numberline {}Chapters}{}}
% \addtocontents{toc}{\protect\addvspace{1pc}%
%        \protect\contentsline {chapter}%
%       {\protect\Large\scshape\mdseries Chapters}{}{}%
%       \hfill\par\protect\addvspace{-18pt}\noindent\hrulefill\par}

\UnderlinedTocSection{Chapters}

\mainmatter

\cleardoublepage
\typeout{******************}
\typeout{**  Chapter 1   **}
\typeout{******************}

  % =========================================================================== %
  \chapter{Preliminaries}
  \label{chap:prelim}
  % =========================================================================== %

    Permutations are a fundamental mathematical concept used productively
    throughout the sciences to encode and understand disorder and rearrangement. 
    The theory of permutation patterns captures this geometric notion of
    disorder, and has yielded a wide variety of productive and surprising
    research over the past several decades. This dissertation presents several
    interrelated projects within this interesting and rapidly developing field.
    Structural, analytic, and probabilistic combinatorics are central to this
    work, and combine to provide unique insight into pattern enumeration. 

    This dissertation is organized as follows: Chapter~\ref{chap:prelim} provides
    an accessible introduction to the ideas and methods at play, followed by four 
    illustrative examples which serve to motivate and introduce the material
    to come. The following four chapters represent self-contained projects
    utilizing these techniques.  Each of these chapters is based partly on
    separate publications~\cite{me-expat,me-polyclass,me-fixpat,me-involutions},
    but together they speak to the utility of structural methods coupled with
    multivariate analysis.  Recursive structural decomposition intersected with
    modern analytic and probabilistic techniques has proven exceptionally useful
    in investigating patterns within permutations, and each chapter focuses on a
    separate facet of this productive combination. 
    
    For an accessible introduction to the field of combinatorics, the reader is
    directed to B\'ona~\cite{BonaWalk}. Stanley~\cite{Stanley1, Stanley2}
    provides a more advanced treatment to the subject as a whole, while
    B\'ona~\cite{BonaPerm} focuses on the combinatorics of permutations. 
    Wilf~\cite{wilfbook} gives an excellent introduction to the theory of
    generating functions, while Petkov\v{s}ek, Wilf, and Zeilberger~\cite{A=B}
    provide a survey of algorithmic methods. 
    Finally, analytic methods in combinatorics are presented best by Flajolet and
    Sedgewick~\cite{flajolet} and by Pemantle and Wilson~\cite{pemantle}, who
    focus on single- and multi-variate methods, respectively.

  % =========================================================================== %
  \section{Permutation Classes}
  \label{prelim:sec:perms}
  % =========================================================================== %

    \index{permutation}
    Permutations owe much of their rich structure to their variety of equivalent
    representations.  In this section we establish some of the basic notation and
    definitions of permutations and permutation classes. Throughout this
    dissertation, let $\Zgeq$ denote the non-negative integers $\{0, 1, 2, 3,
    \dots \}$, $\Pos$ the positive integers $\{1, 2, 3, 4, \dots\}$, and, for a given
    integer $n \in \Pos$, let $[n]$ denote the integers $\{1, 2, \dots n\}$. 

    \subsection{Permutations and Patterns}

      \begin{definition}\label{prelim:def:perm}
        For a given integer $n \in \Pos$, a \emph{permutation of length $n$}
        is a sequence $\pi = \pi_1 \pi_2 \dots \pi_n$ in which
        $\pi_i \in [n]$ and each integer of $[n]$ is used exactly once. There are
        $n!$ permutations of length $n$, the set of all of which is denoted $\S_n$. 
      \end{definition}

      For example, the six permutations of length three are as follows: 
      $$ \S_3 = \{123, 132, 213, 231, 312, 321\}. $$

      Permutations can be represented in many different ways, each leading to
      different generalizations. The above definition is known as the
      \emph{one-line representation} in the literature, and this approach leads
      naturally to the theory of permutation patterns. We start by presenting
      formal definitions of patterns before providing a geometric motivation.

      \begin{definition} \label{prelim:def:orderiso}
        For a positive integer $n$ any two sequences of distinct numbers $\alpha
        = \alpha_1 \alpha_2 \dots \alpha_n$ and $\beta = \bta_1 \bta_2 \dots
        \bta_n$, we say that $\alpha$ and $\bta$ are \emph{order isomorphic}
        (denoted $\alpha \sim \bta$) if
        $$ \alpha_i < \alpha_j \quad \text{if and only if} 
            \quad \bta_i < \bta_j.$$
      \end{definition}

      For example, the sequences $\alp = 9\ 2\ 4$ is order isomorphic to $\bta  = 5\
      1\ 3$, because their entries share the same relative order: the first is
      the biggest, the second is smallest, and the third lies in between. 

      It follows that each sequence $\alpha$ of $n$ distinct numbers is order
      isomorphic to a unique permutation of length $n$, called the \emph{standardization}
      of $\alpha$, and  denoted $\std(\alpha)$. For a given sequence $\alpha$, the
      standardization can be constructed by relabelling the smallest entry of of
      $\alpha$ by $1$, the second smallest by $2$, and so on (i.e., $\std(9\ 2\
      4) = 3\ 1\ 2$). We can now present the formal definition of permutation
      patterns. 

      \begin{definition} \label{prelim:def:patterns}
        Let $n, k \in \Pos$ with $k \leq n$, and let $\pi = \pi_1 \pi_2 \dots
        \pi_n \in \S_n$ and $\sg = \sg_1 \sg_2 \dots \sg_k \in
        \S_k$. Say that $\sg$ is \emph{contained as a pattern} in $\pi$ (denoted
        $\sg \prec \pi$) if there is some subsequence $1 \leq i_1 < i_2 < \dots
        <i_k \leq n$ such that 
        $$ \pi_{i_1} \pi_{i_2} \dots \pi_{i_k} \sim \sg_1 \sg_2 \dots \sg_k.$$
      \end{definition}

      Note that pattern containment is \emph{reflexive}
      ($\pi \prec \pi$ for all permutations $\pi$), \emph{transitive} ($\ro \prec
      \sg$, $\sg \prec \pi$ implies $\ro \prec \pi$), and \emph{anti-symmetric}
      ($\sg \prec \pi$ and $\pi \prec \sg$ implies $\sg = \pi$). These three
      properties mean that the set of all permutations, equipped with this
      ordering, forms a partially ordered set (a poset) known as the
      \emph{pattern poset}.

      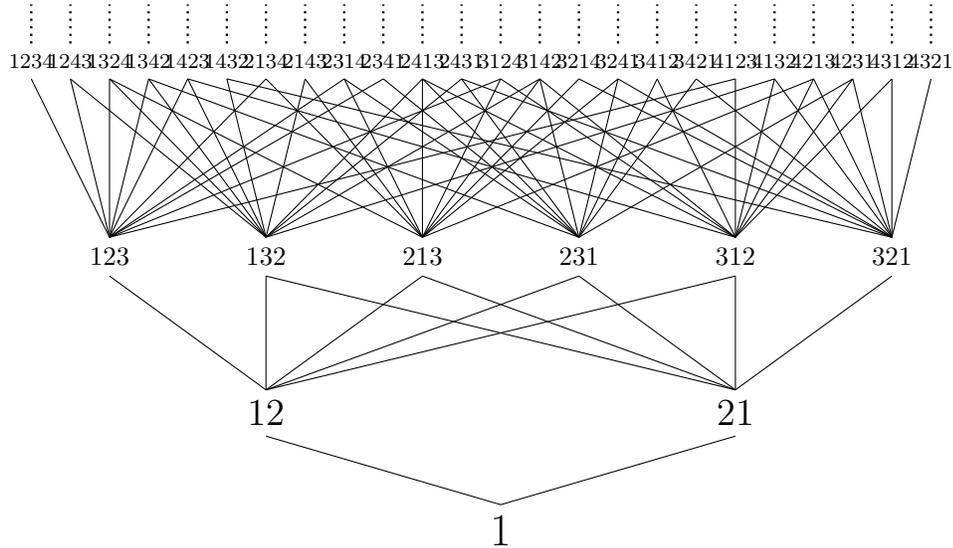
\begin{figure}[t] \centering
        \scalebox{.4}{
          \begin{tikzpicture}[
          every node/.style={},
           scale=.52]
            \node (1) at (12,0) {\Large 1};

            \node (12) at (6,3)  {\large 12};
            \node (21) at (18,3) {\large 21};

            \draw (12.south)-- (1.north);
            \draw (21.south) -- (1.north);

            \node (123) at (2 ,7) {\footnotesize 123};
            \node (132) at (6 ,7) {\footnotesize 132};
            \node (213) at (10,7) {\footnotesize 213};
            \node (231) at (14,7) {\footnotesize 231};
            \node (312) at (18,7) {\footnotesize 312};
            \node (321) at (22,7) {\footnotesize 321};

            \draw (123.south) -- (12.north);
            \draw (132.south) -- (12.north);
            \draw (213.south) -- (12.north);
            \draw (231.south) -- (12.north);
            \draw (312.south) -- (12.north);
            
            \draw (132.south) -- (21.north);
            \draw (213.south) -- (21.north);
            \draw (231.south) -- (21.north);
            \draw (312.south) -- (21.north);
            \draw (321.south) -- (21.north);
            
            \node (1234) at (0 ,12) {\tiny 1234};
            \node (1243) at (1 ,12) {\tiny 1243};
            \node (1324) at (2 ,12) {\tiny 1324};
            \node (1342) at (3 ,12) {\tiny 1342};
            \node (1423) at (4 ,12) {\tiny 1423};
            \node (1432) at (5 ,12) {\tiny 1432};

            \node (2134) at (6 ,12) {\tiny 2134};
            \node (2143) at (7 ,12) {\tiny 2143};
            \node (2314) at (8 ,12) {\tiny 2314};
            \node (2341) at (9 ,12) {\tiny 2341};
            \node (2413) at (10,12) {\tiny 2413};
            \node (2431) at (11,12) {\tiny 2431};

            \node (3124) at (12,12) {\tiny 3124};
            \node (3142) at (13,12) {\tiny 3142};
            \node (3214) at (14,12) {\tiny 3214};
            \node (3241) at (15,12) {\tiny 3241};
            \node (3412) at (16,12) {\tiny 3412};
            \node (3421) at (17,12) {\tiny 3421};

            \node (4123) at (18,12) {\tiny 4123};
            \node (4132) at (19,12) {\tiny 4132};
            \node (4213) at (20,12) {\tiny 4213};
            \node (4231) at (21,12) {\tiny 4231};
            \node (4312) at (22,12) {\tiny 4312};
            \node (4321) at (23,12) {\tiny 4321};

            % 123
            \foreach \p in {1234, 1243, 1324, 1342, 1423, 2134, 2314, 2341, 3124, 4123}
              \draw (\p.south) -- (123.north);
            % 132
            \foreach \p in {1243, 1324, 1342, 1423, 1432, 2143, 2413, 2431, 3142, 4132}
              \draw (\p.south) -- (132.north);
            % 213
            \foreach \p in {1324, 2134, 2143, 2314, 2413, 3124, 3142, 3214, 3241, 4213}
              \draw (\p.south) -- (213.north);
            % 231
            \foreach \p in {1342, 2314, 2341, 2413, 2431, 3142, 3241, 3412, 3421, 4231}
              \draw (\p.south) -- (231.north);
            % 312
            \foreach \p in {1423, 2413, 3124, 3142, 3412, 4123, 4132, 4213, 4231, 4312}
              \draw (\p.south) -- (312.north);
            % 321
            \foreach \p in {4321, 3421, 4231, 2431, 3241, 4312, 4132, 1432, 4213, 3214}
              \draw (\p.south) -- (321.north);

            \foreach \p in {1234, 1243, 1324, 1342, 1423, 1432,
                            2134, 2143, 2314, 2341, 2413, 2431,
                            3124, 3142, 3214, 3241, 3412, 3421,
                            4123, 4132, 4213, 4231, 4312, 4321}
              \draw[dotted, line width=.3mm] (\p.north) -- ++(0,1);
          \end{tikzpicture}
          }
      \caption[The first four levels of the permutation pattern poset.]{
            The first four levels of the permutation pattern poset. Two
            permutations are connected by a line if one is contained in the other
            as a pattern.}
      \label{prelim:fig:poset}
      \end{figure}

      The first four levels of this poset are shown in
      Figure~\ref{prelim:fig:poset}. Note that the number of lines going up from
      each permutation depends only on the length of the permutation, while the
      number going down varies. This will be a topic of study in
      Chapter~\ref{chap:fixpat}, where we will establish the probability
      distribution for the number of large patterns contained within randomly
      selected permutations. 

      If a permutation $\pi$ does not contain a pattern $\sg$, we say that $\pi$
      \emph{avoids} $\sg$. The set of all permutations which avoid a fixed
      pattern $\sg$ is denoted $\Av(\sg)$. Transitivity of pattern containment
      implies that if $\pi \in \Av(\sg)$ and $\ro \prec \pi$, then $\ro \in
      \Av(\sg)$. This relationship motivates our next definition.

      \begin{definition}\label{prelim:def:class}
        Let $\cP$ be a poset. A subset $S \subseteq \cP$ is called a
        \emph{downset} if it is closed downwards. That is, if $x \in S$ and $y
        \prec x$, then $y \in S$. 
        A downset of the permutation pattern poset is called a \emph{permutation
        class}. For a permutation class $\C$, denote by $\C_n$ the set of
        permutations of length $n$ in $\C$. 
      \end{definition}

      The set of all patterns which avoid some specified set of patterns are
      known as the \emph{avoidance classes}, and were first introduced by
      Knuth~\cite{Knuth} in the context of stack sorting.  The investigation of
      these and other classes has sparked a wide range of research over the past
      several decades, with a focus on enumeration.  In
      particular, the question of `which pattern is easiest to avoid?' has been a
      major open question for many years, and a variety of techniques have been
      developed to provide partial answers. The Marcus-Tardos
      Theorem~\cite{MarcusTardos} (which stood open as the Stanley-Wilf
      Conjecture for two decades) motivates much of this work.

      \begin{definition} \label{prelim:def:growthrate}
        Let $\C$ be a permutation class. The \emph{(upper) growth rate} of $\C$ is
        defined as the limit
        $$ \limsup_{n \ra \infty} \sqrt[n]{|\C_n|}.$$
      \end{definition}

      \begin{theorem}[Marcus, Tardos~\cite{MarcusTardos}] \label{thm:MarcusTardos}
        Every proper permutation class has a finite growth rate. 
      \end{theorem}

    \subsection{Wilf-Equivalence}
      
      Though Theorem~\ref{thm:MarcusTardos} says that all proper permutation classes
      have a finite growth rate, finding and classifying these growth
      rates is difficult. Of particular interest is identifying
      those patterns which have the same enumeration, i.e., $\bta, \tau$ such
      that $\Av_n(\bta) = \Av_n(\tau)$ for all $n$. Such a pair $\bta, \tau$ are
      called \emph{Wilf-equivalent}, and the set of all Wilf-equivalent
      permutations form a \emph{Wilf class}. Though showing Wilf-equivalence can be
      hard in general, many equivalences arise from eight trivial symmetries. 

      \begin{definition} \label{prelim:def:symmetries}
        Let $\pi = \pi_1 \pi_2 \dots \pi_n$ a permutation. The
        \emph{reverse}, the \emph{complement}, and the \emph{inverse} of $\pi$
        (denoted $\pi^r$, $\pi^c$, and $\pi^{-1}$, respectively) are defined as
        follows:
        $$ \begin{aligned}
          \left(\pi^r\right)_i &= \pi_{n-i + 1}, \\
          \left(\pi^c\right)_i &= n - \pi_{i} + 1, \text{ and } \\
          \left(\pi^{-1}\right)_{\pi_i} &= i. \\
        \end{aligned} $$
      \end{definition}
      
      Each of these operations map the set of permutations to itself, and each
      \emph{preserves pattern containment}. That is, if $\sg \prec \pi$, then
      $\sg^i \prec \pi^i$, for each $i \in \{r, c, -1\}$. It follows than that
      the class of permutations avoiding a pattern are in bijection with the
      class avoiding any symmetry of this pattern. These three symmetries
      thus generate an automorphism group of the pattern poset, which is
      isomorphic to the dihedral group of order eight. Of these three, only the
      inversion map has any fixed points; a permutation which is its own inverse
      is called an \emph{involution}.  It follows from Smith~\cite{Smith2006}
      that this is the complete set of automorphisms which respect pattern
      containment. Note that further order-respecting isomorphisms between classes
      are explored in Albert, Atkinson and Claesson~\cite{Claesson2013}. Note
      further that Wilf-classes need not contain bases of the same size: 
      Burstein and Pantone~\cite{pantone2014} recently showed the
      Wilf-equivalence of $\Av(1324, 3416725)$ and 
      $\Av(2143,3142,246135)$. 

      For permutations of length three, $123$ and $321$ are complements (and
      reverses) of each other, and thus the classes $\Av(123)$ and $\Av(321)$ 
      have the same enumeration (i.e., $|\Av_n(123)| = |\Av_n(321)|$ for all $n
      \in \Zgeq$). The permutation $132$ can be reversed to obtain $231$ or
      complemented to obtain $312$, and $312$ can be complemented to obtain
      $213$. Therefore the permutations $\{132, 213, 231, 312\}$ are
      Wilf-equivalent, and so there are at most two Wilf classes for length $3$
      permutations. 

      MacMahon, in 1915/16~\cite{PercyBook} enumerated the $123$-avoiding
      permutations while Knuth, in 1968~\cite{Knuth}, enumerated the
      $231$-avoiding permutations, leading to the first non-trivial Wilf
      equivalence. A bijection between $123-$ and $132$-avoiding permutations was
      presented by Simion and Schmidt~\cite{Simion1985} in 1985.

      \begin{theorem}[MacMahon, Knuth~\cite{PercyBook, Knuth}]
      \label{prelim:thm:knuth-equiv}
        The number of permutations of length $n$ avoiding $123$ is equal to the number
        avoiding $231$. 
      \end{theorem}

      We explore this result further in Sections~\ref{prelim:sec:av132}
      and~\ref{prelim:sec:av123}, and rederive this result using geometric
      constructions. Note that two Wilf-equivalent classes can have sharply
      contrasting structure, as we will soon see is the case for $\Av(123)$ and
      $\Av(132)$. 
      Theorem~\ref{prelim:thm:knuth-equiv} shows that there is only one Wilf
      class for length three patterns, which gives false hope for longer
      patterns.
      As we see here, the situation becomes much more complicated as patterns get
      longer. 

      Of the twenty-four patterns of length four, the trivial symmetries show
      that there are at most eight Wilf classes. 
      Non-trivial theorems from Babson and West~\cite{BabsonWest} and
      West~\cite{WestDiss} (and generalized in Backelin, West, and
      Xin~\cite{Backelin2007}) reduce this number to four, and 
      a result of Stankova~\cite{Stankova1994} shows that two of these
      remaining classes are Wilf-equivalent. This leaves the patterns of length
      four partitioned into three Wilf classes. That these three classes do in
      fact have different enumerations can be seen in the data presented in
      Table~\ref{prelim:tab:four-classes}. 

      \begin{table}[t] \centering
        \caption{Enumerations of the three Wilf classes for patterns of length
        four.}
        \label{prelim:tab:four-classes}
        \begin{tabular}{rrrrrrrrr}
        \multicolumn{9}{c}{$|\Av_n(\bta)|$} \\ 
         $n = $ & 1  & 2 & 3 & 4  & 5 & 6 & 7 & 8
          \\ \hline
        $\bta = 1342$ &
          1 & 2 & 6 & 23 & 103 & 512 & 2740 & 15485 \\
        $\bta = 1234$ &
          1 & 2 & 6 & 23 & 103 & 513 & 2761 & 15767 \\
        $\bta = 1324$ &
          1 & 2 & 6 & 23 & 103 & 513 & 2762 & 15793 \\
        \end{tabular}
      \end{table}

      Note that the monotone pattern is neither the easiest nor hardest to avoid,
      as one might expect.  These three cases speak to the complexity involved in
      enumerating permutation classes. The class $\Av(1342)$ was first counted by
      B\'ona~\cite{Bona1997}, and was found to have an algebraic generating
      function and an exponential growth rate  of $8$. The class $\Av(1234)$ was
      enumerated by Gessel~\cite{Gessel1990} and Regev~\cite{Regev}, who provided
      an exact formula the number of permutations of a given length in the class
      and showed that the exponential growth rate is $9$, but showed that the
      generating function is D-finite but nonalgebraic. Finally, the class
      $\Av(1324)$ has not been enumerated and the growth rate is unknown, except
      that it is between $9.42$ (Albert et.  al.~\cite{1324LowerBound}) and
      $13.93$ (B\'ona~\cite{1324UpperBound}). 

      The permutation $1324$ is a \emph{layered} permutation, meaning it can be
      written as a sequence of decreasing runs, the entries of which are each
      larger than the previous layer. Layered permutations were conjectured by
      Arratia~\cite{Arratia1999} to be the easiest to avoid, i.e., their
      avoidance classes have the fastest growth. This conjecture led to interest
      in these patterns~\cite{Claesson2012, Bona2007, Elder2005}, but was
      recently overturned by Fox~\cite{Fox}, who showed that the situation is
      much more complex than small examples suggest. 
      In Chapter~\ref{chap:involutions}
      we consider the problem of finding growth rates of pattern avoiding
      \emph{involutions}, and determine the growth rates of two such sets
      avoiding patterns of length four.

    \subsection{Geometric Motivation}
      
      The investigation and classification of Wilf classes is a deep and complex
      research program.  The primary focus of this dissertation, however, is on
      the \emph{geometric structure} of permutation classes, and the use of this
      structure to understand and explore pattern containment. The concepts
      presented above can all be reconsidered in a geometric context which allows
      for a more intuitive description of permutations
      and their patterns and symmetries. This geometric approach helps to
      illuminate new directions of research, is central to this work.

      \begin{definition} \label{prelim:def:permplot}
        The \emph{plot} of the permutation $\pi$ of length $n$ is
        the set of points $(i, \pi_i) \in \bR^2$ for each $i \in [n]$. 
      \end{definition}

      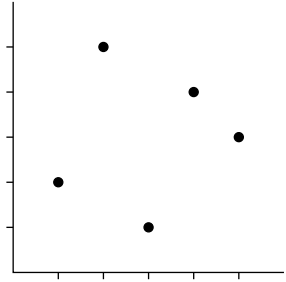
\begin{figure}[t]\centering 
        \begin{tikzpicture}
          [scale = .3, line width = .5pt]

          \draw (0,12) -- (0,0) -- (12,0);
          \foreach \num in {2, 4,..., 10}
          {
            \draw (\num, 0) -- (\num, -.3);
            \draw (0, \num) -- (-.3, \num);
          }
          \foreach \y [count = \x] in {2,5,1,4,3}
            \draw[fill = black] (2*\x,2*\y) circle (2mm);
        \end{tikzpicture}
        \caption{The plot of the permutation $\pi = 2\ 5\ 1\ 4\ 3$.}
        \label{prelim:fig:plot}
      \end{figure}

      The plot of a permutation is shown in Figure~\ref{prelim:fig:plot}. 
      Say that a set of $n$ points in $\bR^2$ is \emph{generic} if no two points
      lie on the same horizontal or vertical line. 
      Say that two generic sets $P$ and $T$ are \emph{order isomorphic} (written
      $P \sim T$) if the axes can be stretched or shrunk in some way to transform
      one into the other. 
      
      It follows that every generic point set is order isomorphic to
      a unique permutation plot, and that order isomorphism is an equivalence
      relation.
      The set of all $n$-element generic point sets, modulo this relation, is
      therefore in bijection with the set of all permutations of length $n$. This
      correspondence allows us to identify a permutation with its plot, and
      provides an alternate geometric definition of permutation patterns,
      illustrated in Figure~\ref{prelim:fig:plotpattern}.

      \begin{definition} 
        Let $n,k \in \Pos$ with $k \leq n$, and let $\pi \in \S_n$ and $\sg \in
        \S_k$. Let $P, T$ be the points in the plots of $\pi$ and $\sg$,
        respectively. Say that $\sg \prec \pi$ if there is some subset $R
        \subseteq S$ for which $R \sim T$. 
      \end{definition}

      \begin{figure}[t]
        \centering
        \begin{tikzpicture}[scale = .4,
                            every node/.style={circle, fill=black, inner sep=.5mm},
                            greyed/.style={fill=lightgray}]
          \node at (2,5) {};
          \node at (3,1) {};
          \node at (5,3) {};
        \end{tikzpicture}
          \hspace{1pc}
          \raisebox{2pc}{$\subset$}
          \hspace{1pc}
        \begin{tikzpicture}[scale = .4,
                            every node/.style={circle, fill=black, inner sep=.5mm},
                            greyed/.style={fill=lightgray}]
          \node[greyed] at (1,2) {};
          \node[] at (2,5) {};
          \node at (3,1) {};
          \node[greyed] at (4,4) {};
          \node at (5,3) {};
        \end{tikzpicture}
          \hspace{1pc}
          \raisebox{2pc}{$= $}
          \hspace{1pc}
        \begin{tikzpicture}[scale = .4,
                            every node/.style={circle, fill=black, inner sep=.5mm},
                            greyed/.style={fill=gray}]
          \foreach \y [count = \x] in {2,5,1,4,3}
          \node at (\x,\y) {};
        \end{tikzpicture}

        \caption{The permutation $\sg = 312$ is contained in the permutation $\pi =
                  25143$. }
        \label{prelim:fig:plotpattern}
      \end{figure}
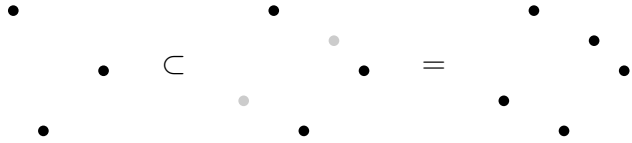

      Many operations on permutations are easier to understand
      through these geometric plots. For example, the plot of a permutation can
      be reflected and rotated to produce new permutations. Letting $\pi = \pi_1
      \pi_2 \dots \pi_n$ be a permutation, the reverse of $\pi$ is obtained by
      reflecting the dots across a vertical line, the complement by reflecting
      across a horizontal line, and the inverse is obtained by reflecting across
      the line $y = x$. That these operations generate a group of automorphisms
      isomorphic to the dihedral group of order eight is clear when
      viewing permutations as plots within a square. It is equally clear from
      this viewpoint that these operations respect pattern containment

      We can also define operations which act on pairs of permutations, combining
      two or more permutations into a single new one, and these operations can
      also be described entirely at the geometric level.  Two such examples are
      the \emph{direct sum} and \emph{skew sum} of permutations. 

      \begin{definition} \label{prelim:def:sums}
        Let $n,k \in \Pos$, and let $\pi \in \S_n$ and $\sg \in \S_k$. The
        \emph{direct sum} of $\pi$ and $\sg$, written $\pi \dsum \sg$, is the
        permutation defined by 
        $$ (\pi \dsum \sg)_i = 
            \left\{ \begin{array}{cc} 
              \pi_i & \text{if} \ i \leq n \\
              \sg_{i-n} + n & \text{if} \ i > n. 
              \end{array} \right. $$
        The \emph{skew sum}, written $\pi \ssum \sg$ is defined similarly:
        $$ (\pi \dsum \sg)_i = 
            \left\{ \begin{array}{cc} 
              \pi_i + k & \text{if} \ i \leq n \\
              \sg_{i-n} & \text{if} \ i > n. 
              \end{array} \right. $$
        A sum-indecomposable (resp. skew-indecomposable) permutation is one which
        \emph{cannot} be written as a direct (resp. skew) sum,
      \end{definition}

      Geometrically, $\pi \dsum \sg$ is the permutation whose plot is represented
      by placing the plot of $\pi$ below and to the left of the plot of $\sg$,
      while $\pi \ssum \sg$ places the plot of $\pi$ above and to the left of
      $\sg$, as shown in Figure~\ref{prelim:fig:sums}.  

      \begin{figure}[t] \centering
        \begin{tikzpicture}[scale=.3]
          \draw[dotted, fill = black!10] 
                (0,6) -- (6,6) -- (6,0) -- (0,0) -- cycle;
          \draw[dotted, fill = black!10] 
                (6,6) -- (6,12) -- (12,12) -- (12,6) -- cycle;
          \draw (0,0) -- (12,0) -- (12,12) -- (0,12) -- cycle;
          \node at (3,3) {$\pi$};
          \node at (9,9) {$\sg$};
        \end{tikzpicture}
        \hspace{4pc}
        \begin{tikzpicture}[scale=.3]
          \draw[dotted, fill = black!10] 
                (0,6) -- (6,6) -- (6,12) -- (0,12) -- cycle;
          \draw[dotted, fill = black!10] 
                (6,0) -- (6,6) -- (12,6) -- (12,0) -- cycle;
          \draw (0,0) -- (12,0) -- (12,12) -- (0,12) -- cycle;
          \node at (3,9) {$\pi$};
          \node at (9,3) {$\sg$};
        \end{tikzpicture}

        \caption{The plots of $\pi \dsum \sg$ and $\pi \ssum \sg$, respectively.}
        \label{prelim:fig:sums}
      \end{figure}

      These definitions will prove essential when describing permutation
      classes. In his thesis~\cite{SteveWaton2007}, Waton describes and explores
      classes defined entirely by points plotted on specified geometric shapes.
      We focus here, however, on more general classes. 

      Direct sums and skew sums are simple examples of the so called
      \emph{inflation} operation. A non-geometric definition of inflation is
      technical and unillustrative, but is natural when viewed as an operation of
      permutation plots. Before defining inflations, we need another definition
      which will itself prove useful. 

      \begin{definition}\label{prelim:def:interval}
        Let $\pi = \pi_1 \pi_2 \dots \pi_n \in \S_n$. An \emph{interval} of
        $\pi$ is a contiguous sequence of entries $\pi_i \pi_{i+1} \dots
        \pi_{i+k}$ whose values form a contiguous sequence of integers. 
      \end{definition}

      For example, in the permutation $\pi = 2743516$, the third,
      fourth, and fifth entries ($435$) form an interval. 
      Every permutation has an interval of size $n$ (the entire permutation) and
      intervals of size one (each entry).  Permutations which have only these
      trivial intervals are especially significant.

      \begin{definition} \label{prelim:def:simple}
        An permutation $\pi \in \S_n$ whose only intervals have size $1$ and $n$ is
        called \emph{simple}. 
      \end{definition}
      
      Simple intervals are useful for describing permutation classes,
      as we will see.  Monotone intervals will be investigated further in
      Chapters~\ref{chap:polyclass} and~\ref{chap:fixpat},
      and simplicity will be a major topic of Chapter~\ref{chap:involutions}. We
      can now define inflations, which will used throughout this dissertation. 

      \begin{definition}\label{prelim:def:inflation}
        Let $\pi \in S_n$, and let $\alp_1, \alp_2 \dots \alp_n$ be 
        permutations of any length. The \emph{inflation} of $\pi$ by the
        permutations $\alp_i$ is defined as the permutation obtained by replacing
        the $i$th entry of $\pi$ with an interval which is order isomorphic to
        the permutation $\alp_i$.  This inflation is denoted 
        $$ \pi[\alp_1, \alp_2, \dots \alp_n].$$
      \end{definition}

      \begin{figure}[t] \centering
        \begin{tikzpicture}[scale=.3]
          \draw (.5,.5) -- (9.5,.5) -- (9.5,9.5) -- (.5,9.5) -- cycle;
          \foreach \i in {2.8, 5, 7.2}{
            \draw[dotted] (.5,\i) -- (9.5, \i);
            \draw[dotted] (\i,.5) -- (\i, 9.5);
          }
          \draw[fill=black] (2, 4) circle (2mm);
          \draw[fill=black] (4, 8) circle (2mm);
          \draw[fill=black] (6, 2) circle (2mm);
          \draw[fill=black] (8,6) circle (2mm);
        \end{tikzpicture}
        \hspace{1pc}
        \begin{tikzpicture}[scale=.3]
          % \draw[fill=black!15] (.75,6.25)--(3.25,6.25)--(3.25,3.75)--(.75,3.75)--cycle;
          % \draw[fill=black!15] (3.75,9.25)--(5.25,9.25)--(5.25,7.75)--(3.75,7.75)--cycle;
          % \draw[fill=black!15] (5.75,3.25)--(8.25,3.25)--(8.25,.75)--(5.75,.75)--cycle;
          % \draw[fill=black!15] (8.75,7.25)--(9.25,7.25)--(9.25,6.75)--(8.75,6.75)--cycle;

          \draw[fill=black!15, dotted] 
              (.5,3.5)--(3.5,3.5)--(3.5,6.5)--(.5,6.5)--cycle;
          \draw[fill=black!15, dotted] 
              (3.5,7.5)--(5.5,7.5)--(5.5,9.5)--(3.5,9.5)--cycle;
          \draw[fill=black!15, dotted] 
              (5.5,.5)--(8.5,.5)--(8.5,3.5)--(5.5,3.5)--cycle;
          \draw[fill=black!15, dotted] 
              (8.5,6.5)--(9.5,6.5)--(9.5,7.5)--(8.5,7.5)--cycle;

          \draw (.5,.5) -- (9.5,.5) -- (9.5,9.5) -- (.5,9.5) -- cycle;
          \foreach \i in {3.5, 5.5, 8.5}{
            \draw[dotted] (\i,.5) -- (\i, 9.5);
          }
          \foreach \i in {3.5, 6.5, 7.5}{
            \draw[dotted] (.5,\i) -- (9.5, \i);
          }
          \foreach \y [count = \x] in {5,4,6,9,8,1,3,2,7}
            \draw[fill = black] (\x,\y) circle (2mm);
        \end{tikzpicture}
        \hspace{1pc}
        \begin{tikzpicture}[scale=.3]
        \draw (.5,.5) -- (9.5,.5) -- (9.5,9.5) -- (.5,9.5) -- cycle;

          \foreach \y [count = \x] in {5,4,6,9,8,1,3,2,7}{
            \draw[fill = black] (\x,\y) circle (2mm);
            \draw[dotted] (\x + 0.5, 0.5) -- (\x + 0.5, 9.5);
            \draw[dotted] (0.5, \x + 0.5) -- (9.5, \x + 0.5);
            }
      
        \end{tikzpicture}
        \caption[The simple permutation $2413$ and its inflation]{
                  The simple permutation $2413$ and its inflation
                  $2413[213, 21, 132, 1] = 546\ 98\ 132\ 7$.} 
        \label{prelim:fig:inflation}
      \end{figure}
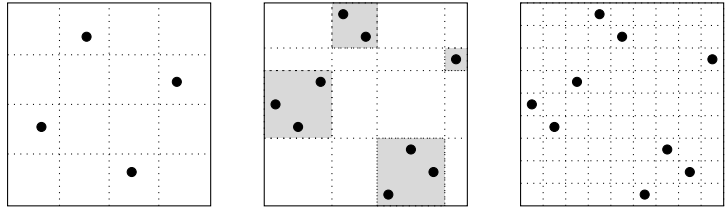

      For example, for any two permutations $\pi$ and $\sg$, $\pi \dsum \sg =
      12[\pi, \sg]$ and $\pi \ssum \sg = 21[\pi, \sg]$. A more complicated
      example is shown in Figure~\ref{prelim:fig:inflation}. While simple
      permutations and inflations are useful for working with and describing
      permutations, their true utility is illustrated in the following theorem,
      which has generalizations to a wider range of combinatorial
      objects~\cite{SubsDecomp}.

      % TODO: is this really necessary?
      % Note that in the literature, inflations by empty permutations are
      % generally not allowed. Throughout this document (and especially in
      % Chapter~\ref{chap:polyclass}), however, we go against tradition and allow
      % inflations by empty permutations, except where specified. This departure
      % is motivated partly by the fact that the set of all inflations of a
      % permutation (if we allow empty permutations) forms a permutation class.
      % We therefore define a \emph{proper} inflation of a permutation $\pi$
      % to be one in which each entry of $\pi$ is inflated by at least one
      % element. Allowing empty inflations will be useful in geometric class
      % descriptions. 

      \begin{theorem}[Substitution Decomposition~\cite{Brignall2008}]
        \label{thm:subsdecomp}
        Every permutation $\pi$ can be written as the inflation of a
        unique simple permutation. Further, if $\pi = \sg[\alp_1, \dots \alp_m]$,
        where each $\alp_i$ is a permutation of length $\geq 1$ and
        $m \geq 4$, then the permutations $\alp_i$ are uniquely determined as
        well. 
      \end{theorem}

  \section{Dyck Paths and the Catalan Numbers}

      Before exploring two examples of permutation classes, we take a brief
      detour and investigate another set of combinatorial objects known as 
      \emph{Dyck paths}. These paths will be used throughout this dissertation,
      and provide a convenient and flexible means of encoding recursive and
      structural information. 
      
      These paths are enumerated by the so-called
      \emph{Catalan numbers}, a ubiquitous and useful sequence of integers.
      Stanley~\cite{Stanley2} has famously collected a series a sixty-six
      examples of combinatorial objects, each enumerated by these numbers. Their
      pervasiveness is due in part to their multiple recursive descriptions.

    \subsection{Paths on the Integer Lattice}
      
      At its most formal, a \emph{Dyck path} of semilength $n$ is a sequence
      $\vec v_1, \vec v_2, \dots \vec v_{2n}$ of vectors $ \vec v_i \in
      \{\vect{1,1}, \vect{1,-1}\}$, satisfying $ \sum_{n=1}^{2n} \vec v_i =
      \vect{2n, 0}$ and, for all integers $k \in [2n]$ and $\vect{x,y} =
      \sum_{n=1}^k \vec v_i$, we have that $y \geq 0$. 

      As usual, a more intuitive definition will be useful. 
      Suppose that, starting from the point $(0,0) \in \mathbb{R}^2$, we want to
      travel to the point $(2n,0)$. Suppose further that are only allowed to walk
      diagonally northeast 
      (from a point $(x,y)$ to $(x+1, y+1)$) or southeast (from a point $(x,y)$
      to $(x+1, y-1)$). 
      Call a northeast step an \emph{upstep} and a southeast step a
      \emph{downstep}. The total number of walks from $(0,0)$ to $(2n,0)$ is then
      $\binom{2n}{n}$, since the number of up steps must equal the number of down
      steps, and so we need only specify which of the $2n$ steps are up. Dyck
      paths can now be defined as follows. 

      \begin{definition}\label{prelim:def:dyckpath}
        A \emph{Dyck path} of semilength $n$ (or of length $2n$) is path $p =
        s_1s_2 \dots s_{2n}$ from $(0,0)$ to $(2n,0)$ using the steps $u =
        \vect{1,1}$ and $d = \vect{1,-1}$ which \emph{never passes below the
        line $y=0$}. 
      \end{definition}

      These paths can be represented as a string of symbols from the alphabet
      $\{u,d\}$, representing upsteps and downsteps, respectively. The path $p =
      uuuddududduuddd$ is shown in Figure~\ref{prelim:fig:dyckpath}. 

      \begin{figure}[t]\centering
        \begin{tikzpicture} [scale = .4]
          \draw (0,0) -- (16,0);
          \draw[->, >= latex] (0,0) -- (1,1);
          \draw[->, >= latex] (1,1) -- (2,2);
          \draw[->, >= latex] (2,2) -- (3,3);
          \draw[->, >= latex] (3,3) -- (4,2);
          \draw[->, >= latex] (4,2) -- (5,1);
          \draw[->, >= latex] (5,1) -- (6,2);
          \draw[->, >= latex] (6,2) -- (7,1);
          \draw[->, >= latex] (7,1) -- (8,2);
          \draw[->, >= latex] (8,2) -- (9,1);
          \draw[->, >= latex] (9,1) -- (10,0);
          \draw[->, >= latex] (10,0) -- (11,1);
          \draw[->, >= latex] (11,1) -- (12,2);
          \draw[->, >= latex] (12,2) -- (13,3);
          \draw[->, >= latex] (13,3) -- (14,2);
          \draw[->, >= latex] (14,2) -- (15,1);
          \draw[->, >= latex] (15,1) -- (16,0);
        \end{tikzpicture}
        \caption{A Dyck path of semilength $8$.}
        \label{prelim:fig:dyckpath}
      \end{figure}
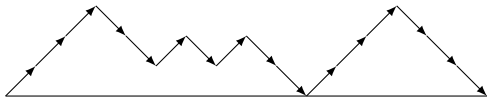

    \subsection{Enumerating Dyck Paths}

      Dyck paths are a fundamental combinatorial object, and their properties
      have been studied extensively~\cite{Chapman1999, Deutsch1999, Denise1995}. 
      Their well understood structure makes them (and their generalizations) a
      useful intermediate object for building bijections between other
      objects~\cite{Claesson2008, bloomvince}. To illustrate their recursive
      structure, we derive their enumeration here. 

      In order to count Dyck paths, we first need to consider their structure,
      and how they can be broken down into smaller pieces. We focus on two
      separate decompositions, which lead to two different recursive
      descriptions, each of which leads to the Catalan numbers. 

      First, let $p = s_1 s_2 \dots s_{2n}$ be a Dyck path, and let $s_i$ be the
      first step which brings it back to the line $y = 0$. Such a step must
      exist, since $s_{2n}$ always ends at this line. It follows
      then that $i$ is even, $s_1 = u$, $s_i = d$, and $s_{i+1} s_{i+2} \dots
      s_{2n}$ is a Dyck path of length $2n-i$. 
      Further, since $s_i$ is the \emph{first} time
      the path touches the line $y=0$, each of the steps $s_2, s_3, \dots
      s_{i-1}$ have a height greater than or equal to $1$, which implies that
      $s_2 s_3 \dots s_{i-1}$ is a Dyck path. This implies that for every Dyck
      path $p$, there exist two smaller Dyck paths $p_1, p_2$ such that $$ p = u
      p_1 d\ p_2 .$$

      It follows that if $\cP$ is the \emph{language} of Dyck paths (i.e., the
      set of all strings of the letters $u,d$ which represent valid Dyck paths),
      then $\cP = u \cP d \cP + \eps$, where the $\eps$ represents the empty
      path. This leads immediately to a generating function relation: if we let
      $c_n$ be the number of Dyck paths of semilength $n$ and $C(z) = \sum_{n\geq
      0} c_n z^n$, then this relation leads to the equation 

      \begin{equation}  \label{prelim:eqn:firstpass}   
        C(z) = z C(z)^2 + 1. 
      \end{equation}

      Before investigating further, we present an alternate decomposition. Let $p
      = s_1 s_2 \dots s_{2n}$ be a Dyck path, and let $i_1, i_2, \dots i_k$ be
      all of the indices with the property that the step $s_i$ ends on the line
      $y=0$. It follows then that each subword $s_{i_j + 1} s_{i_j + 2} \dots
      s_{i_{j+1} -1}$ stays above the line $y=1$, and is therefore itself a Dyck
      path. Therefore, for all Dyck paths $p$, there exist some integer $k$ and
      Dyck paths $p_1, p_2, \dots p_k$ such that 
      $$ p = u p_1 d\  u p_2 d\  \dots u p_k d.$$
      This gives an alternate relation for the generating function $C(z)$
      enumerating Dyck paths:
      
      \begin{equation} \label{prelim:eqn:allpass}
        C(z) = 1 + zC(z) + z^2 C(z)^2 + z^3 C(z)^3 + \dots = \frac{1}{1 - zC(z)}.
      \end{equation}

      The equivalence of equations~\ref{prelim:eqn:allpass}
      and~\ref{prelim:eqn:firstpass} is immediately obvious --- one can be
      rearranged into the other.
      It follows then that these two seemingly different recursions are in fact
      equivalent, and so any object exhibiting either of these recursive
      descriptions are counted by the same numbers. 
      With Dyck paths, both recurrences are clear;
      with other objects, however, they are less transparent. Dyck paths are
      useful in part because of the simplicity of their decompositions, and
      Catalan numbers are ubiquitous because they capture so many of these
      recursions. 

    \subsection{The Catalan Numbers}
    \label{prelim:sub:catalan}
      
      The generating function presented above
      (equation~\ref{prelim:eqn:firstpass}) can be solved using the quadratic
      formula, yielding the following (note that the quadratic formula actually
      yields two solutions, but we discard the one which does not have a series
      expansion with positive integer coefficients) 

      \begin{equation} \label{eqn:catalan-gfn}
        C(z) = \sum_{n \geq 0} c_n z^n = \frac{1 - \sqrt{1 - 4z}}{2z}.
      \end{equation}

      The first few coefficients in the expansion of $C(z)$ are
      $1,1,2, 5, 14, 42, 132, \dots$, and are \OEIS{A000108}.  The generating
      function recurrence $C(z) = zC(z)^2 + 1$ translates to $c_0 =1$ and
      $c_{n+1} = \sum_{k=0}^n c_{k} c_{n-k}$, and this uniquely defines this
      sequence.  The binomial theorem can be used to obtain an exact formula for
      $c_n$ from equation~\ref{eqn:catalan-gfn} above:

      \begin{equation} \label{eqn:catalan-exact}
        c_n = \frac{1}{n+1} \binom{2n}{n} = \binom{2n}{n} - \binom{2n}{n-1}.
      \end{equation}

      We note that the generating function presented above
      (equation~\ref{eqn:catalan-gfn}) has a singularity at $z=1/4$. It follows
      that, when expanded as a power series about $z=0$, $C(z)$ has a radius of
      convergence of $1/4$. The exponential growth rate of a sequence is equal to
      the reciprocal of the radius of convergence, which implies that $\limn
      \sqrt[n]{c_n} = 4$. 
      While Stirling's approximation for the factorials gives a simpler means of
      calculating this growth rate (and allow for the derivation of the
      subexponential growth rate), analytic techniques, summarized in the
      textbook of Flajolet and Sedgewick~\cite{flajolet}, provide a wide
      framework for deriving these exponential growth rates.

  \section{Four Case Studies}

    The advantage to this geometric focus is best illustrated through 
    examples. In this section we present four case studies, each of which
    corresponds roughly to the subject of a later chapter. Together these provide
    motivation and a gentle introduction to the methods used throughout this
    dissertation.
    
    We begin by deriving the 
    enumeration of the classes of $132$- and $123$-avoiding permutations. Though
    they share the same enumeration, these two classes present starkly different
    decompositions. We then combine these ideas and explore the class of
    \emph{$123$- and $231$}-avoiding permutations, motivating the investigation
    of polynomial permutation classes. Finally, we examine an example of the use of
    probabilistic techniques and structural decomposition in finding
    statistical information about classes.

    \subsection{Permutations Avoiding 132}
    \label{prelim:sec:av132}

      We start with the enumeration of the class $\Av (132)$.  The study of
      simples within a permutation class has been a deep and productive line of
      research in recent years~\cite{Brignall2008, Atkinson2005, Brignall2007}.  
      Further, this investigation has seen numerous applications in the
      enumeration of classes~\cite{pantone2013,pantone2014, Albert2012}.
      %
      % TODO add line about 'numerous applications' and cite more people
      %
      While the vast majority of this machinery is not needed for the class $\Av
      (132)$, but in the interest of exposition we hit a small nail with a large
      hammer.  The enumeration of a class using its simples is the core idea of
      Chapter~\ref{chap:involutions}, where we apply it to sets of
      pattern-avoiding involutions. 
      
      A plot of a permutation within $\Av (132)$ has strict restrictions: every
      element to the left of the highest point must be higher than every element
      to the right, since otherwise we would have a $132$ pattern with the
      highest element playing the role of the 3. This highest element then
      divides the plot into two sides. It follows that every entry after the
      peak forms an interval, which implies that the only simples in $\Av (132)$
      are $\{1, 12, 21\}$.

      By describing the simple permutations in the class, we can often obtain a
      full enumeration. The class $\Av (132)$ is uncomplicated enough to be
      described entirely using direct and skew sums, but it falls into a larger
      set of classes, those which have only finitely many simple permutations.
      Such permutation classes posess a number of useful properties, including
      the following theorem, due to Albert and Atkinson. 

      \begin{theorem}[Albert, Atkinson~\cite{Atkinson2005}]
        If a class contains only finitely many simple permutations, then its
        enumeration is given by an algebraic generating function. 
      \end{theorem}

      In addition to theoretical results, the investigation of simple
      permutations and decomposition has led to practical enumeration techniques.
      Once the simples of a class have been obtained, one needs only determine
      the manner in which each simple can be inflated in order to fully describe
      the class. While much of this work has focused on enumerating classes, it
      can also be used to obtain statistical information about the class.
      Section~\ref{prelim:sec:ascents-example} gives an introductory example to
      this technique, while Chapter~\ref{chap:expat} explores the concept
      further. 

      Returning now to the class $\Av(132)$, note that arbitrary inflations of
      the simple permutations $\{1, 12, 21\}$ do not lead to $132$-avoiding
      permutations. Letting $\pi \in \Av(132)$, recall that every entry after the
      maximal entry must have a smaller value than every entry before. The
      substitution decomposition (Theorem~\ref{thm:subsdecomp}) implies that 
      each permutation can be defined as an inflation of precisely one of these:
      the simple permutation $1$ can only be inflated to the length $1$
      permutation, inflations of
      $12$ are the sum-decomposable elements, and the skew-decomposable elements
      are the inflations of $21$. 
      
      For an inflation of $12$, the $2$ can only be inflated by an
      increasing run of entries, or else would contain a $21$ pattern, creating a
      $132$ occurrence with any entry of the inflation of the $1$, which can be
      inflated by any $132$-permutation.  Recall that the substitution
      decomposition does not guarantee uniqueness when inflating the simple
      permutations $12$ and $21$, so we have to be careful. To ensure uniqueness,
      only allow the $2$ of $12$ to be inflated by a single element (if there is
      an increasing run, take it to be part of the $1$).

      Finally, when inflating $21$, the $1$ can be inflated by any
      $132$-permutation, while the $2$ can be inflated by any $132$-avoiding
      permutation which ends in its last element, which can be represented as the
      direct sum of a $132$-avoiding permutation (or the empty permutation) with
      the permutation $1$. We express this as follows, letting $\C$ denote $\Av
      (132)$ and $\eps$ denote the empty permutation:

      $$ \C =  1\big[1\big] \ \bigcup \ 12\big[\C, 1\big] \ \bigcup \
          21\big[(\C \cup \eps) \dsum 1, \C\big]. $$

      Letting $f$ denote the generating function $\sum_{n \geq 0} |\Av_n (132)|
      z^n$, this leads to the following expression

      $$ f =  z + fz  + z(f+1)f .$$

      Solving for $f$ using the quadratic formula gives that $f$ is the generating
      function for the Catalan numbers with the constant term subtracted off.
      This gives an exact formula for the enumeration of $\Av (132)$, 
      as originally derived by Knuth~\cite{Knuth}. 
      
      \begin{theorem} \label{prelim:thm:av132}
        The number of permutations of length $n$ avoiding $132$ is the $n$th Catalan
        number $c_n = \frac{1}{n+1}\binom{2n}{n}$. 
      \end{theorem}

      Note that this result can be obtained using more elementary methods. 
      It follows that a permutation is $132$-avoiding if and only if it can be
      written as $(\pi \dsum 1) \ssum \sg$, where $\pi$ and $\sg$ are
      $132$-avoiding permutations (or empty). Applying this characterization
      iteratively provides a recursive description of the $132$-avoiding
      permutations, shown in Figure~\ref{prelim:fig:av132-first}, and in fact
      characterizes this class.

      \begin{figure}[t] \centering
        \begin{tikzpicture}[scale = .3]
          \draw[dotted, fill=black!2] (0,0) -- (12,0) -- (12,12) -- (0, 12) -- cycle;
          \node at (6,6) {$\C$};

        \end{tikzpicture}
        \hspace{1pc}
        \raisebox{4pc}{$= $}
        \hspace{1pc}
        \begin{tikzpicture}[scale = .3]
          \draw (0,0) -- (12,0) -- (12,12) -- (0, 12) -- cycle;

          \draw[fill = black] (6,12) circle (2mm);
          \draw[dotted, fill=black!5] (.25,11.75) -- (5.75,11.75) 
                                -- (5.75,6.25) -- (.25,6.25) -- cycle;
          \draw[dotted, fill=black!5] (6.25,5.75) -- (11.75,5.75) 
                                -- (11.75,.25) -- (6.25,.25) -- cycle;
          \node at (3,9) {$\C$};
          \node at (9,3) {$\C$};
        \end{tikzpicture}
        \hspace{1pc}
        \raisebox{4pc}{$= $}
        \hspace{1pc}
        \begin{tikzpicture}[scale = .3]
          \draw (0,0) -- (12,0) -- (12,12) -- (0, 12) -- cycle;

          \draw[fill = black] (6,12) circle (2mm);
          \draw[dotted, fill=black!5] (.25,11.75) -- (5.75,11.75) 
                                -- (5.75,6.25) -- (.25,6.25) -- cycle;
          \draw[dotted, fill=black!5] (6.25,5.75) -- (11.75,5.75) 
                                -- (11.75,.25) -- (6.25,.25) -- cycle;
          
          \draw[fill=black] (3, 11.75) circle (2mm);
          \draw[fill=black] (9, 5.75) circle (2mm);

          \draw[fill=black!20] (.5, 11.5) -- (2.9,11.5)
                                -- (2.9, 9.1) -- (.5,9.1) -- cycle;
          \draw[fill=black!20] (3.1, 8.9) -- (5.5,8.9)
                                -- (5.5, 6.5) -- (3.1,6.5) -- cycle;

          \draw[fill=black!20] (6.5, 5.5) -- (8.9, 5.5) 
                                -- (8.9, 3.1) -- (6.5, 3.1) -- cycle;
          \draw[fill=black!20] (9.1, 2.9) -- (11.5, 2.9) 
                                -- (11.5, .5) -- (9.1, .5) -- cycle;

          \node at (1.7, 10.3) {$\C$};
          \node at (4.3, 7.7) {$\C$};
          \node at (7.7, 4.3) {$\C$};
          \node at (10.3, 1.7) {$\C$};

        \end{tikzpicture}
      \caption{A geometric description of the class $\C = \Av 132$.}
      \label{prelim:fig:av132-first}
      \end{figure}
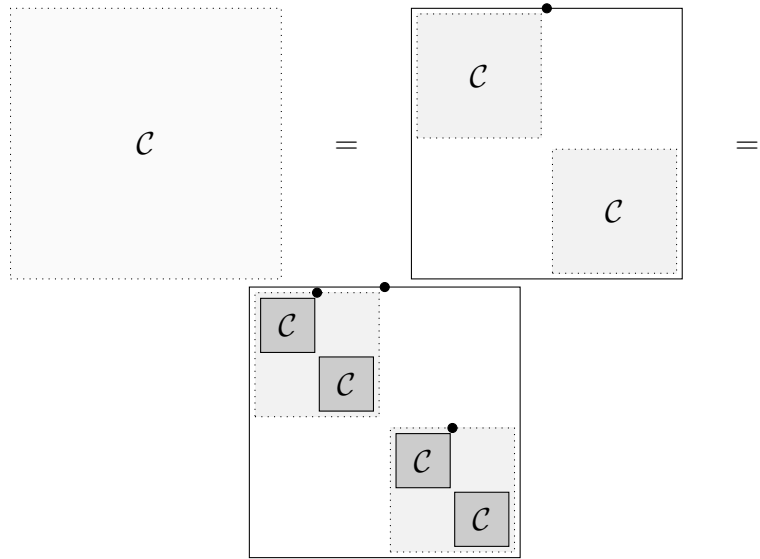

      This recursive decomposition can be used to generate a recursively defined
      bijection $\phi: \Av_n(132) \ra \cD_n$ from permutations in $\Av_n (132)$
      to Dyck paths of semilength $n$, thus reproving 
      Theorem~\ref{prelim:thm:av132} once again. Let $\pi \in \Av_n(132)$, and $\pi = (\pi_1
      \dsum 1) \ssum \pi_2$ be the decomposition defined above. Then define 
      $$\phi(\pi) = u \ \phi(\pi_1) \ d \ \phi(\pi_2).$$
      This recursive definition was originally presented by Knuth~\cite{Knuth}.
      For example, 
      $$ \begin{aligned}
        \phi(74352681) 
          &= u\phi(743526)d\phi(1) \\
          &= u (ud \phi(43526)) d ud \\
          &= uud (u\phi(4352)d) dud \\
          &= uudu (u\phi(43)d)\phi(2) dud\\
          &= uuduu (ud\phi(3))d (ud) dud \\
          &= uuduuud (ud) duddud \\
          &= uuduuududduddud \in \cD_n.
      \end{aligned} $$
      
      There is an alternate, non-recursive bijection $\varphi$, first presented
      in an alternate, non-geometric form by
      Krattenthaler~\cite{Krattenthaler2001},
      whose equivalence to the above definition follows from the work of Claesson
      and Kitaev~\cite{Claesson2008}. Let $\pi \in \Av_n(132)$, and define
      $\varphi(\pi)$ as follows. First, plot $\pi$ and define a lattice path from
      $(1,n)$ to $(n,1)$ using the steps $\{\vect{0,-1}, \vect{1,0}\}$. Take this
      to be the unique path using these steps which maximizes the area underneath
      the path, while remaining below and to the left of each entry of the
      plotted permutation. Finally, translate this to a Dyck path by mapping each
      $\vect{0,-1}$ to be an up step, and each $\vect{1,0}$ to be a down step.
      See Figure~\ref{prelim:fig:dyckbiject-geom} for an example. 

      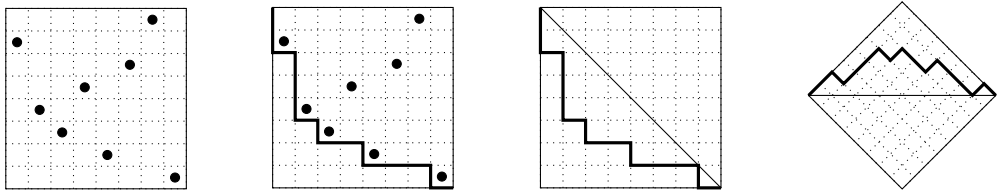
\begin{figure}[t] \centering
        \begin{tikzpicture}[scale=.3]
          \draw (1,1) -- (1,9) -- (9,9) -- (9,1) -- cycle;
          \foreach \y [count = \x] in {7,4,3,5,2,6,8,1}{
            \draw[fill = black] (\x+.5,\y+.5) circle (2mm);
          }
          \foreach \i in {1, ...,9}{
            \draw[dotted] (\i,1) -- (\i,9);
            \draw[dotted] (1,\i) -- (9,\i);
          }
        \end{tikzpicture}
        \hspace{2pc}
        \begin{tikzpicture}[scale=.3]
          \draw (1,1) -- (1,9) -- (9,9) -- (9,1) -- cycle;
          \foreach \y [count = \x] in {7,4,3,5,2,6,8,1}{
            \draw[fill = black] (\x+.5,\y+.5) circle (2mm);
          }
          \foreach \i in {1, ...,9}{
            \draw[dotted] (\i,1) -- (\i,9);
            \draw[dotted] (1,\i) -- (9,\i);
          }
          \draw[very thick] (1,9) --++ (0,-2) --++ (1,0) --++ (0,-3)
                            --++ (1,0) --++(0,-1) --++(2,0) 
                            --++ (0,-1) --++ (3,0) --++(0,-1) 
                            --++ (1,0);
        \end{tikzpicture}
        \hspace{2pc}
        \begin{tikzpicture}[scale=.3]
          \draw (1,1) -- (1,9) -- (9,9) -- (9,1) -- cycle;
          \foreach \i in {1, ...,9}{
            \draw[dotted] (\i,1) -- (\i,9);
            \draw[dotted] (1,\i) -- (9,\i);
          }
          \draw (1,9) -- (9,1);
          \draw[very thick] (1,9) --++ (0,-2) --++ (1,0) --++ (0,-3)
                            --++ (1,0) --++(0,-1) --++(2,0) 
                            --++ (0,-1) --++ (3,0) --++(0,-1) 
                            --++ (1,0);
        \end{tikzpicture}
        \hspace{2pc}
        \begin{tikzpicture}[scale=.22]
        \begin{scope}[shift={(3,5)},rotate=-45, yscale=-1]
          \draw (1,1) -- (1,9) -- (9,9) -- (9,1) -- cycle;
          \foreach \i in {1, ...,9}{
            \draw[dotted] (\i,1) -- (\i,9);
            \draw[dotted] (1,\i) -- (9,\i);
          }
          \draw (1,9) -- (9,1);
          \draw[very thick] (1,9) --++ (0,-2) --++ (1,0) --++ (0,-3)
                            --++ (1,0) --++(0,-1) --++(2,0) 
                            --++ (0,-1) --++ (3,0) --++(0,-1) 
                            --++ (1,0);
          \end{scope}
          \end{tikzpicture}
      \caption{The construction of the Dyck path $\varphi(74352681)$.  }
                  % The final step reflects the path across the
                  % diagonal, and then rotates through an angle of $\pi/4$.} 
      \label{prelim:fig:dyckbiject-geom}
      \end{figure}

    \subsection{Permutations Avoiding 123}
    \label{prelim:sec:av123}

      Despite having the same enumeration, the class $\Av(123)$ presents a stark
      contrast to the class $\Av (132)$. First, there are infinitely many
      simple permutations in the class, which prevents us from using many of the
      tools from the previous example. Enumerating and describing these simples
      is the central idea of Chapter~\ref{chap:involutions}. We first present a
      bijective enumeration of the class, before analyzing the structure. 
      
      As a further example highlighting the benefit of the geometric viewpoint
      note that, remarkably, the bijection $\varphi$ described in
      Figure~\ref{prelim:fig:dyckbiject-geom} leads to a bijection
      $\varphi':\Av_n(123) \ra \cD_n$, using \emph{exactly the same description}.
      See Figure~\ref{prelim:fig:123biject} for an example. Note that
      $\varphi^{-1}
      \circ \varphi'$ is a bijection from $\Av_n (123)$ to $\Av_n (132)$, which is
      equivalent to the one presented by Simion and Schmidt~\cite{Simion1985},
      and shows that the locations of left-to-right minima has the same
      distribution in both classes. 

      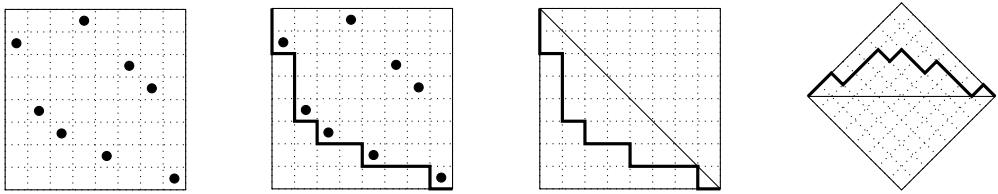
\begin{figure}[t] \centering
        \begin{tikzpicture}[scale=.3]
          \draw (1,1) -- (1,9) -- (9,9) -- (9,1) -- cycle;
          \foreach \y [count = \x] in {7,4,3,8,2,6,5,1}{
            \draw[fill = black] (\x+.5,\y+.5) circle (2mm);
          }
          \foreach \i in {1, ...,9}{
            \draw[dotted] (\i,1) -- (\i,9);
            \draw[dotted] (1,\i) -- (9,\i);
          }
        \end{tikzpicture}
        \hspace{2pc}
        \begin{tikzpicture}[scale=.3]
          \draw (1,1) -- (1,9) -- (9,9) -- (9,1) -- cycle;
          \foreach \y [count = \x] in {7,4,3,8,2,6,5,1}{
            \draw[fill = black] (\x+.5,\y+.5) circle (2mm);
          }
          \foreach \i in {1, ...,9}{
            \draw[dotted] (\i,1) -- (\i,9);
            \draw[dotted] (1,\i) -- (9,\i);
          }
          \draw[very thick] (1,9) --++ (0,-2) --++ (1,0) --++ (0,-3)
                            --++ (1,0) --++(0,-1) --++(2,0) 
                            --++ (0,-1) --++ (3,0) --++(0,-1) 
                            --++ (1,0);
        \end{tikzpicture}
        \hspace{2pc}
        \begin{tikzpicture}[scale=.3]
          \draw (1,1) -- (1,9) -- (9,9) -- (9,1) -- cycle;
          \foreach \i in {1, ...,9}{
            \draw[dotted] (\i,1) -- (\i,9);
            \draw[dotted] (1,\i) -- (9,\i);
          }
          \draw (1,9) -- (9,1);
          \draw[very thick] (1,9) --++ (0,-2) --++ (1,0) --++ (0,-3)
                            --++ (1,0) --++(0,-1) --++(2,0) 
                            --++ (0,-1) --++ (3,0) --++(0,-1) 
                            --++ (1,0);
        \end{tikzpicture}
        \hspace{2pc}
        \begin{tikzpicture}[scale=.22]
        \begin{scope}[shift={(3,5)},rotate=-45, yscale=-1]
          \draw (1,1) -- (1,9) -- (9,9) -- (9,1) -- cycle;
          \foreach \i in {1, ...,9}{
            \draw[dotted] (\i,1) -- (\i,9);
            \draw[dotted] (1,\i) -- (9,\i);
          }
          \draw (1,9) -- (9,1);
          \draw[very thick] (1,9) --++ (0,-2) --++ (1,0) --++ (0,-3)
                            --++ (1,0) --++(0,-1) --++(2,0) 
                            --++ (0,-1) --++ (3,0) --++(0,-1) 
                            --++ (1,0);
          \end{scope}
          \end{tikzpicture}
      \caption{The construction of the Dyck path $\varphi'(74382651)$.
                  }
      \label{prelim:fig:123biject}
      \end{figure}

      A modification of this bijection is central to Chapter~\ref{chap:expat},
      and will be used to count pattern \emph{occurrences} within the class.
      Dyck paths can be used to encode structural information about the
      permutations they represent, and can be easily enumerated.

      \begin{figure}[t] \centering
        \begin{tikzpicture}[scale=.3]
          \foreach \y [count = \x] in {9,7,10,5,8,3,6,1,4}{
            \draw[fill = black] (\x+.5,\y+.5) circle (2mm);
          }

          \foreach \i in {1,2,3}
            \draw[fill = black] (10 + 0.4*\i, 2.2 - 0.4*\i) circle (.5mm);

        \end{tikzpicture}
      \caption{The decreasing oscillations.}
      \label{prelim:fig:oscillation}
      \end{figure}
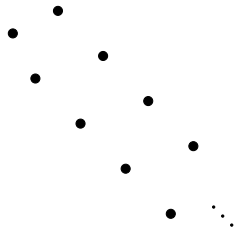

      To see that $\Av (123)$ contains infinitely many simple permutations, we
      define the \emph{decreasing oscillations}, a family of simples which are
      contained within the class. Figure~\ref{prelim:fig:oscillation} gives a
      graphical description of these permutations. 
      Though the simples are not as easily described as in our previous example,
      $\Av(123)$ exhibits a different kind of geometric structure which will be
      equally useful. Since a $123$-avoiding permutation does not contain any
      three increasing entries, it follows that it can be written as the union of
      two decreasing sequences of entries. It follows further that we can
      partition the plot of such a permutation into an alternating sequence of
      monotone decreasing runs. We formalize this in
      Chapter~\ref{chap:involutions}, but for now present an diagram of the
      so-called \emph{staircase decomposition}~\cite{Albert2010, me-involutions}
      in Figure~\ref{prelim:fig:staircase}.

      \begin{figure}[t] \centering
        \begin{tikzpicture}[scale=.25]
          % draw the outer boxes, using a loop
          \foreach \x/\y in {0/0, 5/0, 5/-5, 10/-5, 10/-10, 15/-10}{
            \draw[very thick, color=lightgray] (\x, \y) rectangle (\x + 5, \y + 5);
            \draw (\x+.5, \y+4.55) -- (\x + 4.5, \y + .5 );
          }
          \foreach \i in {1,2,3}
            \draw[fill = black] (20 + 0.8*\i, -10 - 0.8*\i) circle (1mm);

        \end{tikzpicture}
      \caption{The class $\Av(123)$ is precisely those permutations which can be
               plotted on descending lines of the diagram.}
      \label{prelim:fig:staircase}
      \end{figure}
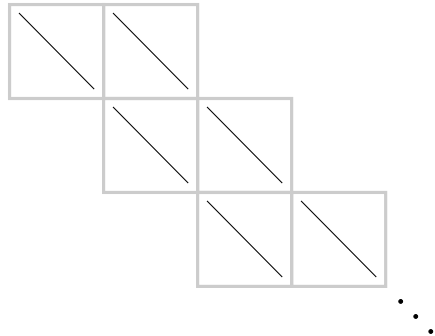

      This decomposition will be used to enumerate and describe the simple
      permutations within the class, which will then be used to enumerate pattern
      avoiding involutions in Chapter~\ref{chap:involutions}. 
      We present one final method of enumerating the class $\Av(123)$, by
      inflating the (infinitely many) simples. In Chapter~\ref{chap:involutions}
      we use the staircase decomposition to enumerate the simples of the class,
      and find that their generating function
      (equation~\ref{eqn:genfcn-simple123perms}) is given by 
      $$ f = \sum_{\substack{ \sg \in \Av(123) \\ \sg \text{ simple}}}
          z^{|\sg|} = \frac{1 - z - \sqrt{1 - 2z - 3z^2}}{2z} .$$

      Each entry of a simple permutation in the class can be inflated only by
      decreasing runs, whose generating functions are given by $\frac{z}{1-z}$.
      It follows then that, since each $z$ in the above generating function
      represents an entry of a simple permutation, replacing $z$ by
      $\frac{z}{1-z}$, we obtain the generating function for all permutations of
      the class. Indeed, after simplifying, we find that this composition gives
      the generating function for the Catalan numbers, with the constant term
      (representing the empty permutation) removed: 
      $$ f\left(\frac{z}{1-z}\right) = \frac{1 - 2z - \sqrt{1-4z}}{2z}.$$

    \subsection{Permutations Avoiding 123 and 231}
    \label{prelim:sec:av123+231}

      Our next example enumerates the class $\Av(123, 231)$ of
      permutations which avoid both $123$ and $231$, using a structural
      description of the class. This example motivates the exploration of the
      \emph{polynomial classes} (the classes whose enumeration is given by a
      polynomial). This will be investigated more fully in
      Chapter~\ref{chap:polyclass}, where an algorithm will be presented which,
      given a structural description, enumerates the class. 

      Since we have already shown that the only simples in the class $\Av (231)$
      are $\{1, 12, 21\}$ (because it is a symmetry of $\Av(132)$), the fact that
      $\Av(123, 231) \subset \Av(231)$ implies that these are the same simples in
      $\Av(123, 231)$. The added restriction of avoiding $123$ changes the way
      these simples can be inflated. Both entries of $12$ can only be inflated by
      decreasing runs, to avoid constructing an occurrence of $123$. Finally, the
      first entry of a $21$ can be inflated only by a decreasing run (to avoid
      $231$), while the second can be inflated by any element from the class. 
      
      After accounting for uniqueness, it follows that every permutation in the
      class can be obtained by inflating the permutation $312$ with (possibly
      empty) descending permutations. Therefore, this class is precisely those
      permutations which can be drawn on the diagram shown in
      Figure~\ref{prelim:fig:polygrid}

      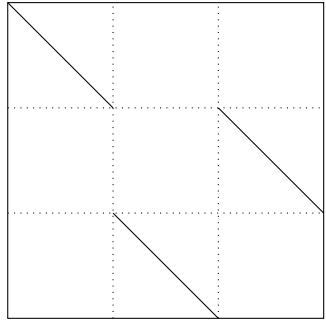
\begin{figure}[t] \centering
        \begin{tikzpicture}[scale=.35]
          \draw (0,0) -- (12,0) -- (12,12) -- (0,12) -- cycle;
          \draw[dotted] (4,0) -- (4,12);
          \draw[dotted] (8,0) -- (8,12);
          \draw[dotted] (0,4) -- (12,4);
          \draw[dotted] (0,8) -- (12,8);
          \draw (0,12) -- (4,8);
          \draw (4,4) -- (8,0);
          \draw (8,8) -- (12,4);
        \end{tikzpicture}
      \caption{The class $\Av(123, 231)$ is precisely those permutations which can be
               plotted on descending lines of the diagram.}
      \label{prelim:fig:polygrid}
      \end{figure}

      This is a simple example of a \emph{grid class}
      %TODO cite murphyvatter
      ~\cite{MurphyVatter}, a useful concept which has produced many new
      enumerations in recent years. It is known~\cite{SophieVince,GridClasses},
      and is presented formally in Theorem~\ref{polyclass:thm:tfae}, that a
      permutation class is enumerated by a polynomial if and only if it is a
      union or intersection of classes which can be represented with such a
      diagram, with only one nonempty cell per row and column. 

      Returning to Figure~\ref{prelim:fig:polygrid}, it is trivial to enumerate
      those permutations which have at least one element in each block: the
      generating function for a single block is $z/(1-z)$, and so the
      generating function for those with no empty blocks are
      $z^3/(1-z)^3$. If the first block is empty, then we have the
      generating function $z^2/(1-z)^2$. If either the second or third
      block is empty, the entire permutation is a single decreasing run, with
      generating function $z/(1-z)$. Therefore, the generating function for
      the entire class is simply the sum of these three:

      \begin{equation} \label{prelim:eqn:polyclass}
        \sum_{n \geq 1} |\Av_n (123,231)| z^n = 
          \frac{z^3}{(1-z)^3} + \frac{z^2}{(1-z)^2} + \frac{z}{1-z} =
          \frac{z^3- z^2 + z}{(1-z)^3}.
      \end{equation}

      Equation~\ref{prelim:eqn:polyclass} expanded using the binomial theorem to
      produce an exact equation for the number of permutations of each length in
      the class. 

      $$ | \Av_n (123, 231) | = \frac{n^2 - n + 2}{2} = \binom{n}{2} + 1.$$

      More complicated decompositions lead to a number of technical obstacles, but
      this same general idea can be used to calculate the polynomials enumerating all
      such classes. This will be presented in Chapter~\ref{chap:polyclass}, and
      an implementation of the algorithm is available
      online~\cite{polyclass-algo}. 

    \subsection{Ascents in 132-Avoiding Permutations}
    \label{prelim:sec:ascents-example}
    
      We end this chapter with an illustrative example which utilizes a class's
      structural decomposition to investigate the distribution of a permutation
      statistic. This example, while relatively simple, serves to showcase the
      techniques which will be used throughout the following chapters, and
      is particularly pertinent to Chapter~\ref{chap:expat}.

      A \emph{permutation statistic} is any function $\chi: \S_n \ra \bR$. In
      practice, we often consider statistics that map from permutations to
      non-negative integers which capture some structural trait of the
      permutation. Examples include the location of the largest element, number
      of cycles, value of the first entry, and number of inversions. In this
      section we consider the number of ascents of a permutation. An
      \emph{ascent} of a permutation $\pi = \pi_1 \pi_2 \dots \pi_n$ is an index
      $i$ such that $\pi_i < \pi_{i+1}$, and the number of ascents in a
      permutation $\pi$ is denoted $\asc(\pi)$.

      For a given permutation $\pi$ of length $n$, it follows that $\asc(\pi) \in \{0, 1,
      \dots n-1\}$. If $i$ is an ascent of $\pi$ then $i$ is a descent of
      $\pi^c$, and so the number of permutations of lengh $n$ with $k$ ascents is equal to
      the number of such permutations with $k$ descents (or $n - k - 1$ ascents).
      This implies in particular that the \emph{average} number of ascents in a
      randomly selected permutation from $\S_n$ is $(n-1)/2$. When we restrict to a
      proper permutation class, however, the distribution can be more difficult
      to compute. 

      For a finite set $S$ of permutations and a statistic $f$, the
      \emph{generating polynomial} for $f$ on $S$ in indeterminate $u$ is 
      $$ \sum_{\pi \in S} u^{f(\pi)}.$$
      For example, if $\S_3 = \{123, 132, 213, 231, 312, 321\}$ then the generating
      polynomial for the number of ascents is $u^2 + 4u + 1$, since there is one
      permutation with two ascents, four with one ascent, and one permutation
      with no ascents. There is one crucial observation: if we take the
      derivative (with respect to $u$) of the generating polynomial and set $u=1$
      we obtain a weighted sum which evaluates to the expected value, or average,
      of the statistic on $S$. Further, by differentiating twice before setting
      $u=1$, and then dividing by two, we obtain the first factorial moment of
      the statistic, which can be used to compute the variance. This process can
      be iterated to calculate higher moments of the distribution. 

      Extending to permutation classes, let $|\pi|$ denote the length of a
      permutation $\pi$ and define the generating function for a statistic $f$
      across a class $\C$ as
      $$ \sum_{\pi \in \C} z^{|\pi|} u^{f(\pi)}.$$
      The coefficient of $z^n$ in this bivariate generating function is precisely
      the generating polynomial for the statistic $f$ on the set $\C_n$, and so
      it follows that by differentiating with respect to $u$ and plugging in
      $u=1$, we can obtain generating functions whose coefficients represent the
      moments of the distribution on $\C_n$. Asymptotic analysis can then be used
      to compute the limiting distribution as $n$ approaches infinity.

      Throughout this section, let $a_{n,k}$ be the number of $132$-avoiding
      permutations of length $n$ which contain exactly $k$ ascents, and let
      $$f(z,u) = \sum_{\pi \in \Av(132)} z^{|\pi|} u^{\asc(\pi)} = 
                \sum_{n \geq 0} \sum_{k \geq 0} a_{n,k} u^k z^n.$$

      Our goal is to derive a closed expression for $f$, and use this to analyze
      the distribution of descents across $\Av(132)$. Consider the recursive
      description of the class, shown in Figure~\ref{prelim:fig:av132-first}, and
      let $\pi = (\ro \dsum 1) \ssum \sg$ be a $132$-avoiding permutation. It
      follows that the number of ascents of $\pi$ is equal to the sum of ascents
      in $\ro$ and $\sg$, plus one \emph{if} $\ro$ is nonempty (otherwise the
      permutation starts with its biggest entry).
      This relationship leads to the following functional equation. 

      $$ f = zf + uz(f - 1) f + 1.$$

      The first term on the right hand side is the case where $\ro$ is empty, the
      second is when $\ro$ is non-empty, and the constant term accounts for the
      empty permutation. We can solve for $f$ above to find the following:
      $$ \begin{aligned}
      f(z,u) &= \frac{1 + (u-1)z - \sqrt{(u^2 - 2u + 1)z^2 - 2(u+1)z + 1}}{2uz} \\
        &= 1 + z + (u+1)z^2 + (u^2 + 3u + 1)z^3 + (u^3 + 6u^2 + 6u + 1)z^4 +
        \dots. \\
        \end{aligned} $$

      Note that substituting $u=1$ gives the generating function for the Catalan
      numbers, as expected. The coefficient of $z^3$ is $(u^2 + 3u + 1)$, as
      there is one $132$-avoiding permutation with two ascents ($123$), three
      with one ascent ($213, 231, 312$), and one with no ascents ($321$).
      Finally, we can obtain the \emph{total} number $a_n$ of ascents in all
      $132$-avoiding permutations of length $n$ by differentiating with respect to $n$ and
      setting $u=1$:

      $$ \begin{aligned} 
         \sum_{n \geq 0} a_n z^n  &= \partial_u f(z,u) \uisone \\
         &= \frac{1 - 3z - (z - 1)\sqrt{1 - 4z}}{1 + z\sqrt{1 - 4z}} \\
         &= \sum_{n \geq 0} \binom{2n-1}{n-2} z^n \\
         &= z^2 + 5z^3 + 21z^4 + 84z^5 + 330 z^6 + 1287z^6 \dots .
      \end{aligned} $$
      
      It follows then that the \emph{average} number of ascents in a randomly
      selected $132$-avoiding permutation is given by this total divided by
      the total number of such permutations, the Catalan numbers. Therefore the
      average is given by 
      $$ \binom{2n-1}{n-2} \frac{n+1}{\binom{2n}{n}} = \frac{n-1}{2}.$$

      Note that this expectation is identical to the average number of ascents in
      a random permutation chosen from the set $\S_n$, and so it follows that
      the property `avoids $132$' is independent from the random variable $\asc$.
      This can also proven bijectively, by constructing a map from
      $\Av_n(132)$ to itself which maps ascents to descents (by mapping the
      permutations to unlabelled binary trees, and then reflecting the tree), but
      the above approach can be extended and generalized to other statistics and
      classes, as we will soon see. 
      
      In Chapter~\ref{chap:expat} we explore how pattern-avoidance changes the
      distribution of other statistics. These same techniques will be revisited
      in Chapter~\ref{chap:fixpat} and used to compute the distribution of
      intervals of size two, which relates to the number of distinct patterns
      within a permutation.

\cleardoublepage
\typeout{******************}
\typeout{**  Chapter 2   **}
\typeout{******************}

\chapter{Pattern Expectation}
  \label{chap:expat}
    
    %\bibentry{me-expat}

    In the set of all permutations of length $n$, all patterns of a fixed length 
    \emph{occur} the same number of times. However, if we restrict to smaller
    classes of permutation, the situation quickly becomes more interesting.  The
    investigation of pattern occurrences within permutations is a recent and
    productive research topic. This chapter explores this new area, and uses it
    to develop connections between permutation classes.

    In particular, we examine the classes of $123$- and $132$-avoiding
    permutations, and show that the number of $231$ patterns is identical in each.
    This identity extends an earlier result of Mikl\'os B\'ona~\cite{Bona2012},
    and its derivation sheds further light on the distribution of pattern
    occurrences within permutation classes. Further, this chapter brings to light
    new equivalences between these classes, building on those presented by
    Elizalde~\cite{sergithesis}, and forming a foundation for
    further study~\cite{Elizalde2013, Rudolph2013, Janson2014}.  This chapter is
    based partly on~\cite{me-expat}.

  % =========================================================================== %
  \section{Pattern Occurrences}
  \label{expat:occurrences}
    
    Our primary concern in this chapter (and much of Chapter~\ref{chap:fixpat})
    will be the number of \emph{occurrences} of a pattern within a permutation.
    The number of occurrences is the number of copies of the pattern we can
    find within a permutations; formally, we define this as follows: 
    
    \begin{definition} \label{def:occurrence} \index{pattern occurrences}
      Let $\sg = \sg_1 \sg_2 \dots \sg_k$ be a pattern of length $k$, and $\pi =
      \pi_1 \pi_2 \dots \pi_n$ a permutation of length $n$. An \emph{occurrence} of the
      pattern $\sg$ in $\pi$ is a subsequence $i_1 < i_2 < \dots < i_k$ such
      that 
      $$ \pi_{i_1} \pi_{i_2} \dots \pi_{i_k} \sim \sg_1 \sg_2 \dots \sg_k. $$ 
      The number of occurrences of $\sg$ in $\pi$, denoted by $\num_\sg(\pi)$, is
      the number of such subsequences.  
    \end{definition}
      
    For example, the permutation $\pi = 462513$ contains $2$ occurrences of the
    pattern $213$, since the first, third, and fourth, as well as the third,
    fifth, and sixth, entries of $p$ form $213$ patterns. Thus,
    $\num_{213}(462513) = 2$.

    Clearly, for permutations $\pi$ of length $n$ and $\sg$ of length $k$, we have
    that $\num_q(\pi)$ is bounded below by $0$ and above by $\binom{n}{k}$. This
    minimum value is realized by taking $\pi$ to be any $\sg$-avoiding
    permutation, and the maximum is attained, for example, when both $\pi$ and
    $\sg$ are ascending permutations.  Our primary concern will be the average
    number of occurrences of a pattern over a set of permutations. In the interest
    of brevity, we will abuse the above notation to apply to sets:
    
    \begin{definition} \label{def:set-occurrence} For a given pattern $\sg$ and a
    set $S$ of permutations, let $\num_q(S)$ denote the total number of
    occurrences of $\sg$ within the set $S$. That is, $$ \num_\sg(S) = \sum_{\pi
    \in S} \num_\sg(\pi).$$ \end{definition}
    
    For example, letting $S = \{2341, 4321, 1234\}$, we have that 
    $$ \num_{123}(S) = 1 + 0 + 4 = 5.$$

    \subsection{Pattern Expectation}
        
      Counting the total number of occurrences of a pattern within a set of
      permutations has an alternate, probabilistic interpretation. The
      \emph{expectation} \index{expectation} of a pattern within a set is defined
      to be the average number of occurrences of the pattern within a randomly
      selected element from the set. Clearly, we have that the expectation of a
      pattern $\sg$ in a set $S$ is equal to $\num_\sg(S) / |S|$. 
      
      This probabilistic interpretation motivates many questions, several of which
      have yielded interesting and surprising answers. We start with an
      illustrative example, whose derivation showcases some of the ideas which
      will be useful later. In particular, \emph{linearity of expectation}
      \index{linearity of expectation} will prove useful. 
      
      \begin{proposition} \label{expat:prop:allperms}
        Let $\sg$ be any pattern of length $k$, and let $n \geq k$.  Then 
        $$ \num_\sg(\S_n) = \frac{n!}{k!} \binom{n}{k}.$$
      \end{proposition}
      \begin{proof}
        We show that the expectation of the pattern $\sg$ is equal to 
        $\binom{n}{k}/k!$, which will imply the desired result. 
        Let $\pi$ be a (uniformly) randomly selected permutation in $\S_n$, and let
        $X$ be the random variable denoting the number of occurrences of $\sg$
        within $\pi$. 

        There are $\binom{n}{k}$ sets of positions of $\pi$ in which a $\sg$
        pattern could possibly occur. For each set $P$, let 
        $$ X_P = \left\{ \begin{array}{cc} 
                  1 & \text{ the entries of $P$ form a $\sg$ pattern} \\
                  0 & \text{ otherwise} 
                \end{array}\right..
        $$

        It now follows that $X = \sum_{P} X_P$, and so by linearity of expectation,
        we have that 
        $$ \Ex{X} = \sum_P \Ex{X_P}.$$

        Finally, for any specified set of indices, all patterns are equally likely.
        Therefore, $\Ex{X_P} = 1 / k!$. Combining, we see that 
        $$ \Ex{X} = \sum_P \frac{1}{k!} = \binom{n}{k}\frac{1}{k!}. $$

        Therefore, we have that 
        $$\num_\sg(\S_n) = \frac{|\S_n|}{k!} \binom{n}{k} =
        \frac{n!}{k!} \binom{n}{k}.$$
      \end{proof}

      Fact \ref{expat:prop:allperms} shows that the total number of pattern occurrences
      within the set of all permutations depends \emph{only} on the length of the
      pattern specified. This contrasts sharply with the fact that the numbers of
      permutations which avoid a given pattern varies widely based on the choice of
      pattern. This discrepancy can be explained in part by the fact that certain
      patterns are better able to overlap with themselves, so that a smaller number
      of permutations contains a higher concentration of pattern occurrences. 
      
      The problem of pattern packing will be discussed in more detail in Chapter
      \ref{chap:fixpat}. In this chapter we examine the pattern expectation of
      of small patterns within avoidance classes. In particular we seek insight to
      the following question, first posed by Joshua Cooper: ``How does the absence
      of one pattern affect the expectation of another?''
      % TODO: cite josh cooper 

    \subsection{Background and Data}
      
      The total number of length $3$ patterns in the sets $\Av_n(123)$ and
      $\Av_n(132)$ are shown below, for $1 \leq n \leq 7$.  

      \begin{table}[t] \label{expat:tab:data}
      \caption[Total number of pattern occurrences]{Total number of pattern
          occurrences for length $3$ patterns in $123$- and $132$-avoiding
          permutations.}
      $$
      \begin{array}{ccccccc}
          \multicolumn{7}{c}{\Av_n(123) } \\
          \text{length} & \num_{123} & \num_{132} & \num_{213}
          & \num_{231} & \num_{312} & \num_{321} \\
          \hline
          3  & 0     &    1  &    1 &    1 &    1 &    1  \\
          4  & 0     &    9  &    9 &   11 &   11 &   16  \\
          5  & 0     &    57 &   57 &   81 &   81 &  144  \\
          6  & 0     &   312 &  312 &  500 &  500 & 1016  \\
          7  & 0     &  1578 & 1578 & 2794 & 2794 & 6271
        \end{array}
      $$

      \vspace{1pc}

      $$
      \begin{array}{ccccccc}
          \multicolumn{7}{c}{\Av_n(132) } \\
          \text{length} & \num_{123} & \num_{132} & \num_{213}
          & \num_{231} & \num_{312} & \num_{321} \\
          \hline
         3  & 1     &    0  &    1 &    1 &    1 &    1  \\
         4  & 10    &    0  &   11 &   11 &   11 &   13  \\
         5  & 68    &    0  &   81 &   81 &   81 &  109  \\
         6  & 392   &    0  &  500 &  500 &  500 &  748  \\
         7  & 2063  &    0  & 2794 & 2794 & 2794 & 4570
       \end{array}
      $$
      \end{table}

      Since both $123$ and $132$ are involutions, \index{involution} inversion
      maps each set to itself, and maps patterns to their inverse. This implies
      the identity $\num_{231} = \num_{312}$ in both sets of permutations.  
      Mikl\'os B\'ona~\cite{Bona2010, Bona2012} investigated the set
      $\Av_n(132)$ and enumerated the total occurrences of each length 3 pattern.
      In particular, he established the identity $\num_{213}(\Av_n(132))
      = \num_{231}(\Av_n(132))$. 

      This implies that the statistics \index{statistic} $\num_{213}$ and
      $\num_{231}$ have the same expectations \index{expectation} over the set of
      $132$-avoiding permutations of length $n$. This identity is surprising in part because
      these two statistics have \emph{different distributions} over this set, but
      share the same average value. 

      The main motivation for Section~\ref{expat:av123} is establishing the
      identity 
      $$\num_{231}(\Av(132)) = \num_{231}(\Av(123)).$$
      This identity extends B\'ona's result, and presents another example of two
      permutation statistics with different distributions having the same mean.

  % =========================================================================== %
  \section{123-avoiding Permutations}
  \label{expat:av123}

      In this section, we derive exact and asymptotic values for $\num_\sg
      (\Av_n(123))$ for $|\sg| \leq 3$ and $n \geq 0$. In addition, we show that
      for $k \geq 1$, the pattern $k\ (k-1)\ (k-2)\ \dots 2\ 1$ has a higher
      expectation than any other pattern of length $k$ for large enough permutations. 
      Finally, applying recent results of Mikl\'os B\'ona, we show that 
      the total number of $231$ patterns is identical within the sets of
      $132$-avoiding and $123$-avoiding permutations of length $n$.

      Throughout this sections, let $n$ be some fixed positive integer. For
      simplicity of notation, we use $\num_\sg$ to denote $\num_\sg (\Av_n(123))$.

    \subsection{Class Structure}

      The class of $123$-avoiding permutations has a rigid structure, which we will
      use to investigate pattern occurrences. 
      Recall (Section~\ref{prelim:sub:catalan}) that $|\Av_n(123)| = c_n$, where $c_n$
      is the $n$th Catalan number.  \index{Catalan number}.  For a permutation
      $\pi = \pi_1 \pi_2 \dots \pi_n$, we say that the entry $\pi_i$ is a
      \emph{left-to-right minimum} (ltr-min) \index{left-to-right minimum} if it
      is smaller than all of the elements to its left, and a \emph{right-to-left
      maximum} (rtl-max) \index{left-to-right maximum} if it is larger than all of
      the elements to its right. 
      
      In a $123$-avoiding permutation $\pi$, \emph{every} element is either a
      ltr-min or a rtl-max (or possibly both), since otherwise it would have a bigger
      element to its right and a smaller element to its left, which would form a
      $123$ pattern. By definition, the sets of ltr-min and of rtl-max are both
      decreasing when read from left to right. Therefore, every $123$-avoiding
      permutation is the union of two decreasing sequences of entries. 

      Breaking down permutations into these two decreasing sequences will prove
      useful in the following sections. However, the possibility of an element
      being both a ltr-min and a rtl-max poses problems. Further restricting our
      permutations will alleviate this issue. 

      \begin{definition} \label{def:indecomposable}
        A permutation $\pi = \pi_1 \pi_2 \dots p_n$ is \emph{skew-decomposable}  if
        there exist permutations $\sg$ and $\ph$ for which $\pi = \sg \ssum \ph$.
        Otherwise, we say that $\pi$ is \emph{skew-indecomposable}.  Denote the set of
        indecomposable $123$-avoiding permutation by $\Avns$. 
        Sum (in)decomposability is defined similarly. 
      \end{definition}
      
      In this chapter we consider only
      skew-(in)decomposability, and so we drop the word `skew' for the simplicity
      of notation.

      Note that if any element of $\pi$ is both a ltr-min and a rtl-max, then $\pi$
      is decomposable. It follows then that every indecomposable 123-avoiding
      permutation can be uniquely decomposed into its left-to-right minima and its
      left-to-right maxima. Further, it follows that \emph{every} 123-avoiding
      permutation can be written as a skew sum of indecomposable 123-avoiding
      permutations. We use this fact to enumerate these permutations.

      \begin{proposition} \label{expat:prop:indecomposable}

        The number of indecomposable $123$-avoiding permutations is $c_{n-1}$, the
        $(n-1)$st Catalan number \index{Catalan number}.
      \end{proposition}
      \begin{proof}
        Let 
        $$ C^*(x) = \sum_{n \geq 0} |\Avns| x^n.$$
        We know that $|\Avn| = c_n$, and so 
        $$ \sum_{n \geq 0} |\Avn| x^n = \frac{1 - \sqrt{1-4x}}{2x} = C(x).$$
        Since every permutation $\pi \in \Avn$ can be written as $sg_1 \ssum \sg_2
        \dots \sg_k$ for some $\sg_1, \sg_2, \dots \sg_k \in \Avns$ and some $k
        \geq 1$, it follows that 
        $$ C(x) = 1 + C^*(x) + (C^*)(x))^2 + (C^*(x))^3 + \dots 
          = \frac{1}{1 - C^*(x)}.$$
        Rearranging this equation leads to 
        $$ C^*(x) = \frac{C(x) - 1}{C(x)} = xC(x).$$
        The second equality follows from the identity $C(x) = xC(x)^2 + 1$. 

        Therefore, 
        $$ \sum_{n \geq 0} |\Avns|x^n = C^*(x) = xC(x) = \sum_{n \geq 1}
        c_{n-1}x^n.$$
      \end{proof}

    \subsection{Patterns of Length 2}  

      To start, we compute the values $\num_{12}$ and $\num_{21}$. Since every pair
      of entries must form either a $12$ or a $21$ pattern, the sum $\num_{12} +
      \num_{21}$ is equal to the total number of pairs of entries amongst the set
      of all $123$-avoiding permutations. Therefore, we have 
      $$ \num_{12} + \num_{21} = \binom{n}{2}.$$

      An \emph{inversion} \index{inversion} of a permutation is an occurrence
      of the pattern $21$. Inversions are a well-known and well-studied
      permutation statistic, and the total number of inversions amongst the set
      $\Av_n(321)$ is known. 

      \begin{theorem}[Cheng, Eu, Fu~\cite{Cheng2007}] 
        The total number of inversions in the set $\Av_n(321)$ is given by 
        $$ \num_{21}(\Av_n(321)) = 4^{n-1} - \binom{2n - 1}{n}.$$
        The generating function for this sequence is as follows:
        $$ \sum_{n \geq 0} \num_{21}(\Av_n(321)) x^n = 
          \frac{x^2 C(x)^2}{1 - 4x}.$$
      \end{theorem}

      By reversing permutations, we see that $\num_{21}(\Av_n(321)) =
      \num_{12}(\Av_n(123))$. This allows us to establish exact answers for the
      number of occurrences of length $2$ patterns within $\Avn$. 

      \begin{proposition} \label{prop:2-patterns}
        The total number of $12$ patterns in $\Avn$ is given by 
        $$ \num_{12} = 4^{n-1} - \binom{2n - 1}{n}.$$
        Further, since 
        $ \num_{21}(\Av_n(321)) = 4^{n-1} - \binom{2n - 1}{n},$
        it follows that 
        $$ \num_{21} = \binom{n}{2} c_n - 4^{n-1} + \binom{2n - 1}{n}.$$
      \end{proposition}

    \subsection{Patterns of Length 3}

      Deriving the number of occurrences for length three patterns is considerably
      more involved, but utilizes some of the same ideas. In this section we find
      both the asymptotic and exact values for the total occurrences of $\num_\sg$
      for each $\sg \in \S_3$.  The key idea will be derive the total number of
      occurrences of a single pattern, and then use the class structures to develop
      the other values. 
      Let 
      $$a_n = \num_{213}, \quad b_n = \num_{231}, \quad \num_{321}.$$

      We start by finding the generating function for the numbers
      $\num_{213}(\Avns)$. While this may seem arbitrary, this will in fact lead
      to generating functions for all other patterns.  Let $\pi$ be a permutation
      in $\Avns$. Recall that each entry in $\pi$ is either a ltr-min or a
      rtl-max, and no entry is both. An occurrence of $213$ within $\pi$ must
      consist of two left-to-right minima followed by a right-to-left maximum. By
      counting the number of entries to the left and below each rtl-max we can
      exactly determine the number of $213$ patterns within $\pi$. 
      
      \begin{lemma} \label{lemma:213pats}
        The generating function $A^*(x)$ for the number of $213$ patterns in
        $\Avns$ is given by
        $$ A^*(x) = \sum_{n\geq 0} \num_{213}(\Avns) =
        \frac{x^3C(x)}{(1-4x)^{3/2}} = \frac{x^2}{2(1-4x)^{3/2}} -
        \frac{x^2}{2(1-4x)}.$$
      \end{lemma}
      \begin{proof}
        The proof consists of three parts: First, we examine the structure of
        permutations in $\Avns$, and find a simple way of counting the number of
        $213$ patterns. Second, we build a bijection onto Dyck paths which maps
        $213$ patterns to a path statistic. Finally, we find the weighted sum of
        all Dyck paths with respect to this statistic.

        The idea of the proof is as follows: We build a bijection from the set of
        permutations $\Avns$ to the set $\D$ of elevated Dyck paths of semilength
        $n$, find a statistic on these paths which corresponds to $213$ patterns,
        and then find the weighted sum of all Dyck paths with respect to this
        statistic.

        Let $\pi$ be a permutation in $\Avns$, and consider the plot of $\pi$.
        Note that, by the indecomposability of $\pi$, there is no entry which is
        simultaneously a ltr-min and a rtl-max. Construct a Dyck path $\phi(\pi)$
        of semilength $n-1$ as follows. First, build a path from $(1,n)$ to
        $(n,1)$ using the steps $\{\vect{1,0}, \vect{0,1} \}$. Let this path be
        the be the unique path which minimizes the area underneath itself while
        lying above all of the entries of $\pi$. This path, a variation of the
        construction presented in Section~\ref{prelim:sec:av123}, is then
        uniquely defined by the locations of the right-to-left maxima, which in
        turn uniquely define the permutation. Finally, rotate each $\vect{1,0}$
        step to be an up step, and each $\vect{0,-1}$ to become a downstep in
        the path $\phi(\pi)$. See Figure~\ref{expat:fig:dyckproof} for an example
        construction. 
        
        This path is a slight modification of the path given by Krattenthaler's
        bijection~\cite{Krattenthaler2001}, taking advantage of the
        indecomposability of the permutation to yield a more geometric description.
        This geometric interpretation of the bijection gives some additional
        insight into the number of $213$ patterns.

      \begin{figure}[t] \centering
        \begin{tikzpicture}[scale=.3, yscale=-1, xscale=-1]
          \draw (1,1) -- (1,9) -- (9,9) -- (9,1) -- cycle;
          \foreach \y [count = \x] in {7,4,3,8,2,6,1,5}{
            \draw[fill = black] (\x+.5,\y+.5) circle (2mm);
          }
          \foreach \i in {1, ...,9}{
            \draw[dotted] (\i,1) -- (\i,9);
            \draw[dotted] (1,\i) -- (9,\i);
          }
        \end{tikzpicture}
        \hspace{2pc}
        \begin{tikzpicture}[scale=.3, yscale=-1, xscale=-1]
          \draw (1,1) -- (1,9) -- (9,9) -- (9,1) -- cycle;
          \foreach \y [count = \x] in {7,4,3,8,2,6,1,5}{
            \draw[fill = black] (\x+.5,\y+.5) circle (2mm);
          }
          \foreach \i in {1, ...,9}{
            \draw[dotted] (\i,1) -- (\i,9);
            \draw[dotted] (1,\i) -- (9,\i);
          }
          \draw[very thick] (1,8) --++ (0,-1) --++ (1,0) --++ (0,-3)
                            --++ (1,0) --++(0,-1) --++(2,0) 
                            --++ (0,-1) --++ (2,0) --++(0,-1) 
                            --++ (1,0);
        \end{tikzpicture}
        \hspace{2pc}
        \begin{tikzpicture}[scale=.3, yscale=-1, xscale=-1]
          \draw (1,1) -- (1,9) -- (9,9) -- (9,1) -- cycle;
          \foreach \i in {1, ...,9}{
            \draw[dotted] (\i,1) -- (\i,9);
            \draw[dotted] (1,\i) -- (9,\i);
          }
          \draw (1,8) -- (8,1);
          \draw[very thick] (1,8) --++ (0,-1) --++ (1,0) --++ (0,-3)
                            --++ (1,0) --++(0,-1) --++(2,0) 
                            --++ (0,-1) --++ (2,0) --++(0,-1) 
                            --++ (1,0);
        \end{tikzpicture}
        \hspace{2pc}
        \begin{tikzpicture}[scale=.22, xscale=-1]
        \begin{scope}[shift={(3,5)},rotate=-45, yscale=-1]
          \draw (1,1) -- (1,9) -- (9,9) -- (9,1) -- cycle;
          \foreach \i in {1, ...,9}{
            \draw[dotted] (\i,1) -- (\i,9);
            \draw[dotted] (1,\i) -- (9,\i);
          }
          \draw (1,8) -- (8,1);
          \draw[very thick] (1,8) --++ (0,-1) --++ (1,0) --++ (0,-3)
                            --++ (1,0) --++(0,-1) --++(2,0) 
                            --++ (0,-1) --++ (2,0) --++(0,-1) 
                            --++ (1,0);
          \end{scope}
          \end{tikzpicture}
      \caption{The construction of the Dyck path $\phi(48371652) =
                  uduuduududddud$. }
      \label{expat:fig:dyckproof}
      \end{figure}
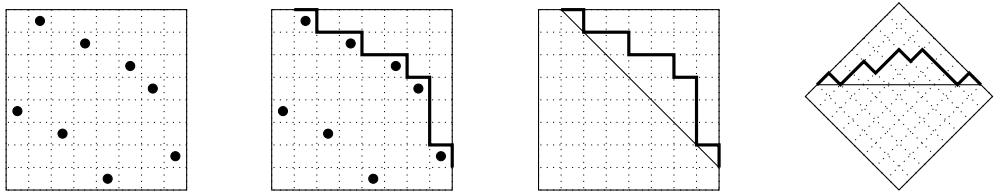

        Note that each rtl-max in $\pi$ produces a peak in $P$. If $\pi_i$ is a
        rtl-max, let the \emph{span of $\pi_i$} ($\Sp \pi_i$) denote the number
        of entries to the left and below this entry. It follows then that 
        $\pi_i$ corresponds to a peak of height $\Sp \pi_i$ above the $x$-axis in $P$.
        An occurrence of $213$ must have a rtl-max as its $3$ entry, and it
        follows then that the $21$ entries must lie in the span of this entry. We
        therefore see that every rtl-max is involved in $\binom{\Sp \pi_i}{2}$
        occurrences of $213$, since we need only choose any two elements in its
        span to act as the $21$. 
        Therefore, if we let $h_{n,k}$ denote the total number of peaks of height
        $k$ in all Dyck paths of semilength $n$, we have that
        $$ \num_{213}(\Avns) =\sum_{k = 1}^{n-1} \binom{k}{2} h_{n-1,k}.$$

        Finally, we can compute $H(x,u) = \sum_{n,k \geq 0} h_{n,k} x^n u^k$ as
        follows. First, note that since each Dyck path begins with an upstep it has
        a unique first point at which the path returns to the $x$-axis, so we can
        decompose each path $P$ of length $n$ into the concatenation of two shorter
        paths $Q$ and $R$. This gives that $P = uQdR$, where $u$ denotes an upstep
        and $d$ a downstep, and each peak of height $k-1$ in $Q$ and height $k$ in
        $R$ leads to a peak of height $k$ in $P$. With this in mind, we have the
        following generating function relation: $$ H(x,u) = ux(H(x,u)+1)C(x) +
        xH(x,u)C(x) .$$ Here the first term counts the peaks from the $uQd$ part,
        including the case when $Q$ is empty. The second term counts the
        contribution from the $R$ part.  Rearranging leads to 
        $$ H(x,u) = \frac{uxC(x)}{1-uxC(x)-xC(x)}.$$

        Now, to count $213$ patterns, we need to count each peak with weight
        $\binom{k}{2}$. By taking derivatives twice with respect to $u$, setting
        $u=1$, dividing by two and scaling by $x$, we find that
        $$ \begin{aligned}
        \sum_{n,k \geq 0} \binom{k}{2} h_{n-1,k}x^n &
        = x \frac{\left. \partial_u ^2 H(x,u)\right|_{u=1}}{2}
        = \frac{x^3C(x)}{(1-4x)^{3/2}} \\
        & = x^3 + 7x^4 + 38x^5 + 187x^6 + 874x^7 + \dots \ .
        \end{aligned}
        $$

        The sequence $0,0,1,7,38,187\dots$ is \OEIS{A000531}.  Finally,
        the correspondence between peaks and $213$ patterns completes the proof.
      \end{proof}

      Now, it is relatively simple to move from the set of indecomposable
      $123$-avoiding permutations to the larger set of all $123$-avoiding
      permutations.

      \begin{theorem} \label{expat:thm:213pats}
        Let $a_n$ be the number of $213$ patterns in $\Av_n 123$. Then
        $$ \sum_{n\geq 0} a_n x^n = \frac{x^3C(x)^3}{(1-4x)^{3/2}} =
            \frac{x-1}{2(1-4x)} - \frac{3x-1}{2(1-4x)^{3/2}}.$$
      \end{theorem}
      \begin{proof}
        Let $A(x)$ be the generating function for the numbers $a_n$, and let
        $A^*(x)$ denote the generating function for the number of $213$ patterns
        in \emph{indecomposable} $123$-avoiding permutations.

        Now, any permutation $\pi$ in $\Av(123)$ can be written uniquely as a
        skew sum of a nonempty indecomposable $123$-avoiding permutation $\sg$
        and another, possibly empty, $123$-avoiding permutation $\ph$. Now, it is
        clear that any $213$ pattern in $\pi$ must be contained entirely in
        either $\sg$ or $\ph$. This leads to the following relation:
        $$ A(x) = A^*(x)C(x) + xC(x)A(x).$$
        Solving for $A$ gives
        $$ A(x) = \frac{A^*(x) C(x)}{1-xC(x)} = C^2(x) A^*(x).$$
        Lemma~\ref{lemma:213pats} now implies
        $$ A(x) = \frac{x^3 C(x)^3}{(1-4x)^{3/2}}.$$
        % Exact formula for the numbers $a_n$ is obtained from the
        % generating function by expanding by partial fractions.
      \end{proof}

      From here, we obtain the generating functions of the other patterns simply
      by relating their enumerations with the one already obtained. The following
      two observations provide linear relations between these numbers. The first
      follows from the simple fact that any three entries must form \emph{some}
      $3$-pattern. 

      \begin{lemma} \label{expat:lem:totalpats}
        On the set $\Avn$, we have that 
        $$ \num_{132} + \num_{213} + \num_{231} + \num_{312} + \num_{321} 
          =  c_n \binom{n}{3}. $$
      \end{lemma}
      \begin{proof}
        Both sides count the total number of $3$-patterns within the class
        $\Avn$.  The right-hand-side is the total number of ways of choosing
        three indices in any $123$-avoiding permutation. Each of these choices 
        is an occurrence of a $3$-patterns other than $123$, which is counted by
        the left-hand-side. 
      \end{proof}

      The next lemma provides a relationship between the numbers $\num_{132},
      \num_{213}, \num_{231}$, and $\num_{312}$ by counting the total number of
      $3$-patterns which contain a \emph{non-inversion} (an occurrence of $12$).

      \begin{lemma} \label{expat:lem:twopats}
        The following equality holds on the set $\Avn$: 
        $$ 2 \num_{132} + 2 \num_{213} + \num_{231} + \num_{312} = 
          (n-2) \num_{12} .$$
      \end{lemma}
      \begin{proof}
        Rewrite this equation as 
        $$ 
        (n-2) \num_{12} - (\num_{132} + \num_{213})
        = \num_{132} + \num_{213} + \num_{231} + \num_{312}.
        $$
        Both sides count the total number of length 3 patterns which contain at
        least one non-inversion. Indeed, the right-hand-side
        counts all 3-patterns except for $321$. The left-hand-side builds such a
        pattern by first choosing a $12$ pattern, and then adding another entry
        to create a 3-pattern. However, this overcounts the patterns $132$ and
        $213$, since each of these contains two $12$-patterns, so we subtract
        these off to correct the equality. 
      \end{proof}

      The generating functions for the numbers $c_n \binom{n}{3}$ and $(n-2)
      \num_{12}$ can be determined from the generating functions we already have.
      These equations can be obtained using techniques explained in Section
      \ref{prelim:sec:ascents-example}. 

      \begin{lemma} \label{expat:lem:genfcns}
        Letting $J(x) = \sum_{n \geq 0} \num_{12}(\Avn) x^n$, the following
        identities hold:
        $$ \begin{aligned}
        \sum_{n \geq 0} c_n \binom{n}{3} 
        &= \frac{x^3 \ddxc (C(x))}{6} \\ 
        \sum_{n \geq 0} (n-2)\num_{12}(\Avn) 
        &= x^3 \ddx \left(\frac{J(x)}{x^2} \right). 
        \end{aligned} $$
      \end{lemma}

      Lemmas~\ref{expat:lem:twopats} and ~\ref{expat:lem:totalpats}, coupled with
      Lemma~\ref{expat:lem:genfcns}, establish a system of linear equations with
      three unknowns, $\num_{213}, \num_{231}$, and $\num_{321}$.
      Any new linear relation or solution to one of these would solve the system,
      giving generating functions and exact formulas for the number of all length
      $3$ patterns within $\Av_n(123)$. 

      The calculation of the $\num_{213}$ provides that missing piece, but we
      note that there are many other identities which, once these lemmas are
      established, are equivalent to Theorem~\ref{expat:thm:213pats}. We collect
      some of these in Corollary~\ref{expat:cor:equivalences}. A direct proof of
      any of them could help to simplify the arguments presented here while
      retaining all of the same results, and provide further insight into the
      connections between $\Av(123)$ and $\Av(132)$. While each of these
      seem tractable to bijective methods, they have resisted many attempts at a
      direct proof  and we include them here partly out of spite.
      First, we present the generating functions for the occurrences of $231$ and
      $321$, which follow by routine (but technical) computation.

      \begin{theorem} \label{expat:thm:231pats}
        The number of $231$ (or $312$) occurrences is given by 
        $$\sum_{n \geq 0} \num_{231}(\Av_n(123)) z^n=
        \frac{3z-1}{(1-4z)^{2}} - \frac{4z^2 - 5z + 1}{(1-4z)^{5/2}}.$$
      \end{theorem}

      \begin{corollary} \label{expat:cor:bridge}
        The total number of $231$ occurrences in $\Av_n(123)$ is equal to the
        number in $\Av_n(132)$. 
      \end{corollary}

      \begin{theorem}
        The number of $321$ occurrences is given by 
        $$
          \sum_{n\geq 0} \num_{321}(\Av_n (123)) z^n =
          \frac{ 8z^3 - 20z^2 + 8z - 1}{(1-4z)^{2}}
          - \frac{36z^3 - 34z^2 + 10z - 1}{(1-4z)^{5/2}}.
        $$
      \end{theorem}

      \begin{corollary} \label{expat:cor:equivalences}
        The following identities hold 
        $$ \num_{21}(\Avn) = 2\num_{213}(\Avns) $$
        $$ \num_{213}(\Avn) + \num_{231}(\Avn) = \num_{231} (\Av
        _{n-1}^*(123))$$
        $$ C(z) \left(\sum_{n\geq 0} \num_{213}(\Avn) z^n \right) =
          z C'(z) \left(\sum_{n \geq 0} \num_{12}(\Avn)z^n \right) $$
        $$ \sum_{n\geq 0} \num_{213} (\Av_n^* (132)  z^n) =
          \sum_{n \geq 0} \big(\num_{132}(\Avns) +
          \num_{231}(\Avns)\big)z^n. $$
      \end{corollary}

    Now we can do some analysis of the main sequences. Using some
    standard generating function analysis \cite{flajolet}, we find
    that the asymptotic growth of the number of length $3$ patterns are
    as follows:

    \vspace{1pc}

    $$ \begin{aligned}
     \num_{213} (\Avn) &\sim \sqrt{\frac{n}{\pi}} 4^{n-1} \\
     \num_{231} (\Avn) &\sim \frac{n}{2} 4^{n-1}  \\
     \num_{321} (\Avn) &\sim \frac{2}{3} \sqrt{\frac{n^3}{\pi}} 4^{n-1}. 
    \end{aligned} $$

    We see that the three sequences each differ by a factor of
    approximately $\sqrt{n}$. Surprisingly, this is the same factor that
    the sequences $\num_{123}, \num_{231}, \num_{321}$ differ by in the
    class $\Av (132)$, as seen in \cite{Bona2012}.

    Each of these generating functions are simple enough that exact
    formulas can be obtained with relatively little hassle. One could
    argue that the asymptotic values are more interesting and
    provide more insight than the complicated formulas, but we present
    them here for completeness.

    \begin{corollary}
      Let $a_n = \num_{132}(\Avn)$, $b_n = \num_{213}(\Avn)$, and $d_n =
      \num_{321}(\Avn)$. Then we have that
      $$ a_n = \frac{n+2}{4} \binom{2n}{n} - 3 \cdot 2^{2n-3} $$

      \vspace{1pc}

      $$ b_n = (2n-1) \binom{2n-3}{n-2} - (2n+1)\binom{2n-1}{n-1} +
         (n+4) \cdot 2^{2n-3}$$
      
      \vspace{1pc}

      $$ \begin{aligned} d_n
          = \frac{1}{6} \binom{2n+5}{n+1} \binom{n+4}{2}
          &- \frac{5}{3} \binom{2n+3}{n} \binom{n+3}{2}
          + \frac{17}{3} \binom{2n+1}{n-1} \binom{n+2}{2} \\
          &- 6\binom{2n-1}{n-2} \binom{n+1}{2} - (n+1) \cdot 4^{n-1}.
        \end{aligned}
      $$
    \end{corollary}

  %===================================================================%
  \subsection{Larger Patterns}

    Some of these same techniques are applicable to larger patterns.  For
    example, we can easily modify Lemmas~\ref{expat:lem:twopats}
    and~\ref{expat:lem:totalpats} to apply to patterns of all sizes. This leads to
    increasingly complicated expressions, but this simple idea can be used to
    prove the following proposition. 

    \begin{proposition}
      Let $k \in \mathbb{Z}^+$, and $\sg$ be any permutation in $\S_k$
      other than the decreasing permutation. Then for $n$ large enough,
      we have that
      $$\num_{k \dots 3 2 1}(\Avn) > \num_{\sg}(\Avn).$$
    \end{proposition}

    \begin{proof}
      Let $\cD$ be the set of permutation in $\S_k$ which are not the
      decreasing permutation.  As in Lemma~\ref{expat:lem:twopats},  we can
      express the number $\binom{n-2}{k-2}\num_{12}(\Avn)$ as a positive linear
      combination of all of $\num_{\sg} (\Avn)$ where $\sg \in \cD$. As in
      Lemma~\ref{expat:lem:totalpats}, we can express $\binom{n}{k} c_n$ as the sum
      of all $\num_{\ro} (\Avn)$ where $\ro \in \S_n$.  It follows that there is
      a positive integer $m$ and positive
      integers $e_i$ such that
      $$ \binom{n}{k} c_n - m\binom{n-2}{k-2} \num_{12} (\Avn) = \num_{k \dots
      321} - \sum_{\sg \in \cD} e_i \num_{\sg} (\Avn).$$ Asymptotic analysis
      shows that the left hand side is eventually positive, and so the first term
      on the right side eventually outgrows the second term, which completes the
      proof.
    \end{proof}

\cleardoublepage
\typeout{******************}
\typeout{**  Chapter 3   **}
\typeout{******************}

  \chapter{Pattern Avoiding Involutions} 
  \label{chap:involutions}

    In this chapter, we investigate sets of pattern-avoiding \emph{involutions}
    \index{involution}. While the enumeration of pattern-avoiding permutations
    has become a major topic of research in recent years, involutions have been
    largely overlooked.  In particular, we focus on finding the Stanley-Wilf
    limit for sets of involutions which avoid patterns of length four. 
    
    Pattern-avoiding involutions were first considered by Simion and
    Schmidt~\cite{Simion1985}, who enumerated the involutions avoiding any length 
    three pattern. As in the case for permutations, the situation quickly becomes
    more complicated for longer patterns. 
    We begin this chapter by examining the simple $123$
    involutions, which will be our primary tool. This chapter is based in part
    on~\cite{me-involutions}.

  % =========================================================================== %
  \section{Definitions and Context}
  \label{involutions:sec:intro}

    \begin{definition}\label{involutions:def:invclass}
      For a given permutation $\bta$, let $\Av^I(\bta)$ denote the set of
      \emph{$\bta$-avoiding involutions}, the set of involutions (permutations
      which are their own inverse) which do not contain $\bta$. 
      Let $\Av_n^I (\bta)$ be the set of permutations of length $n$ within this set. 
    \end{definition}

    Note that $\Av^I(\bta)$ is not necessarily a \emph{class}, as the set of all
    involutions is not closed under the pattern ordering. However we can apply
    many of the same ideas in order to enumerate these sets.  Clearly, $\Av^I
    (\bta) \subseteq \Av(\bta)$, and so the Marcus-Tardos theorem states that
    each set has a finite upper growth rate. Note that due to the symmetry of
    inversion ($\sg \prec \pi$ if and only if $\sg^{-1} \prec \pi^{-1}$), these
    classes are not principally based in the classical sense. Indeed,
    $\AvnI(\bta) = \AvnI(\bta, \bta^{-1})$ for any permutation $\bta$. For
    simplicity of notation, and to parallel the work done in permutations, we
    write such a set with only a single basis element. 

    \subsection{Previous Results}

      Two patterns $\bta, \tau$ are \emph{involution Wilf-equivalent} if
      $|\AvnI(\bta)| = |\AvnI(\tau)|$. 
      Simion and Schmidt completed the classification of the involution
      Wilf-equivalence classes of patterns of length three in their 1985
      paper~\cite{Simion1985} by showing that, for all patterns $\bta \in
      \{123,132,213,321\}$ and $\sg \in \{231, 312\}$,  
      $$ | \AvnI(\bta)| = \binom{n}{\floor{n/2}} \quad 
            \text{and} \quad |\AvnI(\sg) | = 2^{n-1}.$$

      Extending the work of Guibert, Pergola, and Pinzani~\cite{Guibert2001},
      Jaggard~\cite{Jaggard2002} classified the eight involution Wilf-equivalence
      classes for length four patterns.  Of these classes, only two have been
      successfully enumerated: Gessel~\cite{Gessel1990} counted the set
      $\AvnI(1234)$, while Brignall, Huczynska, and Vatter~\cite{Brignall2008}
      provided the enumeration for $\AvnI(2413)$. In this chapter we enumerate two
      of these unknown sets ($\AvnI(1342)$ and $\AvnI(2341)$), and provide bounds
      for a third ($\AvnI(1324)$). 

      Jaggard~\cite{Jaggard2002} computed the values $|\AvnI(\bta)|$ for each $\bta$
      of length four, up to $n=11$. This data (Table~\ref{involutions:tab:Jaggard})
      suggests an ordering on the eight classes, which we will show is misleading.
      For example, it seems clear from his data that there are more involutions
      avoiding $2341$ than avoiding $1234$. However, there are exponentially more $1234$
      avoiding involutions, as we will soon show.

      \begin{table}[t]
      \caption[The enumeration of involutions avoiding a pattern]{The
                enumerations of involutions avoiding a pattern $\beta$ of length
                $4$ for $n=5$, $\dots$, $11$, as presented by
                Jaggard~\cite{Jaggard2002} (ordered by the last row).}
      \label{involutions:tab:Jaggard}
        $$
        \begin{array}{ccccccccc}%{|c|c|c|c|c|c|c|c|c|}
        &&&&&&&&\\[-10pt]
        &\bm{1324}&\bm{1234}&\bm{4231}&\bm{2431}&\bm{1342}&\bm{2341}&\bm{3421}&\bm{2413}\\[1pt]
        \hline
        &&&&&&&&\\[-11pt]
        |\Av^I_{5}(\beta)|&21&21&21&24&24&25&25&24\\[1pt]
        &&&&&&&&\\[-11pt]
        |\Av^I_{6}(\beta)|&51&51&51&62&62&66&66&64\\[1pt]
        &&&&&&&&\\[-11pt]
        |\Av^I_{7}(\beta)|&126&127&128&154&156&170&173&166\\[1pt]
        &&&&&&&&\\[-11pt]
        |\Av^I_{8}(\beta)|&321&323&327&396&406&441&460&456\\[1pt]
        &&&&&&&&\\[-11pt]
        |\Av^I_{9}(\beta)|&820&835&858&992&1040&1124&1218&1234\\[1pt]
        &&&&&&&&\\[-11pt]
        |\Av^I_{10}(\beta)|&2160&2188&2272&2536&2714&2870&3240&3454\\[1pt]
        &&&&&&&&\\[-11pt]
        |\Av^I_{11}(\beta)|&5654&5798&6146&6376&7012&7273&8602&9600\\[1pt]
        \end{array}
        $$
      \end{table}

    \subsection{Simple Involutions}

      Our primary tool will be the substitution decomposition. Inflations and
      involutions are linked by the following theorem, which provides a recipe
      for constructing new involutions from simples. 

      \begin{theorem}[Brignall, Huczynska, Vatter~\cite{Brignall2008}]
        \label{thm:subsdecomp-inv}
        Let $\sg \neq 21$ be a simple permutation of length $m$, and $\alp_1, 
        \alp_2, \dots \alp_m$. Then $\pi = \sg[\alp_1, \alp_2, \dots \alp_m]$ is
        an involution if and only if $\sg$ is an involution and $\alp_i =
        \alp_{\sg_i}^{-1}$. 
        Further, the skew decomposable involutions are either of the form
        $21[\alp_1, \alp_2]$ with $\alp_1 = \alp_2^{-1}$ or $321[\alp_1, \alp_2,
        \alp_3]$ with $\alp_1 = \alp_3^{-1}$ and $\alp_2 = \alp_2^{-1}$. 
      \end{theorem}

      Describing classes as restricted inflations of their simple permutations is
      a new and useful method for enumerating classes of
      \emph{permutations}~\cite{Albert2012}, and we adapt this method to pattern-avoiding
      \emph{involutions}. As we will show, the simple $2341$-avoiding and
      $1342$-avoiding involutions are (almost) the same as the simple
      $123$-avoiding involutions. The enumerations of these sets can then be
      obtained by appropriately inflating these $123$-avoiding involutions.

  \section{Simple 123-Avoiding Permutations}
  \label{involutions:sec:perms}

    We step back from involutions briefly, and investigate the simple
    $123$-avoiding \emph{permutations}. This investigation, while interesting on
    its own, provides a gentle introduction to the generating function techniques
    of Section~\ref{involutions:sec:123simples}. In particular, we mirror the
    techniques used by Albert and Vatter~\cite{Vatter2013} to construct and
    analyze a generating function for the $123$-avoiding permutations. 

    \subsection{The Staircase Decomposition}
      
      In Section~\ref{prelim:sec:av123} we investigated the geometric
      structure of the class $\Av 123$, and showed that it contains infinitely
      many simple permutations. While this class is not a grid
      class~\cite{GridClasses}, it can be defined using similar language. The
      \emph{staircase decomposition} of $\Av 123$ allows one to utilize many of
      the specialized techniques which are typically only applicable to grid
      classes, and is central to our study.  

      Every permutation $\pi \in \Av 123$ can be written as a union of two
      increasing sequences of entries (the left-to-right minima and the
      right-to-left maxima). The plot of such a permutation can be fit into a
      \emph{descending staircase} of blocks, the contents of which are monotone
      decreasing. See Figure~\ref{involutions:fig:staircase}. In general, such a
      decomposition is not unique, but for \emph{simple} $123$-avoiding
      permutations we can define a unique gridding as follows: let the first cell
      contain the longest decreasing prefix of the permutation, each eastward
      cell contain all entries whose value is greater than the smallest in the
      previous cell, and each southward cell contain all entries to the left of
      the rightmost entry of the previous cell. 
    
      \begin{figure}[t] \centering
        \begin{tikzpicture}[scale=.25, yscale=-1]
          % draw the outer boxes, using a loop
          \foreach \x/\y in {0/0, 5/0, 5/5, 10/5}{
            \draw[very thick, color=lightgray] (\x, \y) rectangle (\x + 5, \y + 5);
          }
          % closed dots
          \node[closed] at (2,2) {};
          \node[closed] at (4,4) {};
          \node[closed] at (5.6,.6) {};
          \node[closed] at (5.9,5.9) {};
          \node[closed] at (6.2,1.2) {};
          \node[closed] at (7.6,7.6) {};
          \node[closed] at (8.4,3.4) {};
          \node[closed] at (10.5,5.5) {};
          \node[closed] at (12,7) {};
        \end{tikzpicture}
      \caption{The staircase decomposition for the permutation $759381642$.}
      \label{involutions:fig:staircase}
      \end{figure}
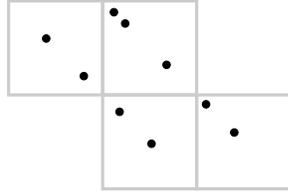

      This staircase decomposition was first introduced in~\cite{Albert2010} in
      the study of subclasses of $\Av 321$. As $123$ is the complement of $321$,
      our decomposition is a mirror image of theirs. Note that this decomposition
      separates the left-to-right minima and right-to-left maxima. We will use
      this fact later to build a bivariate generating function that keeps track
      of these entries separately. 

    \subsection{Iterative Process}

      Let $f = \sum_{\pi \in \Av_n 123} x^{n}$. We follow the exposition
      presented by Albert and Vatter in~\cite{Vatter2013} by first giving an
      \emph{almost} correct derivation, then fixing two small errors to obtain
      the correct result. 

      We can build a simple $123$-avoiding permutations iteratively using the
      staircase decomposition by filling one cell at the time. We must, however,
      be careful to ensure simplicity at each step along the way. To this end, we
      fill up an infinite staircase with \emph{filled dots} and \emph{hollow
      dots}; a filled dot represents an entry of the permutation, while a hollow
      dot represents a region which must be filled by at least one entry in order
      to maintain simplicity. Filled dots can be filled with a monotone run of
      entries, but each pair must be \emph{split} by a hollow dot in the next
      cell. Such a diagram with no hollow dots represents a simple $123$-avoiding
      permutations, while a diagram with hollow dots is still a work in progress.
      Since there are only two cells `active' at a time (the current one, and the
      next one), we can represent this process as an iterative system, and our
      goal is then to find a fixed point of the iteration. 

      We build $f$ one cell at a time. At step one, we have a single hollow dot
      in the first cell. At step two, we can fill this hollow dot with a
      descending run of filled dots, but each pair of these necessitates a hollow
      dot in the next cell to split them. During step three, each hollow dot in
      cell two can be filled with a descending run, but again we must place
      hollow dots in cell three to maintain simplicity. See
      Figure~\ref{involutions:fig:perm-iteration} for an example of this
      development.

      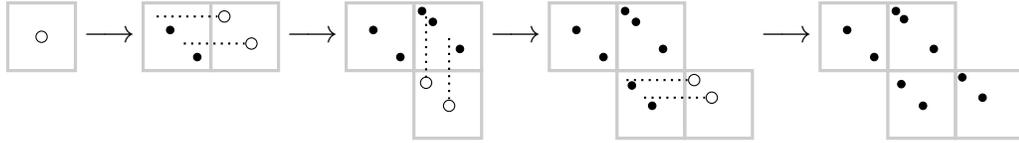
\begin{figure}[t] \centering
          \begin{tikzpicture}[scale=.18, yscale=-1]

            \foreach \x/\y in {0/0}{
              \draw[very thick, color=lightgray] (\x, \y) rectangle (\x + 5, \y + 5);
            }
            \node[open2] at (2.5,2.5) {};
            
          \node at (7.5,2.5) {$\longrightarrow$};

          \begin{scope}[shift={(10,0)}]
            \foreach \x/\y in {0/0, 5/0}{
              \draw[very thick, color=lightgray] (\x, \y) rectangle (\x + 5, \y + 5);
            }
            \node[closed] at (2,2) {};
            \node[closed] at (4,4) {};
            \node[open2] at (6,1) {};
            \node[open2] at (8,3) {};
            \draw[dotted, thick] (3,3)--(7.5,3);
            \draw[dotted, thick] (1,1)--(5.5,1);
          \end{scope}
          
          \node at (22.5,2.5) {$\longrightarrow$};
          
          \begin{scope}[shift={(25,0)}]
            \foreach \x/\y in {0/0, 5/0, 5/5}{
              \draw[very thick, color=lightgray] (\x, \y) rectangle (\x + 5, \y + 5);
            }
            \node[closed] at (2,2) {};
            \node[closed] at (4,4) {};
            \node[closed] at (5.6,.6) {};
            \node[open2] at (5.9,5.9) {};
            \node[closed] at (6.4,1.4) {};
            \node[open2] at (7.6,7.6) {};
            \node[closed] at (8.4,3.4) {};
            \draw[dotted, thick] (5.9,1)--(5.9,5.5);
            \draw[dotted, thick] (7.6,2.6)--(7.6,7.1);
          \end{scope}
          
          \node at (37.5,2.5) {$\longrightarrow$};

          \begin{scope}[shift={(40,0)}]
            \foreach \x/\y in {0/0, 5/0, 5/5, 10/5}{
              \draw[very thick, color=lightgray] (\x, \y) rectangle (\x + 5, \y + 5);
            }
            \node[closed] at (2,2) {};
            \node[closed] at (4,4) {};
            \node[closed] at (5.6,.6) {};
            \node[closed] at (6.1,6.1) {};
            \node[closed] at (6.4,1.4) {};
            \node[closed] at (7.6,7.6) {};
            \node[closed] at (8.4,3.4) {};
            \node[open2] at (10.7,5.7) {};
            \node[open2] at (12,7) {};
            \draw[dotted, thick] (7,7)--(11.5,7);
            \draw[dotted, thick] (5.7,5.7)--(10.2,5.7);
          \end{scope}
          
          \node at (57.5,2.5) {$\longrightarrow$};

          \begin{scope}[shift={(60,0)}]
            \foreach \x/\y in {0/0, 5/0, 5/5, 10/5}{
              \draw[very thick, color=lightgray] (\x, \y) rectangle (\x + 5, \y + 5);
            }
            \node[closed] at (2,2) {};
            \node[closed] at (4,4) {};
            \node[closed] at (5.6,.6) {};
            \node[closed] at (6,6) {};
            \node[closed] at (6.2,1.2) {};
            \node[closed] at (7.6,7.6) {};
            \node[closed] at (8.4,3.4) {};
            \node[closed] at (10.5,5.5) {};
            \node[closed] at (12,7) {};
          \end{scope}

          \end{tikzpicture}
        \caption{The evolution of the permutation $759381642$ by our recurrence.}
        \label{involutions:fig:perm-iteration}
      \end{figure}

      Let $f_i$ be the generating function at stage $i$ of this evolution, with the
      exponent of $x$ indicating the number of filled dots and the exponent of
      $y$ indicating hollow dots (so $f_1 = y$). A hollow dot can be filled with
      a run of filled dots, each pair of which requires a hollow dot, and we have
      the option of placing a new hollow dot above the run. It follows then that
      in each step, each occurrence of $u$ will be replaced by 
      $$ x(1+y) + x^2(y+y^2) + x^3 (y^2 + y^3) + \dots = \frac{x(1+y)}{1-xy}. $$ 
      Thus, we have
      {\small
      $$ f_1(x,y) = y, \quad f_2 
        = f_1\left(x, \frac{x(1+y)}{1-xy}\right) = \frac{x(1+y)}{1 - xy}, \quad 
        f_{i+1} = f_i\left(x, \frac{x(1+y)}{1-xy}\right) \dots .$$
      }

      Since we are interested in permutations with arbitrarily many staircase
      cells, we want to find the limit $f = \limn f_i $. It follows then that $f$
      is a \emph{fixed point} of the iteration $x \ra x; y \ra \frac{x(y+1)}{1 -
      xy}$. Since $f(x,y) = f\left(x, \frac{x(y+1)}{1 - xy}\right)$, can solve
      for $x$ to find
      $$ y = \frac{x(y+1)}{1 - xy} \implies 
         y = \frac{1 - x - \sqrt{1 - 2x - 3x^2}}{2x}.$$

      Thus we have
      $$ \begin{aligned}
        f &= f_1\left(x,\frac{1 - x - \sqrt{1 - 2x - 3x^2}}{2x}\right)\\
          &= \frac{1 - x - \sqrt{1 - 2x - 3x^2}}{2x} \\
          &= x + x^2 + 2x^3 + 4x^4 + 9x^5 + 21x^6 + \dots.
      \end{aligned} $$

      These coefficients are the Motzkin numbers, a well-studied and understood
      sequence (\OEIS{A001006}), but are unfortunately \emph{not} the number of
      simple $123$-avoiding permutations.  This is due to the aforementioned
      errors, which we will now correct.

    \subsection{Correcting the Errors}
    \label{involutions:sub:errors}

      Our iteration was correct, but there are some slight discrepancies arising
      in the first two steps of the iteration which must be accounted for. In the
      second step, the `optional' hollow dot above the topmost element is
      actually required, else the permutation will start with its largest entry
      (and therefore not be simple). Furthermore, when this required dot is
      inflated in the third step, the optional dot is in fact forbidden, else we
      will violate the greediness of the gridding. See
      Figure~\ref{involutions:fig:errors} for an illustration.

      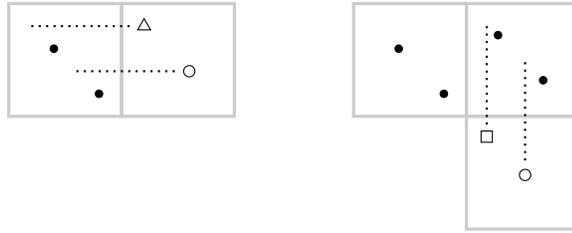
\begin{figure}[t] \centering
        \begin{tikzpicture}[scale=.3, yscale=-1]
          \foreach \x/\y in {0/0, 5/0}{
            \draw[very thick, color=lightgray] (\x, \y) rectangle (\x + 5, \y + 5);
          }
          \draw[color=white] (-1,-1) -- (10.5,-1) -- (10.5,10.5) -- (-1,10.5) -- cycle;
          \node[closed] at (2,2) {};
          \node[closed] at (4,4) {};
          \node[open2, regular polygon,regular polygon sides=3] at (6,1) {};
          \node[open2] at (8,3) {};
          \draw[dotted, thick] (3,3)--(7.5,3);
          \draw[dotted, thick] (1,1)--(5.5,1);
          \node[white] at (5,10) {};
        \end{tikzpicture}
          \hspace{2pc} 
        \begin{tikzpicture}[scale=.3, yscale=-1]
          \foreach \x/\y in {0/0, 5/0, 5/5}{
            \draw[very thick, color=lightgray] (\x, \y) rectangle (\x + 5, \y + 5);
          }
          \draw[color=white] (-1,-1) -- (10.5,-1) -- (10.5,10.5) -- (-1,10.5) -- cycle;
          \node[closed] at (2,2) {};
          \node[closed] at (4,4) {};
          \node[zball, minimum size=.15cm] at (5.9,5.9) {};
          \node[closed] at (6.4,1.4) {};
          \node[open2] at (7.6,7.6) {};
          \node[closed] at (8.4,3.4) {};
          \draw[dotted, thick] (5.9,1)--(5.9,5.5);
          \draw[dotted, thick] (7.6,2.6)--(7.6,7.1);
        \end{tikzpicture}
        \caption[The hollow triangle represents the location of the hollow dot
                  which is required]{
            The hollow triangle represents the location of the hollow dot
            which is required, and the hollow square represents the location of
            the hollow dot which is forbidden.}
      \label{involutions:fig:errors}
      \end{figure}
    
      Fortunately, however, these issues only affect the first three iterations:
      afterwards, the iteration works as initially described. We can therefore
      compensate by simply computing the first three by hand, and then plug in
      the value of $y$ which leads to the fixed point, as found above. 
      As above, we have $f_1 = y$. Since the next optional point is required and
      will be treated differently in the next step, we mark it with a $t$ to
      differentiate it from the standard hollow dots. Thus 
      $$ f_2(x,y,t) = \frac{xt}{1 - xy}.$$
      To compute $f_3$, we perform the standard iteration on the variable $y$,
      and change the variable $t$ into a generating function representing runs of
      filled dots with no option to place one above. This leads to 
      $$ f_3 = f_2 \left( x, \frac{x(y+1)}{1-xy}, \frac{x}{1-xy}\right).$$
    
      At this point the standard iteration, taken to infinity, produces the
      correct generating function, which can be used to enumerate the class $\Av
      (123)$, as shown in Section~\ref{prelim:sec:av123}.

      \begin{equation} \label{eqn:genfcn-simple123perms}
      \begin{split}
        f(z) &= f_3\left(x, \frac{1 - x - \sqrt{1 - 2x - 3x^2}}{2x}\right) \\
          &= \frac{2x^2}{1 + x^2 + (1+x)\sqrt{1 - 2x - 3x^2}} \\
          &= x^2 + 2x^4 + 2x^5 + 7x^6 + 14x^7 + 37x^8 + \dots.
        \end{split}
      \end{equation}
      The coefficients are \OEIS{A187306}.

  \section{Simple 123-Avoiding Involutions}
  \label{involutions:sec:123simples}

    We return now to the problem of enumerating the simple $123$-avoiding
    involutions. Though this is more difficult, the iterative development of the
    generating function for the simple $123$-avoiding permutations presented above
    forms the basis for our study. As we will eventually be inflating these
    involutions to enumerate the avoiding sets, we want to keep track of
    left-to-right minima ($\ltrm$), right-to-left maxima ($\rtlm$), and fixed
    points ($\fp$) separately.  Our goal here will be to find the generating
    functions $s^{(i)}(u,v)$, defined below
    $$ s^{(i)}(u,v) := \sum_{\substack{
              \text{simple } \sg \in \AvnI 123 \\
              \text{with } \fp(\sg) = i }}
              u^{\ltrm(\sg)} v^{\rtlm(\sg)} .$$

    \subsection{Extending the Iteration} 

      We proceed defining an iterative process similar to the development
      presented in Section~\ref{involutions:sec:perms}. 
      This iterative process can be extended in a variety of ways, as we will
      soon see. Note, for example, that we could have used a two-part recurrences
      to keep track of the top cells and bottom cells separately; it follows then
      that this process can be used to enumerate the left-to-right minima
      \emph{separately} from the right-to-left maxima with a more technical (but
      no more conceptually difficult) computation. The following sections will
      rely on some tedious and technical calculations, but the core ideas are
      relatively easy to express. 

      Geometrically, an involution is a permutation whose plot is symmetric about
      the line $y = x$ through the plane. As such, we can build a simple
      $123$-avoiding involution using the staircase decomposition \emph{starting
      from the center}, and building out in both directions.
      Figure~\ref{involutions:fig:staircase-center} shows the two possible cases.
      When there is a single fixed point, the case is uniquely determined by
      considering whether the fixed point is a \rtlmax{} or \ltrmin{}.

      \begin{figure}[t]
      \centering
        \begin{tikzpicture}[scale=.45]
          \draw[thick,fill=lightgray] (0,0) rectangle (2,2);
          \draw[thick] (-2,0) rectangle (0,2);
          \draw[thick] (0,0) rectangle (2,-2);
          \draw[thick] (-2,2) rectangle (0,4);
          \draw[thick] (2,-2) rectangle (4,0);
          \node at (-2.6,4.7) {$\ddots$};
          \node at (4.55,-2.15) {$\ddots$};
        \end{tikzpicture}
        \hspace{4pc}
        \begin{tikzpicture}[scale=.45,cm={-1,0,0,-1,(0,0)}]
          \draw[thick,fill=lightgray] (0,0) rectangle (2,2);
          \draw[thick] (-2,0) rectangle (0,2);
          \draw[thick] (0,0) rectangle (2,-2);
          \draw[thick] (-2,2) rectangle (0,4);
          \draw[thick] (2,-2) rectangle (4,0);
          \node at (-2.55,4.2) {$\ddots$};
          \node at (4.55,-2.6) {$\ddots$};
        \end{tikzpicture}
        \caption[The diagrams on which we can draw simple permutations]{
          The diagrams on which we can draw simple permutations $\sigma
          \in \Av^I(123)$ that contain a single fixed point.  The starting point
          of the iteration is the shaded cell.}
        \label{involutions:fig:staircase-center}
      \end{figure}
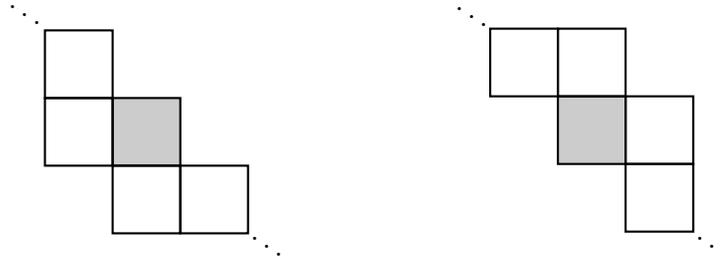
    
      As in Section~\ref{involutions:sec:perms}, we start with a single hollow
      dot in the center cell, and proceed outwards in both directions
      simultaneously while mainaining symmetry. However, the number of fixed
      points determines how we proceed from here. In the interest of clarity, we
      develop the single fixed point case in detail, and give a sketch of the
      details of the other cases. 

    \subsection{Single Fixed Point}

      We first develop the generating function $\htso(z) = s^{(1)}(z,z)$, which
      only keeps track of the number of such permutations of each length and
      ignores the ltr-min and rtl-max, and then indicate how to obtain the
      more general $\htso(u,v)$. The set of all simple $(123)$-avoiding
      involutions with exactly one fixed point can be partitioned based on
      whether the fixed point is a ltr-min or a rtl-max. These two sets are
      in bijection with each other, as mapping a permutation to its reverse
      complement maps one set to the other. Therefore it suffices to enumerate
      those in which the fixed point is a rtl-max, and then simply double the
      result to obtain the full generating function (or in the case of $s^{(1)}$,
      add the result to itself with the rtl-max and ltr-min switched). 
    
      Assume that the fixed point is a rtl-max. The first hollow dot must then
      be inflated by an odd number of filled dots (with the fixed point at the
      center). The hollow dots here behave a bit differently than in the previous
      section: each pair of filled dots can be split either below or to the left,
      or both. Of these possible splittings, one of them (see
      Figure~\ref{involutions:fig:bad-case} yields a skew decomposable
      permutation, violating the simplicity condition. We can account for this
      with a calculation which takes the symmetry into account. 

      \begin{figure}[t]
      \centering
        \begin{tikzpicture}[scale=.15]
          % draw the outer boxes, using a loop
          \foreach \x/\y in {-10/0, 0/0, 0/-10}{
            \draw[very thick, color=lightgray] (\x, \y) rectangle (\x + 10, \y + 10);
          }
          % dotted line through diagonal
          \draw[dotted] (-9,-9) -- (11,11);

          % closed dots
          \foreach \i in {-4, -2, 0, 2, 4}{
            \node[closed] at (5-\i,5+\i) {};
          }
          % open dots
          \foreach \i in {1,3}{
            \node[open] at (-5-\i, 5+\i) {};
            \draw[color=gray] (-4.5-\i,5+\i) -- (5.5-\i, 5+\i);
          }
          \foreach \i in {-1,-3}{
            \node[open_y] at (5-\i, -5+\i) {};
            \draw[color=gray] (5-\i,-4.5+\i) -- (5-\i, 5.5+\i);
          }
        \end{tikzpicture}
      \caption[An example of a bad placement]{
          An example of a bad placement of splitting entries that leads to
          a skew decomposable permutation.} 
      \label{involutions:fig:bad-case}
      \end{figure}
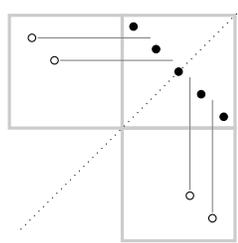

      Suppose that the initial cell (which contains the fixed point) contains a total
      of $2k+1$ entries. It follows that $k$ of these entries lie below and to the
      right of the fixed point. Because $\sigma$ is simple, each of the $2k$ adjacent
      pairs of entries in this cell must be separated by entries in the cell below,
      by entries in the cell to the left, or by both types of entries. Each adjacent
      pair lying above and to the left of the fixed point has a corresponding
      adjacent pair (its image under inversion) which lies below and to the right of
      the fixed point; if we split the former to the left, then the inverse-image of
      the separating entry splits the latter below, and vice versa.

      We can split each adjacent pair with as few as $k$ entries in the cell below
      the fixed point, and this can be done in $2^k$ ways by picking which of each
      two corresponding pairs of entries to split below. Similarly, the number of
      ways to have $k+i$ separating entries in the cell below is given by $2^{k-i}{k
      \choose i}$, since we can first pick which of the $i$ corresponding pairs of
      gaps between entries are split both to the left and below, then we choose which
      of each of the remaining $k-i$ corresponding pairs are split below or to the
      left.

      As in the derivation in Section~\ref{involutions:sub:errors}, there are a
      few slight difficulties we must take into account, but again they only
      arise in the first three steps of the iteration. We therefore construct
      these three steps by hand, before letting the iteration go to infinity. 

      Every choice of separating entries leads to a simple permutation except
      one: if we split all of the pairs of entries to the right of the fixed point by
      entries below the initial cell and split no other pairs, then the resulting
      permutation will be skew decomposable, as shown in
      Figure~\ref{involutions:fig:bad-case}. We compensate for these ``bad cases'' by
      subtracting the term $x/(1 - x^2y)$.

      It follows that
      $$ \begin{aligned}
        \htso_2(x,y,z)
        &=
        \frac{2z}{y}\left(\sum_{k=0}^\infty \left(x^{2k+1}\sum_{i=0}^k
            2^{k-i}{k\choose i}y^{k+i}\right) - \frac{x}{1-x^2y}\right)\\ 
        &=
        \frac{2x^3z(1+y)}{(1-x^2y)(1-2x^2y-x^2y^2)}.
      \end{aligned} $$
      % s2 := 2*z*(x^2*y + x^2) / (x^4*y^3 + (2*x^4 - x^2)*y^2 - 3*x^2*y + 1);

      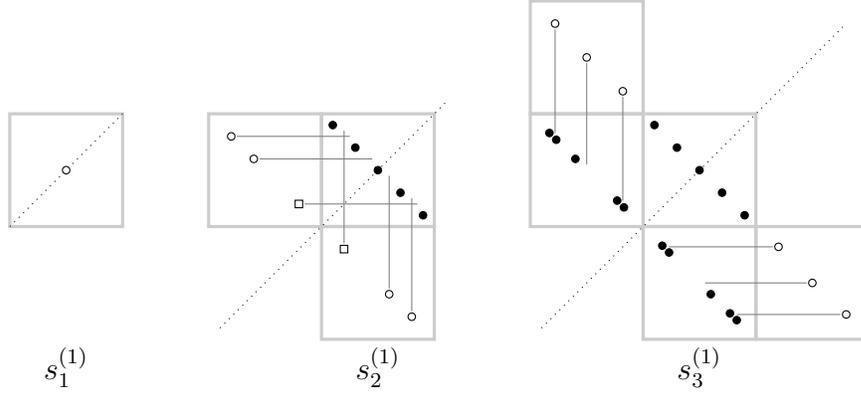
\begin{figure}[t]
        \centering
        \begin{tikzpicture}[scale=.15]

          % draw the outer boxes, using a loop
          \foreach \x/\y in {0/0}
          \draw[very thick, color=lightgray] 
          (\x, \y) -- (\x + 10, \y) -- (\x + 10, \y + 10) -- (\x, \y + 10) -- cycle;
          % dotted line through diagonal
          \draw[dotted] (0,0) -- (10,10);

          % draw the center y dot
          \node[open_y] at (5,5) {};
          \node[fill=none] at (5,-12.5) {$s^{(1)}_1$};
        \end{tikzpicture}
        \hspace{2em}
        \begin{tikzpicture}[scale=.15]
          % draw the outer boxes, using a loop
          \foreach \x/\y in {-10/0, 0/0, 0/-10}
          \draw[very thick, color=lightgray] 
          (\x, \y) -- (\x + 10, \y) -- (\x + 10, \y + 10) -- (\x, \y + 10) -- cycle;
          % dotted line through diagonal
          \draw[dotted] (-9,-9) -- (11,11);

          % closed dots
          \foreach \i in {-4, -2, 0, 2, 4}{
            \node[closed] at (5-\i,5+\i) {};
          }
          % open dots
          \foreach \i in {1,3}{
            \node[open] at (-5-\i, 5+\i) {};
            \draw[color=gray] (-4.5-\i,5+\i) -- (5.5-\i, 5+\i);
          }
          \foreach \i in {-1,-3}{
            \node[open_y] at (5-\i, -5+\i) {};
            \draw[color=gray] (5-\i,-4.5+\i) -- (5-\i, 5.5+\i);
          }
          \node[zball] at (2,-2) {};
          \node[zball] at (-2,2) {};
          \draw[color=gray] (2,-1.5) -- (2, 8.5);
          \draw[color=gray] (-1.5,2) -- (8.5,2);

          \node[fill=none] at (5,-12.5) {$s^{(1)}_2$};
        \end{tikzpicture}
        \hspace{2em}
        \begin{tikzpicture}[scale=.15] % s3
          % draw the outer boxes, using a loop
          \foreach \x/\y in {-10/10, -10/0, 0/0, 0/-10, 10/-10}
          \draw[very thick, color=lightgray] 
          (\x, \y) -- (\x + 10, \y) -- (\x + 10, \y + 10) -- (\x, \y + 10) -- cycle;
          % dotted line through diagonal
          \draw[dotted] (-9,-9) -- (18,18);

          % closed dots
          \foreach \i in {-4, -2, 0, 2, 4}{
            \node[closed] at (5-\i,5+\i) {};
          }
          \foreach \i in {1,-3.3,-2.7, 2.7,3.3}{
            \node[closed] at (-5-\i, 5+\i) {};
          }
          \foreach \i in {2.7,3.3, -1,-2.7, -3.3}{
            \node[closed] at (5-\i, -5+\i) {};
          }

          % open dots
          \foreach \i in {-3, 3}{
            \draw[color=gray] (14.5-\i, -4.8+\i) -- (5.3-\i, -4.8+\i);
            \draw[color=gray] (-4.8+\i, 14.5-\i) -- (-4.8+\i, 5.3-\i);
            \node[open_y] at (15-\i, -4.8+\i) {};
            \node[open] at (-4.8+\i, 15-\i) {};
          }
          % extra open dot
          \draw[color=gray] (14.5,-5) -- (5.5, -5);
          \draw[color=gray] (-5,14.5) -- (-5, 5.5);
          \node[open_y] at (15,-5) {};
          \node[open] at (-5,15) {};
          \node[fill=none] at (5,-12.5) {$s^{(1)}_3$};
        \end{tikzpicture}
      \caption{Three stages of the recurrence, in the case when the single
              fixed point is a right-to-left maximum.}
      \label{involutions:fig:3stages}
      \end{figure}

      The $2$ in $\htso_2$ accounts for both cases, where the fixed point is a
      \rtlmax{} and a \ltrmin{}, while the $z/y$ factor counts the topmost
      hollow dot in the cell below the fixed point by $z$ instead of $y$, as it
      will require special care. By our definition of greediness, this topmost
      hollow dot, shown as a hollow square in
      Figure~\ref{involutions:fig:3stages}, is not allowed to produce an
      hollow dot above it in the next cell. Therefore, when substituting for $z$ to
      obtain $\htso_3$, we substitute $x^2/(1-x^2y)$ instead of
      $x^2(1+y)/(1-x^2y)$. As such, we obtain

      $$ \htso_3(x,y)
          =s_2\left(x, \frac{x^2(1+y)}{1-x^2y}, \frac{x^2}{1-x^2y}\right). $$
      % s3 := subs({y = x^2*(1+y) / (1 - x^2*y), z = x*(1 - x^2*y)}, s2);

      After this point, the same iteration leads from $\htso_i$ for $\htso_{i+1}$
      for all $i \geq 3$. Since the filled dots above the center cell are
      completely determined by those below, we need only consider the expansion
      of hollow dots in the bottom cell. Their expansion is exactly as in
      Section~\ref{involutions:sec:perms}, except that each expansion of a hollow
      dot adds dots in both the bottommost cell and the topmost. Letting $i \geq
      3$, this leads to the relation
      \begin{equation}\label{involutions:eqn:iteration-inv}
        \htso_{i+1}(x,y) = \htso_i \left(x, \frac{x^2(y+1)}{1-x^2y} \right).
      \end{equation}

      To find the limit of this iteration, it suffices to find at fixed point,
      and plug it in for $y$ in the expression $\htso_3(x,y)$. 
      This leads to 
      \begin{equation} \label{involutions:eqn:htso}
        \begin{split}
        \htso(x) 
        &= \htso_3 \left(x, \frac{1 - x^2 - \sqrt{1-2x^2 - 3x^4}}{2x^2}\right) \\
        &= \frac{2x^5 \left(1 + x^2 + \sqrt{1 - 2x^2-3x^4}\right)}{%
            (1+x^2)^2 \left(1 - 3x^2 + (1-2x^2)\sqrt{1-2x^2-3x^4}\right)} \\
        &= 2x^5 + 2x^7 + 10x^9 + 22x^{11} + 68x^{13} + 184x^{15} + 530x^{17}
        +\dots \\
        \end{split}
      \end{equation}

      Note that an involution with only a single fixed point is necessarily of
      odd length, and so the power series in equation~\ref{involutions:eqn:htso}
      contains no terms with even powers.

      Rather than repeat this full derivation to find $\sone(u,v)$, we simply
      indicate the changes to make to the above calculation.  Recall that $u$
      (resp. $v$) represents a filled dot which is a ltr-min (reps. rtl-max), and
      introduce new variables $y_u$ and $y_v$ which represent hollow dots which
      are ltr-min and rtl-max, respectively.  We can assume that the fixed point
      is a rtl-max, because then we can just add this generating function to
      itself with the $u$ and $v$ swapped to obtain the full generating function
      $\sone(u,v)$. 

      A hollow dot in a lower cell, represented by $y_u$, then leads to filled
      dots in the lower cell (represented by $u$) and hollow dots in an upper
      cell (represented by $y_v$s).  A similar description of hollow dots in an
      upper cell leads to the iterations
      \begin{equation}\label{involutions:eqn:dual-iterations}
        \begin{split}
        y_u &\mapsto \frac{u^2 (1 + y_v)}{1 - u^2 y_v} \\ 
        y_v &\mapsto \frac{v^2 (1 + y_u)}{1 - v^2 y_u}.
        \end{split}
      \end{equation}

      To find the fixed point of this iteration, we can compute two iterations
      and solve. That is, solve for $y_v$ in the expression 

      $$ \begin{aligned} 
          y_v &=  \frac{v^2 (1 + y_u)}{1 - v^2 y_u} \\
          &= \frac{v^2 \left(1 + \frac{u^2 (1 + y_v)}{1 - u^2 y_v}\right)}{%
            1-v^2 \frac{u^2 (1 + y_v)}{1 - u^2 y_v}} \\
          &= \frac{v^2 (1 + u^2}{%
              1 - u^2v^2 - u^2y_v - u^2v^2y_v}. 
      \end{aligned} $$

      Solving this system yields the fixed point of the iteration: 
      \begin{equation} \label{involutions:eqn:dualfixed}
        y_v = \frac{1 - u^2v^2 - \sqrt{1 - 6u^2v^2 - 4u^2v^4 - 
                                       4u^2v^2 - 3u^4v^4}}{%
                    2u^2(1+v^2)}.
      \end{equation}

      Mirroring the construction of $\htso$, we can derive
      $\sone_3(u,v,y_u, y_v)$ by hand using these extra variables. Note that
      there will be no $y_u$ terms in this expression, because at the third stage
      the only hollow dots will be in cells corresponding to left-to-right
      minima. The limit of the iteration is then given by plugging in the fixed
      point to this expression. This gives the generating function for the case
      when the fixed point is a rtl-max, but by swapping occurrences of $u$ and
      $v$ and then adding it back to itself, we obtain the full generating
      function $\sone$. 

      {\small
      \begin{equation} \label{involutions:eqn:sone}
        \begin{split}
        \sone(u,v) &= 
          \frac{u^2v^3(1+u^2)(1+2v^2+u^2v^2+r)}
          {(1+v^2)(1-6u^2v^2-4u^2v^4-4u^4v^2-3u^4v^4+(1-3u^2v^2-2u^4v^2)r)}\\
          &\text{where} \quad r := \sqrt{1-6u^2v^2-4u^2v^4-4u^4v^2-3u^4v^4}.
        \end{split}
      \end{equation}
      }

    \subsection{Zero and Two Fixed Points}
    \label{involutions:sub:02fp}
      
      We now turn to the remaining two cases, in which the involution has no
      fixed points or two fixed points. The derivation is largely the same as the
      single fixed point case, so we simply sketch the changes that must be made. 
      Each of these has their own idiosyncrasies, but they can be dealt with
      easily. 

      First, consider the case of involutions with no fixed points. Such a
      permutation cannot be uniquely gridded, because the diagonal line on which
      the fixed points would lie can be taken to pass through either a lower or
      upper central cell. It follows, however, that every involution with no
      fixed points can be decomposed in both ways, and so it suffices to assume
      that the diagonal line passes through an upper cell, and take this to be
      our initial cell. 

      Since there is no fixed point, this initial cell must have an even number
      of elements. We build the first three iterations by hand, in the same
      manner as the one fixed point case, before substituting the fixed point of
      the iteration. A similar bad case (Figure~\ref{involutions:fig:bad-case})
      must be accounted for, and the same restriction applies to the topmost
      hollow dot of the second cell, as shown in
      Figure~\ref{involutions:fig:3stages}. 

      The generating function $\htsz$ enumerating the class according to length, and the
      corresponding bivariate generating function $\szero$ enumerating the
      ltr-min and ltr-max entries are given below. 

      {\small
      \begin{equation} \label{involutions:eqn:htszero}
        \begin{split}
        \htsz(x) &=
        \frac{2x^6(1+x^2-\sqrt{1-2x^2-3x^4})}{2-2x^2-10x^4-6x^6+(2-6x^4-4x^6)\sqrt{1-2x^2-3x^4}}\\
        % s := (2*x^6 + 2*x^8 - 2*x^6*sqrt(1-2*x^2-3*x^4)) /
        % (2-2*x^2-10*x^4-6*x^6+(2-6*x^4-4*x^6)*sqrt(1-2*x^2-3*x^4))
        &= x^8+2x^{10}+8x^{12}+22x^{14}+68x^{16}+198x^{18}+586x^{20}+\cdots.
        \end{split}
      \end{equation}

      \begin{equation} \label{involutions:eqn:szero}
        \begin{split}
        \szero(u,v)
        &= \frac{2u^2v^4(1+u^2)(1+2u^2+u^2v^2-r)}
        {(1-u^2v^2+r)(1-6u^2v^2-4u^2v^4-4u^4v^2-3u^4v^4+(1+2v^2+u^2v^2)r)} \\
          &\text{where} \quad r := \sqrt{1-6u^2v^2-4u^2v^4-4u^4v^2-3u^4v^4}.
        \end{split}
      \end{equation}
      }

      Finally, we consider the case of involutions with two fixed points. As with
      the case of no fixed points, such a permutation can be drawn on either of
      the two diagrams shown in Figure~\ref{involutions:fig:staircase}. To ensure
      uniqueness, break our own rules slightly to say that the topmost fixed
      point is the center of the initial cell, while the bottom fixed point lies
      on the southwest corner of this cell. See
      Figure~\ref{involutions:fig:twofix} for an example, and note that in this
      case, the hollow square is allowed to produce a hollow dot above itself in
      the next cell, as this no longer violates the greediness of the
      decomposition (because of the lower fixed point).

      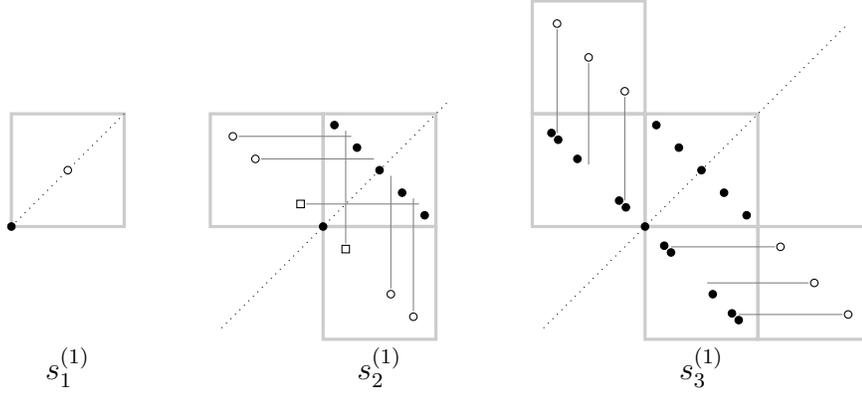
\begin{figure}[t]
        \centering
        \begin{tikzpicture}[scale=.15]

          % draw the outer boxes, using a loop
          \foreach \x/\y in {0/0}
          \draw[very thick, color=lightgray] 
          (\x, \y) -- (\x + 10, \y) -- (\x + 10, \y + 10) -- (\x, \y + 10) -- cycle;
          % dotted line through diagonal
          \draw[dotted] (0,0) -- (10,10);

          % draw the center y dot
          \node[open_y] at (5,5) {};
          \node[closed] at (0,0) {};
          \node[fill=none] at (5,-12.5) {$s^{(1)}_1$};
        \end{tikzpicture}
        \hspace{2em}
        \begin{tikzpicture}[scale=.15]
          % draw the outer boxes, using a loop
          \foreach \x/\y in {-10/0, 0/0, 0/-10}
          \draw[very thick, color=lightgray] 
          (\x, \y) -- (\x + 10, \y) -- (\x + 10, \y + 10) -- (\x, \y + 10) -- cycle;
          % dotted line through diagonal
          \draw[dotted] (-9,-9) -- (11,11);
          \node[closed] at (0,0) {};

          % closed dots
          \foreach \i in {-4, -2, 0, 2, 4}{
            \node[closed] at (5-\i,5+\i) {};
          }
          % open dots
          \foreach \i in {1,3}{
            \node[open] at (-5-\i, 5+\i) {};
            \draw[color=gray] (-4.5-\i,5+\i) -- (5.5-\i, 5+\i);
          }
          \foreach \i in {-1,-3}{
            \node[open_y] at (5-\i, -5+\i) {};
            \draw[color=gray] (5-\i,-4.5+\i) -- (5-\i, 5.5+\i);
          }
          \node[zball] at (2,-2) {};
          \node[zball] at (-2,2) {};
          \draw[color=gray] (2,-1.5) -- (2, 8.5);
          \draw[color=gray] (-1.5,2) -- (8.5,2);

          \node[fill=none] at (5,-12.5) {$s^{(1)}_2$};
        \end{tikzpicture}
        \hspace{2em}
        \begin{tikzpicture}[scale=.15] % s3
          % draw the outer boxes, using a loop
          \foreach \x/\y in {-10/10, -10/0, 0/0, 0/-10, 10/-10}
          \draw[very thick, color=lightgray] 
          (\x, \y) -- (\x + 10, \y) -- (\x + 10, \y + 10) -- (\x, \y + 10) -- cycle;
          % dotted line through diagonal
          \draw[dotted] (-9,-9) -- (18,18);
          \node[closed] at (0,0) {};

          % closed dots
          \foreach \i in {-4, -2, 0, 2, 4}{
            \node[closed] at (5-\i,5+\i) {};
          }
          \foreach \i in {1,-3.3,-2.7, 2.7,3.3}{
            \node[closed] at (-5-\i, 5+\i) {};
          }
          \foreach \i in {2.7,3.3, -1,-2.7, -3.3}{
            \node[closed] at (5-\i, -5+\i) {};
          }

          % open dots
          \foreach \i in {-3, 3}{
            \draw[color=gray] (14.5-\i, -4.8+\i) -- (5.3-\i, -4.8+\i);
            \draw[color=gray] (-4.8+\i, 14.5-\i) -- (-4.8+\i, 5.3-\i);
            \node[open_y] at (15-\i, -4.8+\i) {};
            \node[open] at (-4.8+\i, 15-\i) {};
          }
          % extra open dot
          \draw[color=gray] (14.5,-5) -- (5.5, -5);
          \draw[color=gray] (-5,14.5) -- (-5, 5.5);
          \node[open_y] at (15,-5) {};
          \node[open] at (-5,15) {};
          \node[fill=none] at (5,-12.5) {$s^{(1)}_3$};
        \end{tikzpicture}
      
      \caption{The decomposition of an involution with two fixed points.} 
      \label{involutions:fig:twofix}
      \end{figure}

      Note also that the `bad case' (Figure~\ref{involutions:fig:bad-case}) is no
      longer a bad case, as the lower fixed point maintains simplicity. Also, we
      are now allowed to add a hollow dot in the second cell immediately to the
      right of the lower fixed point, as long as we insert a hollow dot above
      this entry in the third cell. Taking these factors into consideration, we
      have the following generating functions for $\htsz$ and $\szero$. 

    \begin{equation} \label{involutions:eqn:htstwo}
      \begin{split}
      \htsz(x)
      &=
      \frac{x^4(2+5x^2+3x^4-(2+x^2)\sqrt{1-2x^2-3x^4})}
      {1-x^2-5x^4-3x^6+(1+2x^2+x^4)\sqrt{1-2x^2-3x^4}}\\
      &=
      3x^6+4x^8+15x^{10}+36x^{12}+105x^{14}+288x^{16}+819x^{18}+\cdots.
      \end{split}
    \end{equation}

    % s := (2*x^4+5*x^6+3*x^8-(2*x^4+x^6)*sqrt(1-2*x^2-3*x^4)) / (1-x^2-5*x^4-3*x^6+(1+2*x^2+x^4)*sqrt(1-2*x^2-3*x^4));
    \begin{equation} \label{involutions:eqn:stwo}
      \begin{split}
      \szero(u,v)
      &= \frac{uv^3\pa{2+7u^2+4u^2v^2+4u^4+3u^4v^2-(2+u^2)r}}
      {1-6u^2v^2-4u^2v^4-4u^4v^2-3u^4v^4+(1+2v^2+u^2v^2)r} \\
      &\text{where} \quad r := \sqrt{1-6u^2v^2-4u^2v^4-4u^4v^2-3u^4v^4}.
      \end{split}
    \end{equation}

    We can now combine the generating functions $\szero, \sone, \stwo$ to obtain
    a generating function for all simple $123$-avoiding permutations, enumerated
    by number of left-to-right minima and right-to-left maxima. However, it will
    be convenient to keep these separate, because in the next section we will
    explore inflations of these permutations, and oftentimes fixed points have
    different inflation rules from other entries.

  \section{Enumerating Pattern Avoiding Involutions}
  \label{involutions:sec:enumerations}
    
    We are now in position to enumerate the sets $\Av^I(1342)$ and $\Av^I(2341)$.
    Our tool for both of these is to first show that the simples in each set
    (almost) coincides with the simples within $\Av^I(123)$. This allows us to
    describe each of these sets by inflations of these simples, and so we need
    only determine what inflations are allowed to enumerate the sets.

    \subsection{Involutions Avoiding 1342}
    \label{involutions:sub:1342}

      Clearly, every involution avoiding $1342$ must also avoid $1342^{-1} =
      1423$. We first show that the set of simples in this set are precisely
      the $123$-avoiding simple involutions. This will be easy once we establish
      suitable notation. 

      \begin{definition} \label{involutions:def:subs-closure}
        Given a permutation class $\C$, define its \emph{substitution closure}
        $\vect{\C}$ to be the largest class with the same simple permutations as
        $\C$. 
      \end{definition}

      By definition, since $\Av(123) \subseteq \Av(1342, 1423)$, we have that
      the $123$-avoiding simples are contained in $\Av(1342, 1423)$.
      Atkinson, Ru\v{s}kuc, and Smith~\cite{Atkinson2011} investigated
      substitution closures, and found that 
      $$ \vect{\Av(123)} = \Av(24153, 25314, 31524, 41352, 246135, 415263).$$
      Each of these basis elements contains either $1342$ or $1423$, and so we
      have the following relation and its consequences. 
      $$ \Av(1342, 1423) \subseteq \vect{\Av(123)}.$$ 

      \begin{proposition} \label{involutions:prop:1342simpleperms}
        The simple permutations within $\Av(1342, 1423)$ are precisely the same
        as the simple permutations within $\Av(123)$. 
      \end{proposition}

      \begin{corollary} \label{involutions:cor:1342simple}
        The simple involutions within $\Av(1342, 1423)$ are precisely the same
        as the simple involutions within $\Av(123)$. 
      \end{corollary}

      To enumerate the set we now need only describe the allowable inflations
      which maintain pattern avoidance and involutionicity. We divide the simples
      into three classes: first we have the inflations of $1$, which themselves
      must be simple. Then come the inflations of $12$ and $21$, the sum- and
      skew-decomposable permutations, respectively. Finally we consider
      inflations of simples of length greater than three.

      We begin by by defining $f$ to be the generating function for the class
      $\Av(1342,1423)$ and $f_\oplus$ (resp., $f_\ominus$) the generating
      function for the sum (resp., skew) decomposable permutations of this class.
      We then define $g$ to be the generating function for the set $\Av^I(1342)$
      and $g_\oplus$ (resp., $g_\ominus$) the generating function for the sum
      (resp., skew) decomposable $1342$-avoiding involutions.

      First we describe the sum decomposable permutations
      $\pi=\alpha_1\oplus\alpha_2$ counted by $g_\oplus$. By
      Proposition~\ref{thm:subsdecomp}, we can assure uniqueness of
      decomposition by requiring that $\alpha_1$ is sum indecomposable. To produce
      an involution, $\alpha_1$ and $\alpha_2$ must be involutions as well. In
      order for $\pi$ to avoid the patterns $1342$ and $1423$, it is required that
      $\alpha_1$ avoids these patterns, and that $\alpha_2$ avoids the patterns
      $231$ and $312=231^{-1}$.

      In fact, the class $\Av(231,312)$, known as the class of \emph{layered
      permutations}, consists entirely of involutions because a permutation lies in
      $\Av(231,312)$ if and only if it can be expressed as a sum of some number of
      decreasing permutations. The layered permutations of length $n$ are in
      bijection with compositions of $n$, and hence there are $2^{n-1}$
      permutations of length $n$ in $\Av(231,312)$. Therefore, $g_\oplus$ satisfies
      the equation 
      
      $$ g_\oplus = \pa{g - g_\oplus}\pa{\frac{x}{1-2x}}. $$

      From this expression it follows that

        \begin{equation}
        \label{involutions:eqn:1342-1}
        g_\oplus = \frac{gx}{1-x}.
        \end{equation}

      Next we must briefly consider the permutation class $\Av(1342,1423)$.
      Kremer~\cite{Kremer2000, KremerPS} showed that
      this class is counted by the large Schr\"oder numbers, \OEIS{A006318}, and
      has generating function 
      $$ f(x) = \frac{1-x-\sqrt{1-6x+x^2}}{2}. $$
      Since this permutation class is skew closed (because both $1342$ and $1423$
      are skew indecomposable), it follows by Proposition~\ref{thm:subsdecomp}
      that, since 
      $f_\ominus = (f - f_\ominus)f$ and $f_\ominus = \frac{f^2}{1+f}$, 

      $$ f - f_\ominus = \frac{f}{1+f} = \frac{1+x-\sqrt{1-6x+x^2}}{4}. $$

      This is the generating function for the \emph{small} Schr\"oder numbers,
      \OEIS{A001003}.

      Returning our attention to $\Av^I(1342)$, which is also skew closed, we
      note that skew indecomposable permutations in this set are of the form
      $\alpha_1\ominus\alpha_2\ominus\alpha_1^{-1}$ where $\alpha_1$ is a skew
      decomposable member of $\Av(1342,1423)$ and $\alpha_2$ is an arbitrary (and
      possibly empty) member of $\Av^I(1342)$. Therefore we see that

      \begin{equation}
        \label{involutions:eqn:1342-2}
        g_\ominus = \pa{f(x^2)-f_\ominus(x^2)}(1+g).
      \end{equation}
        
      Lastly, we must enumerate $1342$-avoiding involutions which are inflations of
      simple permutations of length at least four. Any such simple permutation must
      have at least two right-to-left maxima and by simplicity every right-to-left
      maximum must have some entry both below it and to the left. Hence to avoid
      creating a copy of $1342$ or $1423$, we may only inflate right-to-left maxima
      by decreasing intervals. An entry which is a left-to-right minimum can be
      inflated by any permutation in the class $\Av(1342,1432)$. However, to ensure
      that the inflated permutation is an involution, we must inflate each fixed
      point by an involution. Additionally, if we inflate the entry with value
      $\sigma(i)$ by the permutation $\alpha$, we must make sure to inflate the
      entry with value $i$ by $\alpha^{-1}$.

      Consider $s^{(0)}(u,v)$, which is the generating function for
      simple involutions of length at least four which avoid $123$ and have zero
      fixed points. To inflate each right-to-left maximum by a decreasing
      permutation in a way that yields an involution, we substitute 
      
      $$ v^2 = \frac{x^2}{1-x^2} .$$

      This follows because if $\sigma(i)$ is a right-to-left maximum of the simple
      $123$-avoiding involution $\sigma$ then the entry with value $i$ will also be
      a right-to-left maximum, and we must substitute a permutation and its inverse
      into this pair of entries of $\sigma$. Because the class $\Av(1342,1423)$ is
      counted by the large Schr\"oder numbers, the inflations of the simple
      involutions of length at least four with zero fixed points are counted by

      \begin{equation}
        \label{involutions:eqn:1342-3}
        \eval{s^{(0)}(u,v)}_{u^2=f(x^2),\;v^2=x^2/(1-x^2)}.
      \end{equation}
        
      Recall that $s^{(1)}(u,v)$ counts only those simple involutions
      whose single fixed point is a right-to-left maximum. Since this fixed point
      must be inflated by a decreasing permutation, we count inflations of such
      permutations by 
      
      \begin{equation}
        \label{involutions:eqn:1342-4}
        \pa{\eval{\frac{s^{(1)}(u,v)}{v}}_{u^2=f(x^2),\;v^2=x^2/(1-x^2)}}
        \cdot\frac{x}{1-x}.
      \end{equation}

      To count those simple involutions whose single fixed point is a left-to-right
      minimum, we need only swap $u$ and $v$. Thus, inflations of these are counted
      by the generating function 

      \begin{equation}
        \label{involutions:eqn:1342-5}
        \pa{\eval{\frac{s^{(1)}(v,u)}{u}}_{u^2=f(x^2),\;v^2=x^2/(1-x^2)}}\cdot g.
        \end{equation}

      Finally, we must account for inflations of those simple involutions which
      contain exactly two fixed points, one of which is a right-to-left maximum
      while the other is a left-to-right minimum. These permutations are counted by

      \begin{equation} \label{involutions:eqn:1342-6}
        \pa{\eval{\frac{s^{(2)}(u,v)}{uv}}_{u^2=f(x^2),\;v^2=x^2/(1-x^2)}}\cdot\frac{gx}{1-x}.
      \end{equation}

      By summing the contributions of 
      \eqref{involutions:eqn:1342-1}--\eqref{involutions:eqn:1342-6} 
      and accounting for the single permutation of length $1$, one finds that
      
      $$ g(x) = \frac{x\pa{1-2x+x^2+\sqrt{1-6x^2+x^4}}}{2\pa{1-3x+x^2}}.$$
        % a := x*(1-2*x+x^2+sqrt(1-6*x^2 + x^4))/(2*(1-3*x+x^2));

      It can then be computed that the growth rate of involutions avoiding $1342$
      is $1$ plus the golden ratio, 
      
      $$ 1+\frac{1+\sqrt{5}}{2} \approx 2.62. $$

    \subsection{Involutions Avoiding 2341}
    \label{involutions:sub:2341}

      We turn our attention now to enumerating the $2341$-avoiding involutions.
      Note that each involution avoiding $2341$ must also avoid $2341^{-1} =
      4123$. We begin by examining the simple involutions which avoid these
      patterns. Note that, in this case, the simple permutations of the class
      $\Av(2341, 4123)$ are \emph{not} the same as the simples of $\Av(123)$.
      When we restrict to involutions, however, we find that the simples of
      $\Av^I(2341)$ are \emph{almost} the same as the simples of $\Av^I(123)$. 

      \begin{theorem}\label{involutions:thm:2341simples}
        The simple $2341$-avoiding involutions are precisely the union of set of
        $123$-avoiding simple involutions along with the permutation $5274163$. 
      \end{theorem}

      We delay the technical proof of this theorem to the end of this section. 

      Now that we know the simples, we need only determine the ways in which they
      can be inflated. 
      As in the previous section, we enumerate the $2341$-avoiding involutions by
      separately enumerating the sum decomposable permutations, the skew decomposable
      permutations, and the inflations of simple permutations of length at least
      four. Again we define $g$ to be the generating function for the set
      $\Av^I(2341)$ and $g_\oplus$ (resp., $g_\ominus$) the generating function for
      the sum (resp., skew) decomposable $2341$-avoiding involutions.

      In this case we see that $\Av^I(2341)$ is sum closed, so we have
      $$ g_\oplus = (g - g_\oplus)g. $$

      This then leads that 

      \begin{equation}
        \label{involutions:eqn:2341-1}
        g_\oplus = \frac{g^2}{1+g}.
      \end{equation}

      By Proposition~\ref{thm:subsdecomp-inv}, the skew decomposable permutations
      must have the form $321[\alpha_1,\alpha_2,\alpha_1^{-1}]$, where $\alpha_1$ is
      skew indecomposable and $\alpha_2$ is a (possibly empty) involution.
      Furthermore, to avoid the occurrence of a $2341$ or a $4123$ pattern, we must
      also have that $\alpha_1,\alpha_2\in \Av(123)$.

      Recall that the $123$-avoiding permutations are enumerated by the Catalan numbers, which
      have generating function 

      $$ c(x) = \frac{1-2x-\sqrt{1-4x}}{2x}.$$

      Since the class $\Av(123)$ is skew closed, when we denote the generating
      function for the skew decomposable $123$-avoiding permutation, it follows
      (as in the previous section) that

      $$ c-c_\ominus = \frac{c}{1+c} = x(c+1).$$

      As mentioned in the Section~\ref{involutions:sec:intro}, Simion and
      Schmidt~\cite{Simion1985} proved that
      $$ |\Av^I_n(123)|={n\choose \lfloor n/2\rfloor}. $$

      These terms are known as the central binomial coefficients,
      \OEIS{A001405}. These permutations thus have the generating function

      $$ \frac{1-4x^2-\sqrt{1-4x^2}}{4x^2-2x}. $$
    % ( 1-4*x^2-sqrt(1-4*x^2) ) / ( 4*x^2-2*x )

      Therefore, the generating function which counts our choices for the pair
      $(\alpha_1,\alpha_1^{-1})$ is $x^2(c(x^2)+1)$, and the generating function for
      all skew decomposable $2341$-avoiding involutions is

      \begin{equation}
      \label{involutions:eqn:2341-2}
        g_\ominus
        =
        \pa{x^2\pa{c(x^2)+1}}
        \cdot
        \pa{\frac{1-4x^2-\sqrt{1-4x^2}}{4x^2-2x}+1}
      \end{equation}

      Next, we consider inflations of the simple permutations in $\Av^I(123)$. In
      both cases, every entry of such a simple permutation can only be inflated by a
      decreasing permutation, as any inflation by a permutation with an increase
      would create a copy of $2341$ or $4123$. Thus inflations of the simple
      permutations counted by $s^{(0)}$ contribute

      \begin{equation}
      \label{involutions:eqn:2341-3}
        \eval{s^{(0)}(u,v)}_{u^2=v^2=x^2/(1-x^2)}.
      \end{equation}

      Inflations of the simple permutations counted by $s^{(1)}$ contribute

      \begin{equation}
        \label{involutions:eqn:2341-4}
        2\pa{\eval{\frac{s^{(1)}(u,v)}{v}}_{u^2=v^2=x^2/(1-x^2)}}
        \cdot
        \frac{x}{1-x}.
      \end{equation}

      Next, inflations of simple permutations counted by $s^{(2)}$ contribute
        \begin{equation}
        \label{involutions:eqn:2341-5}
        \pa{\eval{\frac{s^{(2)}(u,v)}{uv}}_{u^2=v^2=x^2/(1-x^2)}}
        \cdot
        \pa{\frac{x}{1-x}}^2.
      \end{equation}

      Lastly, we consider inflations of $5274163$. Because this permutation has
      three fixed points, the $2341$-avoiding involutions formed by inflations of
      $5274163$ are counted by

      \begin{equation}
      \label{involutions:eqn:2341-6}
      \pa{\frac{x^2}{1-x^2}}^2\pa{\frac{x}{1-x}}^3.
      \end{equation}

      By combining the contributions
      \eqref{involutions:eqn:2341-1}--\eqref{involutions:eqn:2341-6} and
      accounting for the single permutation of length $1$, it can be computed
      that $g$ has minimal polynomial Therefore, $b$ satisfies the functional
      equation
        
      $$ b = x + \frac{b^2}{1+b} +  x^2(c(x^2)+1)\pa{\frac{1 + x +
      xc(x^2)}{\sqrt{1-4x^2}}} + I(x). $$

      From this it follows that $b$ has minimal polynomial shown below. 

      % VV removed the following whyyyyyyy? 
      %Therefore, $b$ satisfies the functional equation
      %	\[b = x + \f{b^2}{1+b} +  x^2(c(x^2)+1)\pa{\f{1 + x + xc(x^2)}{\sqrt{1-4x^2}}} + I(x),\]
      %from which it follows that $b$ has minimal polynomial

      {\footnotesize
      \begin{align*}
      t^2g^2 + (48x^{16}-158x^{15}+101x^{14}+334x^{13}-627x^{12}+60x^{11}+801x^{10}-684x^9-231x^8\\
      \qquad\qquad\qquad\qquad +624x^7-221x^6-162x^5+151x^4-24x^3-17x^2+8x-1)tg \\
      \qquad\qquad +(18x^{15}-51x^{14}+16x^{13}+125x^{12}-169x^{11}-48x^{10}+256x^9-130x^8-131x^7\\
      \qquad\qquad\qquad\qquad +159x^6-11x^5-60x^4+28x^3+3x^2-5x+1)tx
      \end{align*}

        % g- x+ 1024*g^2*x^32- 7680*g^2*x^31+ 21632*g^2*x^30- 13936*g^2*x^29-
        % 71231*g^2*x^28+  190428*g^2*x^27- 97220*g^2*x^26- 353122*g^2*x^25+
        % 685926*g^2*x^24- 146902*g^2*x^23- 993199*g^2*x^22+  1234802*g^2*x^21+
        % 123607*g^2*x^20- 1540756*g^2*x^19+ 1157071*g^2*x^18+ 501148*g^2*x^17-
        % 1310214*g^2*x^16+ 554552*g^2*x^15+ 469462*g^2*x^14- 607214*g^2*x^13+
        % 128836*g^2*x^12+ 190294*g^2*x^11- 150337*g^2*x^10+ 14740*g^2*x^9+
        % 34105*g^2*x^8- 18600*g^2*x^7+ 1364*g^2*x^6+ 2202*g^2*x^5- 866*g^2*x^4+
        % 52*g^2*x^3+ 44*g^2*x^2- 12*g^2*x+ 1536*g*x^32- 10816*g*x^31+
        % 27616*g*x^30- 9398*g*x^29- 107199*g*x^28+ 231732*g*x^27- 46461*g*x^26-
        % 521665*g*x^25+ 763082*g*x^24+ 97894*g*x^23- 1356250*g*x^22+
        % 1189946*g*x^21+ 643818*g*x^20- 1895477*g*x^19+ 870475*g*x^18+
        % 1038245*g*x^17- 1430419*g*x^16+ 229155*g*x^15+ 749792*g*x^14-
        % 580611*g*x^13- 32017*g*x^12+ 265563*g*x^11- 125797*g*x^10- 23276*g*x^9+
        % 45198*g*x^8- 14423*g*x^7- 2666*g*x^6+ 3197*g*x^5- 768*g*x^4- 60*g*x^3+
        % 69*g*x^2- 14*g*x+ g^2+ 576*x^32- 3792*x^31+ 8666*x^30+ 25*x^29-
        % 38826*x^28+ 67719*x^27+ 10690*x^26- 180711*x^25+ 196002*x^24+
        % 113538*x^23- 427059*x^22+ 240376*x^21+ 319756*x^20- 523748*x^19+
        % 87377*x^18+ 391144*x^17- 331522*x^16- 55549*x^15+ 233985*x^14-
        % 104192*x^13- 53142*x^12+ 69894*x^11- 15154*x^10- 14264*x^9+ 10056*x^8-
        % 1013*x^7- 1371*x^6+ 607*x^5- 40*x^4- 37*x^3+ 11*x^2
    
      In the expression above, $t$ is defined as 
      \begin{align*}
        t &= 32x^{16}-120x^{15}+113x^{14}+206x^{13}-540x^{12}+223x^{11}+561x^{10}-725x^9\\
        &\qquad\qquad +26x^8+514x^7-326x^6-55x^5+141x^4-50x^3-4x^2+6x-1.
      \end{align*}
      }
            %32*x^16-120*x^15+113*x^14+206*x^13-540*x^12+223*x^11+
            %561*x^10-725*x^9+26*x^8+514*x^7-326*x^6-55*x^5+141*x^4-
            %50*x^3-4*x^2+6*x-1

      Note that though this minimal polynomial looks complicated, it is in fact
      quadratic in $g$, so it is not difficult to solve it explicitly. While the
      explicit solution is even more complicated than the minimal polynomial,
      this makes it relatively easy to compute the minimal polynomial for the
      growth rate of $\Av^I(2341,4123)$, which is

      \begin{align*}
        x^{16} - 6x^{15} + 4x^{14} + 50x^{13} - 141x^{12} + 55x^{11} + 326x^{10}
        - 514x^9 - 26x^8 + 725x^7\\
        \qquad - 561x^6 - 223x^5 + 540x^4 - 206x^3 - 113x^2 + 120x - 32.
        % x^(16) - 6*x^(15) + 4*x^(14) + 50*x^(13) - 141*x^(12) + 55*x^(11) +
        % 326*x^(10) - 514*x^9 - 26*x^8 + 725*x^7 - 561*x^6 - 223*x^5 + 540*x^4 -
        % 206*x^3 - 113*x^2 + 120*x - 32
      \end{align*}

      The growth rate itself is approximately $2.54$.

      % Actual generating function:
      % - 1/2*(48*x^16- 158*x^15+ 101*x^14+ 334*x^13- 627*x^12+ 60*x^11+ 801*x^10-
      % 684*x^9- 231*x^8+ 624*x^7- 221*x^6- 162*x^5+ 151*x^4- 24*x^3- 17*x^2+ 8*x- 1+
      % (- (4*x^2- 1)*(x+ 1)^8*(x- 1)^20)^(1/2))/(32*x^16- 120*x^15+ 113*x^14+
      % 206*x^13- 540*x^12+ 223*x^11+ 561*x^10- 725*x^9+ 26*x^8+ 514*x^7- 326*x^6-
      % 55*x^5+ 141*x^4- 50*x^3- 4*x^2+ 6*x- 1)

      We now return to the proof of Theorem~\ref{involutions:thm:2341simples}.
      The proof is rather technical, and relies on listing and eliminating a
      variety of cases. This was greatly assisted by Albert's
      PermLab~\cite{PermLab} software.

        \begin{proof}[Proof of Theorem~\ref{involutions:thm:2341simples}] 

          The proof of this theorem consists of the investigation of many cases
          relating to the placement of the fixed points in a $2341$-avoiding
          simple involution. Recall that such an involution must also avoid
          $2341^{-1} = 4123$. To better understand these permutations, we utilize
          \emph{permutation diagrams}, depicted in
          Figures~\ref{figure:6_1-group-1}, \ref{figure:6_1-group-2}, and
          \ref{figure:6_1-group-3}.
          Each of these diagrams consists of the plot of a permutation, together
          with a coloring of the cells. A cell is white if we are allowed to
          insert an entry without creating an occurrence of $2341$ or $4123$, and
          dark gray otherwise. A cell is light gray if we explicitly forbid any
          entries through the course of our arguments. The \emph{rectangular
          hull} of a set $S$ of entries is defined to be the smallest
          axis-parallel rectangle which contains all points of $S$. Finally, the
          \emph{inverse image} of a point $(x,y)$ is the point $(y,x)$,
          equivalent to the image of the point when reflecting across the line
          $y=x$. These tools will be useful in describing and understanding the
          various cases of this proof. 
            
          Let $\sg$ be a $2341$-avoiding simple involution, and claim that either
          $\sg$ avoids $123$ or $\sg = 5274163$. Suppose that $\sg$ contains at
          least one $123$ pattern. 
          Of all of the possible occurrences of $123$, 
          we focus on a single occurrence of this
          pattern, the one in which the $3$ is the topmost possible entry, the
          $1$ is the bottommost for the chosen $3$, and the $2$ is the rightmost
          for the chosen $1$ and $3$. It follows then that $\sg$ can be drawn on
          the diagram shown in Figure~\ref{bigproof1}. Note that, despite the
          apparent symmetry, these three entries are \emph{not necessarily} fixed
          points, because each white cell could be inflated by different numbers
          of entries. Thus, we must consider separate cases in which some
          combination of these entries lie on the diagonal.

          \textbf{Case 1:}
          For our first case, assume that each of these entries are in fact fixed
          points. Then, since $\sg$ is an involution, the cells labelled $A,B,C$
          must all be empty, since otherwise the plot would not be symmetric
          about the line passing through the diagonal. It follows then that $\sg$
          can be plotted on the diagram shown in Figure~\ref{bigproof2}. 
          We now claim that $\sg = 5274163$. 
          
          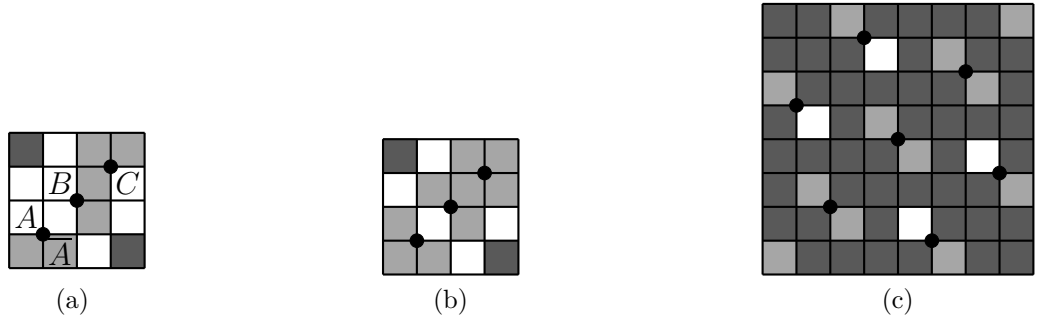
\begin{figure} \centering 
            %Tikz output
            \subfloat[]{\label{bigproof1}
            \begin{tikzpicture}[scale=.45]
              %Forbidden regions 
              \filldraw[light-gray](0,0) rectangle (1,1);
              \filldraw[dark-gray](0,3) rectangle (1,4); 
              \filldraw[light-gray](1,0) rectangle (2,1); 
              \filldraw[light-gray](2,1) rectangle (3,2);
              \filldraw[light-gray](2,2) rectangle (3,3);
              \filldraw[light-gray](2,3) rectangle (3,4); 
              \filldraw[dark-gray](3,0) rectangle (4,1); 
              \filldraw[light-gray](3,3) rectangle (4,4); %Points
              \draw[black, fill=black] (1,1) circle (0.2); 
              \draw[black, fill=black] (2,2) circle (0.2); 
              \draw[black, fill=black] (3,3) circle (0.2);
              \draw[thick](0,0)--(0,4); 
              \draw[thick](0,0)--(4,0);
              \draw[thick](1,0)--(1,4); 
              \draw[thick](0,1)--(4,1);
              \draw[thick](2,0)--(2,4); 
              \draw[thick](0,2)--(4,2);
              \draw[thick](3,0)--(3,4); 
              \draw[thick](0,3)--(4,3);
              \draw[thick](4,0)--(4,4); 
              \draw[thick](0,4)--(4,4); 
              \node at (.5,1.5) {$A$}; 
              \node at (1.5,2.5) {$B$}; 
              \node at (3.5,2.5) {$C$}; 
              \node at (1.5,.5) {$\overline{A}$}; 
            \end{tikzpicture}%
            } \hfill
            \subfloat[]{\label{bigproof2}
            \begin{tikzpicture}[scale=.45]
                %Forbidden regions 
              \filldraw[light-gray](0,0) rectangle (1,1);
              \filldraw[light-gray](0,1) rectangle (1,2);
              \filldraw[dark-gray](0,3) rectangle (1,4);
              \filldraw[light-gray](1,0) rectangle (2,1);
              \filldraw[light-gray](1,2) rectangle (2,3);
              \filldraw[light-gray](2,1) rectangle (3,2);
              \filldraw[light-gray](2,2) rectangle (3,3);
              \filldraw[light-gray](2,3) rectangle (3,4);
              \filldraw[dark-gray](3,0) rectangle (4,1);
              \filldraw[light-gray](3,2) rectangle (4,3);
              \filldraw[light-gray](3,3) rectangle (4,4); %Points 
              \draw[black, fill=black] (1,1) circle (0.2); 
              \draw[black, fill=black] (2,2) circle (0.2); 
              \draw[black, fill=black] (3,3) circle (0.2);
              %Gridlines 
              \draw[thick](0,0)--(0,4); 
              \draw[thick](0,0)--(4,0);
              \draw[thick](1,0)--(1,4); 
              \draw[thick](0,1)--(4,1);
              \draw[thick](2,0)--(2,4); 
              \draw[thick](0,2)--(4,2);
              \draw[thick](3,0)--(3,4); 
              \draw[thick](0,3)--(4,3);
              \draw[thick](4,0)--(4,4); 
              \draw[thick](0,4)--(4,4);
            \end{tikzpicture}
            } \hfill
            \subfloat[]{\label{bigproof3}
            \begin{tikzpicture}[scale=.45]
              %Forbidden regions 
              \filldraw[light-gray](0,0) rectangle (1,1);
              \filldraw[dark-gray](0,1) rectangle (1,2);
              \filldraw[dark-gray](0,2) rectangle (1,3);
              \filldraw[dark-gray](0,3) rectangle (1,4);
              \filldraw[dark-gray](0,4) rectangle (1,5);
              \filldraw[light-gray](0,5) rectangle (1,6);
              \filldraw[dark-gray](0,6) rectangle (1,7);
              \filldraw[dark-gray](0,7) rectangle (1,8);
              \filldraw[dark-gray](1,0) rectangle (2,1);
              \filldraw[dark-gray](1,1) rectangle (2,2);
              \filldraw[light-gray](1,2) rectangle (2,3);
              \filldraw[dark-gray](1,3) rectangle (2,4);
              \filldraw[dark-gray](1,5) rectangle (2,6);
              \filldraw[dark-gray](1,6) rectangle (2,7);
              \filldraw[dark-gray](1,7) rectangle (2,8);
              \filldraw[dark-gray](2,0) rectangle (3,1);
              \filldraw[light-gray](2,1) rectangle (3,2);
              \filldraw[dark-gray](2,2) rectangle (3,3);
              \filldraw[dark-gray](2,3) rectangle (3,4);
              \filldraw[dark-gray](2,4) rectangle (3,5);
              \filldraw[dark-gray](2,5) rectangle (3,6);
              \filldraw[dark-gray](2,6) rectangle (3,7);
              \filldraw[light-gray](2,7) rectangle (3,8);
              \filldraw[dark-gray](3,0) rectangle (4,1);
              \filldraw[dark-gray](3,1) rectangle (4,2);
              \filldraw[dark-gray](3,2) rectangle (4,3);
              \filldraw[dark-gray](3,3) rectangle (4,4);
              \filldraw[light-gray](3,4) rectangle (4,5);
              \filldraw[dark-gray](3,5) rectangle (4,6);
              \filldraw[dark-gray](3,7) rectangle (4,8);
              \filldraw[dark-gray](4,0) rectangle (5,1);
              \filldraw[dark-gray](4,2) rectangle (5,3);
              \filldraw[light-gray](4,3) rectangle (5,4);
              \filldraw[dark-gray](4,4) rectangle (5,5);
              \filldraw[dark-gray](4,5) rectangle (5,6);
              \filldraw[dark-gray](4,6) rectangle (5,7);
              \filldraw[dark-gray](4,7) rectangle (5,8);
              \filldraw[light-gray](5,0) rectangle (6,1);
              \filldraw[dark-gray](5,1) rectangle (6,2);
              \filldraw[dark-gray](5,2) rectangle (6,3);
              \filldraw[dark-gray](5,3) rectangle (6,4);
              \filldraw[dark-gray](5,4) rectangle (6,5);
              \filldraw[dark-gray](5,5) rectangle (6,6);
              \filldraw[light-gray](5,6) rectangle (6,7);
              \filldraw[dark-gray](5,7) rectangle (6,8);
              \filldraw[dark-gray](6,0) rectangle (7,1);
              \filldraw[dark-gray](6,1) rectangle (7,2);
              \filldraw[dark-gray](6,2) rectangle (7,3);
              \filldraw[dark-gray](6,4) rectangle (7,5);
              \filldraw[light-gray](6,5) rectangle (7,6);
              \filldraw[dark-gray](6,6) rectangle (7,7);
              \filldraw[dark-gray](6,7) rectangle (7,8);
              \filldraw[dark-gray](7,0) rectangle (8,1);
              \filldraw[dark-gray](7,1) rectangle (8,2);
              \filldraw[light-gray](7,2) rectangle (8,3);
              \filldraw[dark-gray](7,3) rectangle (8,4);
              \filldraw[dark-gray](7,4) rectangle (8,5);
              \filldraw[dark-gray](7,5) rectangle (8,6);
              \filldraw[dark-gray](7,6) rectangle (8,7);
              \filldraw[light-gray](7,7) rectangle (8,8); %Points 
              \draw[black, fill=black] (1,5) circle (0.2); 
              \draw[black, fill=black] (2,2) circle (0.2); 
              \draw[black, fill=black] (3,7) circle (0.2);
              \draw[black, fill=black] (4,4) circle (0.2); 
              \draw[black, fill=black] (5,1) circle (0.2); 
              \draw[black, fill=black] (6,6) circle (0.2); 
              \draw[black, fill=black] (7,3) circle (0.2);
              %Gridlines 
              \draw[thick](0,0)--(0,8); 
              \draw[thick](0,0)--(8,0);
              \draw[thick](1,0)--(1,8); 
              \draw[thick](0,1)--(8,1);
              \draw[thick](2,0)--(2,8); 
              \draw[thick](0,2)--(8,2);
              \draw[thick](3,0)--(3,8); 
              \draw[thick](0,3)--(8,3);
              \draw[thick](4,0)--(4,8); 
              \draw[thick](0,4)--(8,4);
              \draw[thick](5,0)--(5,8); 
              \draw[thick](0,5)--(8,5);
              \draw[thick](6,0)--(6,8); 
              \draw[thick](0,6)--(8,6);
              \draw[thick](7,0)--(7,8); 
              \draw[thick](0,7)--(8,7);
              \draw[thick](8,0)--(8,8); 
              \draw[thick](0,8)--(8,8);
            \end{tikzpicture}
            }

            \caption{Permutation diagrams referenced in the
                     proof of Theorem~\ref{involutions:thm:2341simples}.}
            \label{figure:6_1-group-1} 
          \end{figure}

          By simplicity, the rectangular hull of the leftmost two entries shown
          in Figure~\ref{bigproof2} must be split by an entry either in the white
          cell above it or in the white cell to its right. Since $\sg$ is an
          involution, it follows then that there are in fact splitting entries in
          both of these cells. Assume that the splitting entry in the cell above
          is the topmost possible entry and the one to the right is the rightmost
          possible. A similar argument applied to the rectangular hull of the
          rightmost two entries produces a permutation diagram depicted in
          Figure~\ref{bigproof3}. 

          We now claim that we can go no further. There are only four remaining
          white cells in Figure~\ref{bigproof3}, and no two of these cells shares
          a row or column. It follows then that by placing entries in any of
          these cells, we would be creating intervals which cannot be split by
          any other entry, thus violating simplicity. It follows then that the
          only simple $2341$-avoiding involution which contains an occurrence of
          $123$ in which each entry is a fixed point is the permutation
          $5274163$, as desired. 
          
          % =========================================================== %

          \textbf{Case 2:}
          Now suppose that the rightmost entry of our specified $123$ occurrence
          is not a fixed point. It therefore must lie either above or below the
          reflection line, i.e., it must be either above and to the left or
          below and to the right of its inverse image. Suppose first that it is
          below this line of reflection, and so its inverse image must lie above
          and to the left. There is only one candidate cell, the result is shown
          in Figure~\ref{bigproof4}.

          Note that, in a general involution, if two entries from an increase
          (resp., a decrease) then their inverse image also forms an increase
          (resp., a decrease). It follows then that the third entry from the left
          shown in Figure~\ref{bigproof4} (the $2$ of the original $123$ pattern)
          cannot lie above or on the reflection line, and so must lie below.
          Therefore its inverse image lies above. There is only one appropriate
          white cell in which this entry can lie, as shown in
          Figure~\ref{bigproof5}. If the leftmost entry in this figure were a
          fixed point, then the permutation would begin with its smallest entry,
          violating simplicity. This entry therefore lies below the reflection
          line, and has an inverse above and to its left. This leads to
          Figure~\ref{bigproof6}, but we see that the this leads to a non simple,
          and in fact sum decomposable, permutation, because the bottom-leftmost
          three by three rectangular hull cannot be split by any other entries.
          This case therefore leads to a contradiction, and can be eliminated.

          \begin{figure} \centering 
            \subfloat[]{\label{bigproof4}
            \begin{tikzpicture}[scale=.4]
              \filldraw[dark-gray](0,3) rectangle (1,4); 
              \filldraw[dark-gray](0,4) rectangle (1,5); 
              \filldraw[light-gray](1,0) rectangle (2,1);
              \filldraw[dark-gray](2,0) rectangle (3,1); 
              \filldraw[dark-gray](2,1) rectangle (3,2); 
              \filldraw[light-gray](3,1) rectangle (4,2);
              \filldraw[dark-gray](3,2) rectangle (4,3); 
              \filldraw[light-gray](3,3) rectangle (4,4); 
              \filldraw[light-gray](3,4) rectangle (4,5);
              \filldraw[dark-gray](4,0) rectangle (5,1); 
              \filldraw[dark-gray](4,3) rectangle (5,4); 
              \filldraw[light-gray](4,4) rectangle (5,5); %Points
              \draw[black, fill=black] (1,1) circle (0.2); 
              \draw[black, fill=black] (2,4) circle (0.2); 
              \draw[black, fill=black] (3,2) circle (0.2);
              \draw[black, fill=black] (4,3) circle (0.2); %Gridlines
              \draw[thick](0,0)--(0,5); 
              \draw[thick](0,0)--(5,0);
              \draw[thick](1,0)--(1,5); 
              \draw[thick](0,1)--(5,1);
              \draw[thick](2,0)--(2,5); 
              \draw[thick](0,2)--(5,2);
              \draw[thick](3,0)--(3,5); 
              \draw[thick](0,3)--(5,3);
              \draw[thick](4,0)--(4,5); 
              \draw[thick](0,4)--(5,4);
              \draw[thick](5,0)--(5,5); 
              \draw[thick](0,5)--(5,5); 
            \end{tikzpicture}
            } \hfill
            \subfloat[]{\label{bigproof5}
            \begin{tikzpicture}[scale=.4]
              %Forbidden regions
              \filldraw[light-gray](0,0) rectangle (1,1);
              \filldraw[dark-gray](0,2) rectangle (1,3);
              \filldraw[dark-gray](0,3) rectangle (1,4);
              \filldraw[dark-gray](0,4) rectangle (1,5);
              \filldraw[dark-gray](0,5) rectangle (1,6);
              \filldraw[light-gray](1,0) rectangle (2,1);
              \filldraw[dark-gray](1,2) rectangle (2,3);
              \filldraw[dark-gray](1,3) rectangle (2,4);
              \filldraw[dark-gray](2,0) rectangle (3,1);
              \filldraw[dark-gray](2,1) rectangle (3,2);
              \filldraw[dark-gray](2,4) rectangle (3,5);
              \filldraw[dark-gray](3,0) rectangle (4,1);
              \filldraw[dark-gray](3,1) rectangle (4,2);
              \filldraw[dark-gray](3,5) rectangle (4,6);
              \filldraw[dark-gray](4,0) rectangle (5,1);
              \filldraw[light-gray](4,1) rectangle (5,2);
              \filldraw[dark-gray](4,2) rectangle (5,3);
              \filldraw[light-gray](4,3) rectangle (5,4);
              \filldraw[light-gray](4,4) rectangle (5,5);
              \filldraw[dark-gray](4,5) rectangle (5,6);
              \filldraw[dark-gray](5,0) rectangle (6,1);
              \filldraw[dark-gray](5,3) rectangle (6,4);
              \filldraw[dark-gray](5,4) rectangle (6,5);
              \filldraw[light-gray](5,5) rectangle (6,6); %Points 
              \draw[black, fill=black] (1,1) circle (0.2); 
              \draw[black, fill=black] (2,4) circle (0.2); 
              \draw[black, fill=black] (3,5) circle (0.2);
              \draw[black, fill=black] (4,2) circle (0.2); 
              \draw[black, fill=black] (5,3) circle (0.2); %Gridlines
              \draw[thick](0,0)--(0,6); 
              \draw[thick](0,0)--(6,0);
              \draw[thick](1,0)--(1,6); 
              \draw[thick](0,1)--(6,1);
              \draw[thick](2,0)--(2,6); 
              \draw[thick](0,2)--(6,2);
              \draw[thick](3,0)--(3,6); 
              \draw[thick](0,3)--(6,3);
              \draw[thick](4,0)--(4,6); 
              \draw[thick](0,4)--(6,4);
              \draw[thick](5,0)--(5,6); 
              \draw[thick](0,5)--(6,5);
              \draw[thick](6,0)--(6,6); 
              \draw[thick](0,6)--(6,6);
            \end{tikzpicture}
            } \hfill
            \subfloat[]{\label{bigproof6}
            \begin{tikzpicture}[scale=.4]
                %Forbidden regions 
              \filldraw[light-gray](0,0) rectangle (1,1);
              \filldraw[dark-gray](0,3) rectangle (1,4);
              \filldraw[dark-gray](0,4) rectangle (1,5);
              \filldraw[dark-gray](0,5) rectangle (1,6);
              \filldraw[dark-gray](0,6) rectangle (1,7);
              \filldraw[light-gray](1,0) rectangle (2,1);
              \filldraw[dark-gray](1,3) rectangle (2,4);
              \filldraw[dark-gray](1,4) rectangle (2,5);
              \filldraw[dark-gray](1,5) rectangle (2,6);
              \filldraw[dark-gray](1,6) rectangle (2,7);
              \filldraw[light-gray](2,0) rectangle (3,1);
              \filldraw[dark-gray](2,3) rectangle (3,4);
              \filldraw[dark-gray](2,4) rectangle (3,5);
              \filldraw[dark-gray](3,0) rectangle (4,1);
              \filldraw[dark-gray](3,1) rectangle (4,2);
              \filldraw[dark-gray](3,2) rectangle (4,3);
              \filldraw[dark-gray](3,5) rectangle (4,6);
              \filldraw[dark-gray](4,0) rectangle (5,1);
              \filldraw[dark-gray](4,1) rectangle (5,2);
              \filldraw[dark-gray](4,2) rectangle (5,3);
              \filldraw[dark-gray](4,6) rectangle (5,7);
              \filldraw[dark-gray](5,0) rectangle (6,1);
              \filldraw[dark-gray](5,1) rectangle (6,2);
              \filldraw[light-gray](5,2) rectangle (6,3);
              \filldraw[dark-gray](5,3) rectangle (6,4);
              \filldraw[light-gray](5,4) rectangle (6,5);
              \filldraw[light-gray](5,5) rectangle (6,6);
              \filldraw[dark-gray](5,6) rectangle (6,7);
              \filldraw[dark-gray](6,0) rectangle (7,1);
              \filldraw[dark-gray](6,1) rectangle (7,2);
              \filldraw[dark-gray](6,4) rectangle (7,5);
              \filldraw[dark-gray](6,5) rectangle (7,6);
              \filldraw[light-gray](6,6) rectangle (7,7); %Points 
              \draw[black, fill=black] (1,2) circle (0.2); 
              \draw[black, fill=black] (2,1) circle (0.2); 
              \draw[black, fill=black] (3,5) circle (0.2);
              \draw[black, fill=black] (4,6) circle (0.2); 
              \draw[black, fill=black] (5,3) circle (0.2); 
              \draw[black, fill=black] (6,4) circle (0.2); %Gridlines 
              \draw[thick](0,0)--(0,7);
              \draw[thick](0,0)--(7,0); 
              \draw[thick](1,0)--(1,7);
              \draw[thick](0,1)--(7,1); 
              \draw[thick](2,0)--(2,7);
              \draw[thick](0,2)--(7,2); 
              \draw[thick](3,0)--(3,7);
              \draw[thick](0,3)--(7,3); 
              \draw[thick](4,0)--(4,7);
              \draw[thick](0,4)--(7,4); 
              \draw[thick](5,0)--(5,7);
              \draw[thick](0,5)--(7,5); 
              \draw[thick](6,0)--(6,7);
              \draw[thick](0,6)--(7,6); 
              \draw[thick](7,0)--(7,7);
              \draw[thick](0,7)--(7,7); 
            \end{tikzpicture}
            } \hfill
            \subfloat[]{\label{bigproof7}
            \begin{tikzpicture}[scale=.4]
                %Forbidden regions 
              \filldraw[light-gray](0,0) rectangle (1,1);
              \filldraw[dark-gray](0,3) rectangle (1,4);
              \filldraw[dark-gray](0,4) rectangle (1,5);
              \filldraw[light-gray](1,0) rectangle (2,1);
              \filldraw[light-gray](2,1) rectangle (3,2);
              \filldraw[light-gray](2,2) rectangle (3,3);
              \filldraw[light-gray](2,3) rectangle (3,4);
              \filldraw[light-gray](2,4) rectangle (3,5);
              \filldraw[dark-gray](3,0) rectangle (4,1);
              \filldraw[light-gray](3,2) rectangle (4,3);
              \filldraw[light-gray](3,4) rectangle (4,5);
              \filldraw[dark-gray](4,0) rectangle (5,1);
              \filldraw[light-gray](4,3) rectangle (5,4);
              \filldraw[light-gray](4,4) rectangle (5,5); %Points 
              \draw[black, fill=black] (1,1) circle (0.2); 
              \draw[black, fill=black] (2,2) circle (0.2); 
              \draw[black, fill=black] (3,4) circle (0.2);
              \draw[black, fill=black] (4,3) circle (0.2); %Gridlines
              \draw[thick](0,0)--(0,5); 
              \draw[thick](0,0)--(5,0);
              \draw[thick](1,0)--(1,5); 
              \draw[thick](0,1)--(5,1);
              \draw[thick](2,0)--(2,5); 
              \draw[thick](0,2)--(5,2);
              \draw[thick](3,0)--(3,5); 
              \draw[thick](0,3)--(5,3);
              \draw[thick](4,0)--(4,5); 
              \draw[thick](0,4)--(5,4);
              \draw[thick](5,0)--(5,5); 
              \draw[thick](0,5)--(5,5);
            \end{tikzpicture}
            } \hfill
            \subfloat[]{\label{bigproof8}
            \begin{tikzpicture}[scale=.4]
              %Forbidden regions 
              \filldraw[light-gray](0,0) rectangle (1,1);
              \filldraw[dark-gray](0,2) rectangle (1,3);
              \filldraw[dark-gray](0,3) rectangle (1,4);
              \filldraw[dark-gray](0,4) rectangle (1,5);
              \filldraw[dark-gray](0,5) rectangle (1,6);
              \filldraw[dark-gray](0,6) rectangle (1,7);
              \filldraw[light-gray](1,0) rectangle (2,1);
              \filldraw[dark-gray](1,2) rectangle (2,3);
              \filldraw[dark-gray](1,3) rectangle (2,4);
              \filldraw[dark-gray](1,4) rectangle (2,5);
              \filldraw[dark-gray](2,0) rectangle (3,1);
              \filldraw[dark-gray](2,1) rectangle (3,2);
              \filldraw[dark-gray](2,2) rectangle (3,3);
              \filldraw[dark-gray](2,5) rectangle (3,6);
              \filldraw[dark-gray](3,0) rectangle (4,1);
              \filldraw[dark-gray](3,1) rectangle (4,2);
              \filldraw[light-gray](3,2) rectangle (4,3);
              \filldraw[dark-gray](3,3) rectangle (4,4);
              \filldraw[dark-gray](3,4) rectangle (4,5);
              \filldraw[dark-gray](3,5) rectangle (4,6);
              \filldraw[light-gray](3,6) rectangle (4,7);
              \filldraw[dark-gray](4,0) rectangle (5,1);
              \filldraw[dark-gray](4,1) rectangle (5,2);
              \filldraw[dark-gray](4,3) rectangle (5,4);
              \filldraw[dark-gray](4,6) rectangle (5,7);
              \filldraw[dark-gray](5,0) rectangle (6,1);
              \filldraw[dark-gray](5,2) rectangle (6,3);
              \filldraw[dark-gray](5,3) rectangle (6,4);
              \filldraw[dark-gray](5,6) rectangle (6,7);
              \filldraw[dark-gray](6,0) rectangle (7,1);
              \filldraw[dark-gray](6,4) rectangle (7,5);
              \filldraw[dark-gray](6,5) rectangle (7,6);
              \filldraw[light-gray](6,6) rectangle (7,7); %Points 
              \draw[black, fill=black] (1,1) circle (0.2); 
              \draw[black, fill=black] (2,5) circle (0.2); 
              \draw[black, fill=black] (3,3) circle (0.2);
              \draw[black, fill=black] (4,6) circle (0.2); 
              \draw[black, fill=black] (5,2) circle (0.2); 
              \draw[black, fill=black] (6,4) circle (0.2); %Gridlines 
              \draw[thick](0,0)--(0,7);
              \draw[thick](0,0)--(7,0); 
              \draw[thick](1,0)--(1,7);
              \draw[thick](0,1)--(7,1); 
              \draw[thick](2,0)--(2,7);
              \draw[thick](0,2)--(7,2); 
              \draw[thick](3,0)--(3,7);
              \draw[thick](0,3)--(7,3); 
              \draw[thick](4,0)--(4,7);
              \draw[thick](0,4)--(7,4); 
              \draw[thick](5,0)--(5,7);
              \draw[thick](0,5)--(7,5); 
              \draw[thick](6,0)--(6,7);
              \draw[thick](0,6)--(7,6); 
              \draw[thick](7,0)--(7,7);
              \draw[thick](0,7)--(7,7); 
            \end{tikzpicture}
            } 

            \caption{Permutation diagrams referenced in the proof of
            Theorem~\ref{involutions:thm:2341simples}.}
            \label{figure:6_1-group-2} 
          \end{figure}
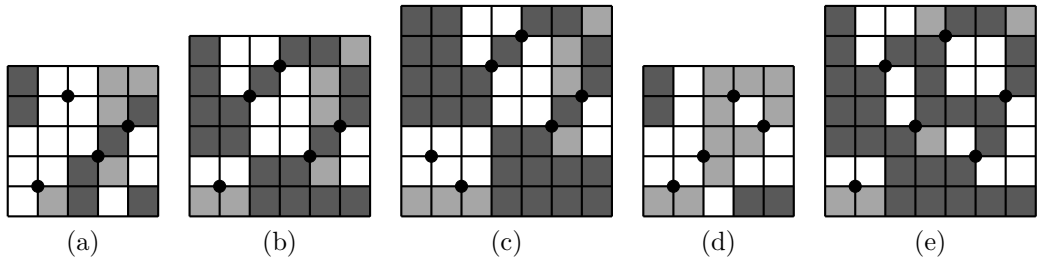

          Suppose now that instead of lying below the reflection line, the $3$ of
          our $123$ pattern shown in Figure~\ref{bigproof1} lies above, and so its
          inverse image is in a cell below and to the right, of which there are
          two. If the inverse image, however, is in the lower of these two then
          by an argument analogous to the paragraph above we reach a violation of
          simplicity. Therefore the inverse image of the rightmost entry must lie
          in the cell directly below and to its right. The fact that $\sg$ is an
          involution allows us to forbid the placing of entries into cells where
          the inverse image would create a forbidden pattern, leading to
          Figure~\ref{bigproof7}. Now the rectangular hull of the rightmost two
          entries must be split to preserve simplicity, and in fact must be split
          below and to the left to preserve involutionicity, leading to
          Figure~\ref{bigproof8}. Our situation is now analagous to that shown in
          Figure~\ref{bigproof6}, in that any placement of entries will lead to a
          sum decomposable (and hence non simple) permutation. Therefore this
          case (in which the last entry of the original $123$ is not a fixed
          point), can be discarded. 

          \textbf{Case 3:}
          Finally, we consider the case where the $3$ of the $123$ is a fixed
          point, but some other entry is not. Suppose first that the middle entry
          of Figure~\ref{bigproof1} lies above the reflection line. We are then
          forced into a situation identical to that shown in
          Figure~\ref{bigproof8} except rotated by $180$ degrees, leading to a
          contradiction. Assuming that the middle entry is above the reflection
          line leads, and recalling that the inverse image of two increasing
          points are themselves increasing, leads to Figure~\ref{bigproof9}.

          \begin{figure} \centering
            \subfloat[]{\label{bigproof9}
            \begin{tikzpicture}[scale=.4]
              %Forbidden regions 
              \filldraw[light-gray](0,0) rectangle (1,1);
              \filldraw[dark-gray](0,4) rectangle (1,5);
              \filldraw[light-gray](1,0) rectangle (2,1);
              \filldraw[light-gray](2,0) rectangle (3,1);
              \filldraw[light-gray](2,3) rectangle (3,4);
              \filldraw[light-gray](3,1) rectangle (4,2);
              \filldraw[light-gray](3,2) rectangle (4,3);
              \filldraw[light-gray](3,3) rectangle (4,4);
              \filldraw[light-gray](3,4) rectangle (4,5);
              \filldraw[dark-gray](4,0) rectangle (5,1);
              \filldraw[light-gray](4,3) rectangle (5,4);
              \filldraw[light-gray](4,4) rectangle (5,5);
              \filldraw[light-gray](1,3) rectangle (2,4); %Points 
              \draw[black, fill=black] (1,1) circle (0.2); 
              \draw[black, fill=black] (2,3) circle (0.2); 
              \draw[black, fill=black] (3,2) circle (0.2);
              \draw[black, fill=black] (4,4) circle (0.2); %Gridlines
              \draw[thick](0,0)--(0,5); 
              \draw[thick](0,0)--(5,0);
              \draw[thick](1,0)--(1,5); 
              \draw[thick](0,1)--(5,1);
              \draw[thick](2,0)--(2,5); 
              \draw[thick](0,2)--(5,2);
              \draw[thick](3,0)--(3,5); 
              \draw[thick](0,3)--(5,3);
              \draw[thick](4,0)--(4,5); 
              \draw[thick](0,4)--(5,4);
              \draw[thick](5,0)--(5,5); 
              \draw[thick](0,5)--(5,5);
            \end{tikzpicture}
            } \hfill
            \subfloat[]{\label{bigproof10}
            \begin{tikzpicture}[scale=.4]
              %Forbidden regions 
              \filldraw[light-gray](0,0) rectangle (1,1);
              \filldraw[light-gray](0,1) rectangle (1,2);
              \filldraw[light-gray](0,2) rectangle (1,3);
              \filldraw[dark-gray](0,4) rectangle (1,5);
              \filldraw[light-gray](1,0) rectangle (2,1);
              \filldraw[light-gray](1,3) rectangle (2,4);
              \filldraw[light-gray](2,0) rectangle (3,1);
              \filldraw[light-gray](2,3) rectangle (3,4);
              \filldraw[light-gray](3,1) rectangle (4,2);
              \filldraw[light-gray](3,2) rectangle (4,3);
              \filldraw[light-gray](3,3) rectangle (4,4);
              \filldraw[light-gray](3,4) rectangle (4,5);
              \filldraw[dark-gray](4,0) rectangle (5,1);
              \filldraw[light-gray](4,3) rectangle (5,4);
              \filldraw[light-gray](4,4) rectangle (5,5); %Points 
              \draw[black, fill=black] (1,1) circle (0.2); 
              \draw[black, fill=black] (2,3) circle (0.2); 
              \draw[black, fill=black] (3,2) circle (0.2);
              \draw[black, fill=black] (4,4) circle (0.2); %Gridlines
              \draw[thick](0,0)--(0,5); 
              \draw[thick](0,0)--(5,0);
              \draw[thick](1,0)--(1,5); 
              \draw[thick](0,1)--(5,1);
              \draw[thick](2,0)--(2,5); 
              \draw[thick](0,2)--(5,2);
              \draw[thick](3,0)--(3,5); 
              \draw[thick](0,3)--(5,3);
              \draw[thick](4,0)--(4,5); 
              \draw[thick](0,4)--(5,4);
              \draw[thick](5,0)--(5,5); 
              \draw[thick](0,5)--(5,5);
            \end{tikzpicture}
            } \hfill
            \subfloat[]{\label{bigproof11}
            \begin{tikzpicture}[scale=.4]
                %Forbidden regions 
              \filldraw[light-gray](0,0) rectangle (1,1);
              \filldraw[dark-gray](0,1) rectangle (1,2);
              \filldraw[light-gray](0,2) rectangle (1,3);
              \filldraw[light-gray](0,3) rectangle (1,4);
              \filldraw[dark-gray](0,6) rectangle (1,7);
              \filldraw[dark-gray](1,0) rectangle (2,1);
              \filldraw[dark-gray](1,1) rectangle (2,2);
              \filldraw[light-gray](1,2) rectangle (2,3);
              \filldraw[light-gray](1,3) rectangle (2,4);
              \filldraw[dark-gray](1,6) rectangle (2,7);
              \filldraw[light-gray](2,0) rectangle (3,1);
              \filldraw[light-gray](2,1) rectangle (3,2);
              \filldraw[dark-gray](2,2) rectangle (3,3);
              \filldraw[dark-gray](2,3) rectangle (3,4);
              \filldraw[light-gray](2,4) rectangle (3,5);
              \filldraw[light-gray](2,5) rectangle (3,6);
              \filldraw[light-gray](3,0) rectangle (4,1);
              \filldraw[light-gray](3,1) rectangle (4,2);
              \filldraw[dark-gray](3,2) rectangle (4,3);
              \filldraw[dark-gray](3,4) rectangle (4,5);
              \filldraw[dark-gray](3,5) rectangle (4,6);
              \filldraw[dark-gray](3,6) rectangle (4,7);
              \filldraw[light-gray](4,2) rectangle (5,3);
              \filldraw[dark-gray](4,3) rectangle (5,4);
              \filldraw[dark-gray](4,4) rectangle (5,5);
              \filldraw[dark-gray](4,5) rectangle (5,6);
              \filldraw[dark-gray](4,6) rectangle (5,7);
              \filldraw[light-gray](5,2) rectangle (6,3);
              \filldraw[dark-gray](5,3) rectangle (6,4);
              \filldraw[dark-gray](5,4) rectangle (6,5);
              \filldraw[light-gray](5,5) rectangle (6,6);
              \filldraw[light-gray](5,6) rectangle (6,7);
              \filldraw[dark-gray](6,0) rectangle (7,1);
              \filldraw[dark-gray](6,1) rectangle (7,2);
              \filldraw[dark-gray](6,3) rectangle (7,4);
              \filldraw[dark-gray](6,4) rectangle (7,5);
              \filldraw[light-gray](6,5) rectangle (7,6);
              \filldraw[light-gray](6,6) rectangle (7,7); %Points 
              \draw[black, fill=black] (1,5) circle (0.2); 
              \draw[black, fill=black] (2,2) circle (0.2); 
              \draw[black, fill=black] (3,4) circle (0.2);
              \draw[black, fill=black] (4,3) circle (0.2); 
              \draw[black, fill=black] (5,1) circle (0.2); 
              \draw[black, fill=black] (6,6) circle (0.2); %Gridlines 
              \draw[thick](0,0)--(0,7);
              \draw[thick](0,0)--(7,0); 
              \draw[thick](1,0)--(1,7);
              \draw[thick](0,1)--(7,1); 
              \draw[thick](2,0)--(2,7);
              \draw[thick](0,2)--(7,2); 
              \draw[thick](3,0)--(3,7);
              \draw[thick](0,3)--(7,3); 
              \draw[thick](4,0)--(4,7);
              \draw[thick](0,4)--(7,4); 
              \draw[thick](5,0)--(5,7);
              \draw[thick](0,5)--(7,5); 
              \draw[thick](6,0)--(6,7);
              \draw[thick](0,6)--(7,6); 
              \draw[thick](7,0)--(7,7);
              \draw[thick](0,7)--(7,7); 
            \end{tikzpicture}
            } \hfill
            \subfloat[]{\label{bigproof12}
            \begin{tikzpicture}[scale=.4]
              \filldraw[light-gray](0,0) rectangle (1,1);
              \filldraw[light-gray](0,1) rectangle (1,2);
              \filldraw[light-gray](0,2) rectangle (1,3);
              \filldraw[light-gray](0,3) rectangle (1,4);
              \filldraw[dark-gray](0,5) rectangle (1,6);
              \filldraw[light-gray](1,0) rectangle (2,1);
              \filldraw[light-gray](1,4) rectangle (2,5);
              \filldraw[dark-gray](1,5) rectangle (2,6);
              \filldraw[light-gray](2,0) rectangle (3,1);
              \filldraw[light-gray](2,4) rectangle (3,5);
              \filldraw[light-gray](3,0) rectangle (4,1);
              \filldraw[light-gray](3,4) rectangle (4,5);
              \filldraw[light-gray](4,1) rectangle (5,2);
              \filldraw[light-gray](4,2) rectangle (5,3);
              \filldraw[light-gray](4,3) rectangle (5,4);
              \filldraw[light-gray](4,4) rectangle (5,5);
              \filldraw[light-gray](4,5) rectangle (5,6);
              \filldraw[dark-gray](5,0) rectangle (6,1);
              \filldraw[dark-gray](5,1) rectangle (6,2);
              \filldraw[light-gray](5,4) rectangle (6,5);
              \filldraw[light-gray](5,5) rectangle (6,6); %Points 
              \draw[black, fill=black] (1,2) circle (0.2); 
              \draw[black, fill=black] (2,1) circle (0.2); 
              \draw[black, fill=black] (3,4) circle (0.2);
              \draw[black, fill=black] (4,3) circle (0.2); 
              \draw[black, fill=black] (5,5) circle (0.2); %Gridlines
              \draw[thick](0,0)--(0,6); 
              \draw[thick](0,0)--(6,0);
              \draw[thick](1,0)--(1,6); 
              \draw[thick](0,1)--(6,1);
              \draw[thick](2,0)--(2,6); 
              \draw[thick](0,2)--(6,2);
              \draw[thick](3,0)--(3,6); 
              \draw[thick](0,3)--(6,3);
              \draw[thick](4,0)--(4,6); 
              \draw[thick](0,4)--(6,4);
              \draw[thick](5,0)--(5,6); 
              \draw[thick](0,5)--(6,5);
              \draw[thick](6,0)--(6,6); 
              \draw[thick](0,6)--(6,6);
            \end{tikzpicture}
            }

            \caption{Permutation diagrams referenced in the proof of
            Theorem~\ref{involutions:thm:2341simples}.}
            \label{figure:6_1-group-3} 
          \end{figure}
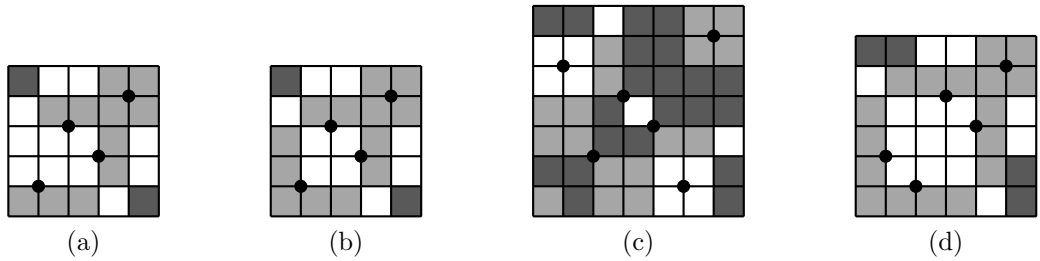

          First assume that the leftmost entry shown in Figure~\ref{bigproof9} is a
          fixed point, leading to Figure~\ref{bigproof10}. Simplicity then
          requires that there be an entry in the bottommost white cell whose
          inverse image is in the leftmost white cell, yielding
          Figure~\ref{bigproof11}. The center of this diagram, however, contains
          an interval which cannot be split, contradicting our assumption that
          the leftmost entry of Figure~\ref{bigproof9} is a fixed point.
          Letting this entry lie above the reflection line leads to a
          contradiction analogous to Figure~\ref{bigproof6}, and so let this
          entry lie below the line, with its inverse image above and to the left.
          Inspecting the various cases shows that this inverse image must lie in
          the cell immediately above and to the left, producing
          Figure~\ref{bigproof12}. The rectangular hull of the leftmost two
          entries can be split in two ways, but one of them leads to a sum
          decomposable permutation and the other leads to a non involution.

          Our final remaining case is when the middle entry of
          Figure~\ref{bigproof1} is a fixed point. Using similar methods to those
          presented above, we find that the that the leftmost entry must also be
          a fixed point. However, this case has already been investigated.

          It therefore follows that there is precisely one simple $2341$-avoiding
          involution.  Since every $123$-avoiding simple involution must also
          avoid $2341$, it follows that the set of all simple $2341$-avoiding
          involutions is equal to the set of simple $123$-avoiding involutions
          together with the permutation $5274163$. 
          
        \end{proof}

\cleardoublepage
\typeout{******************}
\typeout{**  Chapter 4   **}
\typeout{******************}

  \chapter{Polynomial Classes and Genomics}
  \label{chap:polyclass}

    This chapter examines the so called \emph{polynomial classes}, those
    permutation classes whose enumeration is given by a polynomial for large
    enough sizes. Much research in the area of permutation classes focuses on
    characterizing exponential growth rates, with a particular
    focus on the principally based classes. Considerably less attention has been
    paid to the small permutation classes~\cite{Vatter2011,Vatter2010} of which the
    polynomial classes, having subexponential growth, are an example. 

    These classes have recently found biological applications to the field of
    genomics. Evolution and mutation of organisms can be modelled as a
    rearrangement of a sequence of genes, and permutations have recently been
    applied to model these rearrangements~\cite{GenomeBook}. The physical
    mechanics of genome rearrangement have led to a variety of operations on
    permutations, and the theory of geometric grid classes~\cite{GridClasses}
    provides a geometric foundation from which to study these various operations.
    The polynomial classes are a subset of these grid classes, and arise when
    modelling the evolutionary distance. 

    Polynomial classes can characterized in a number of ways, but determining the
    actual polynomial which enumerates such a class can be computationally
    difficult. While there are several established methods for enumerating
    permutation classes, many of these are inefficient and none take advantage of
    the inherent structure in these classes.  In this chapter, we introduce an
    algorithm which quickly and efficiently enumerates a polynomial class from a
    structural description of the class.  This allows for an extension of
    existing genomic data, as well as a framework for further investigation. This
    chapter is based in part on~\cite{me-polyclass}, and the algorithm,
    implemented in Python, is freely available online~\cite{polyclass-algo}.

  % =========================================================================== %
  \section{Class Structure}
  \label{polyclass:sec:gridclasses}
  % =========================================================================== %

    \begin{definition} \label{polyclass:def:polyclass}
      A permutation class $\C$ is a \emph{polynomial class} if and only if the
      function $p(n) = |\C_n|$ is given by a polynomial for large enough $n$. 
    \end{definition} 

    It is not obvious that this definition gives way to a strict geometric
    description, as we shall soon see. Geometric grid classes provides a range of
    tools for analyzing the geometric properties of permutation class structure,
    and has produced new enumerative techniques for classes. To describe
    polynomial classes, however, we don't need the full machinery of geometric
    grid classes; these classes can be defined entirely using inflations
    (Definition~\ref{prelim:def:inflation}). 

    Note first that the polynomial classes fall under the purview of several
    established approaches, which could \emph{theoretically} be used to enumerate
    the classes~\cite{GridClasses, RegInsEnc, Atkinson2005, Linton2005,
    Brignall2008}. However, each of these approaches has its own drawbacks, and
    none provides an enumeration directly from a structural description of the
    class. Further, the work presented here illuminates some of the preliminary
    obstacles preventing a similar algorithmic approach to geometric grid
    classes.

    \subsection{Peg Permutations}

      Polynomial classes can be viewed by considering a set of restricted
      inflations of a finite set of permutations. In order to properly analyze
      these inflations, we introduce an additional structure on permutations which
      will be used to specify which inflations are allowed. 

      \begin{definition} \label{polyclass:def:pegperm}
        A \emph{peg permutation} $\tro$ is a permutation $\ro = \ro_1 \ro_2 \dots
        \ro_n$ in which each entry is decorated with either a $+$, $-$, or
        $\bullet$. The \emph{length} of a peg permutation $\tro$ is just the
        length of the underlying permutation $\ro$. 
      \end{definition}
      
      For example, $ \tro = \pl3 \mn1 \dt2 \pl4$ is a peg permutation of length $4$,
      and there are $3^n n!$ peg permutations of length $n$. 
      We denote peg permutations with a tilde, while the underlying permutation
      (with decoration removed) is written without.

      We allow peg permutations to be inflated with \emph{monotone} intervals. The
      entries marked with a $+$ (resp. $-$) can be inflated with ascending (resp.
      decreasing) runs. Entries marked with a $\blt$ can be inflated with a single
      entry. Note that we go against tradition and allow empty inflations. It
      follows then that such an inflation can be described simply as a peg
      permutation together with a sequence of integers which represent the number of
      elements by which to inflate each entry. We formalize this below. 

      \begin{definition} \label{polyclass:def:inflation}
        Let $\tro = \tro_1 \tro_2 \dots \tro_n$ be a peg permutation of length
        $n$, and $\vec i = (i_1, i_2, \dots i_n)$. Then let $\tro(I)$ be the
        permutation obtained by inflating entry $\tro_k$ by an interval of size
        $i_k$ according to the decoration of $\tro_k$: an ascending run if the
        decoration is a $+$, a descending run if it is a $-$, and a single entry if
        a $\blt$. If $\tro_k$ has a dot, then $i_k$ must be $0$ or $1$, otherwise
        $i_k \in \Zgeq$. 
      \end{definition}

      Recall, for example, the class $\Av(123, 231)$ examined in
      Section~\ref{prelim:sec:av123+231}. The decomposition of this class was
      shown in Figure~\ref{prelim:fig:polygrid}, and can be described as
      inflations of the peg permutation $\pl 3 \pl 1 \pl 2$. 

      Like many definitions in this dissertation, this one is best illustrated with
      a graphic example. Figure~\ref{polyclass:fig:inflation} shows a peg
      permutation being inflated and then standardized into a permutation. The
      following definition and theorem provide our desired characterization of
      polynomial classes. 

      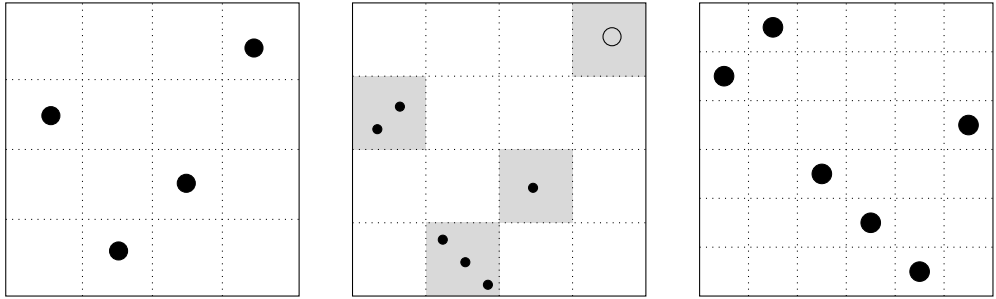
\begin{figure}[t] \centering
          % \begin{tikzpicture}[scale=.3]
          %   \draw[black] (0,0) -- (13,0) -- (13,13) -- (0,13) -- cycle;
          %   \draw[fill = black!60, color=black!60] (2,8) circle (8mm);
          %   \draw[fill = black!60, color=black!60] (5,2) circle (8mm);
          %   \draw[fill = black!60, color=black!60] (8,5) circle (8mm);
          %   \draw[fill = black!60, color=black!60] (11,11) circle (8mm);
          % \end{tikzpicture}
          % \hspace{1pc}
          % \begin{tikzpicture}[scale=.3]
          %   \draw[black] (0,0) -- (13,0) -- (13,13) -- (0,13) -- cycle;

          %   \draw[color=black!20, fill=black!20] (2,8) circle (16mm);
          %   \draw[fill = black] (1.5,7.5) circle (4mm);
          %   \draw[fill = black] (2.5,8.5) circle (4mm);

          %   \draw[color=black!20, fill=black!20] (5,2) circle (20mm);
          %   \draw[fill = black] (4,3) circle (4mm);
          %   \draw[fill = black] (5,2) circle (4mm);
          %   \draw[fill = black] (6,1) circle (4mm);

          %   \draw[color=black!20, fill=black!20] (8,5) circle (12mm);
          %   \draw[fill = black] (8,5) circle (4mm);

          %   \draw[color=black!20, fill=black!20] (11,11) circle (8mm);
          % \end{tikzpicture}
          % \hspace{1pc}
          % \begin{tikzpicture}[scale=.3]
          %   \draw[black] (0,0) -- (13,0) -- (13,13) -- (0,13) -- cycle;
          %   \draw[fill = black] (1,9) circle (4mm);
          %   \draw[fill = black] (3,11) circle (4mm);

          %   \draw[fill = black] (5,5) circle (4mm);
          %   \draw[fill = black] (7,3) circle (4mm);
          %   \draw[fill = black] (9,1) circle (4mm);

          %   \draw[fill = black] (11,7) circle (4mm);
          % \end{tikzpicture}

        \begin{tikzpicture}[scale=.3]
          \draw[black] (0,0) -- (13,0) -- (13,13) -- (0,13) -- cycle;
          \draw[fill=black] (2,8) circle (4mm);
          \draw[fill=black] (5,2) circle (4mm);
          \draw[fill=black] (8,5) circle (4mm);
          \draw[fill=black] (11,11) circle (4mm);
          \foreach \i in {3.4, 6.5, 9.6}{
            \draw[dotted] (0,\i) -- (13,\i);
            \draw[dotted] (\i,0) -- (\i,13);
          }
        \end{tikzpicture}
        \hspace{1pc}
        \begin{tikzpicture}[scale=.3]

          \draw[draw=none,fill=black!15] (0, 6.5) rectangle (3.25,9.75);
          \draw[draw=none,fill=black!15] (3.25, 0) rectangle (6.5,3.25);
          \draw[draw=none,fill=black!15] (6.5, 3.25) rectangle (9.75,6.5);
          \draw[draw=none,fill=black!15] (9.75, 9.75) rectangle (13,13);

          \draw[black] (0,0) -- (13,0) -- (13,13) -- (0,13) -- cycle;

          \foreach \i in {3.25, 6.5, 9.75}{
            \draw[dotted] (0,\i) -- (13,\i);
            \draw[dotted] (\i,0) -- (\i,13);
          }

          \draw[fill = black] (1.1,7.4) circle (2mm);
          \draw[fill = black] (2.1,8.4) circle (2mm);

          \draw[fill = black] (4, 2.5) circle (2mm);
          \draw[fill = black] (5, 1.5) circle (2mm);
          \draw[fill = black] (6, .5) circle  (2mm);

          \draw[fill = black] (8,4.8) circle (2mm);

          \draw[] (11.5,11.5) circle (4mm);
        \end{tikzpicture}
        \hspace{1pc}
        \begin{tikzpicture}[scale=.325]
          \draw[black] (0,0) -- (12,0) -- (12,12) -- (0,12) -- cycle;
          \draw[fill = black] (1,9) circle (4mm);
          \draw[fill = black] (3,11) circle (4mm);

          \draw[fill = black] (5,5) circle (4mm);
          \draw[fill = black] (7,3) circle (4mm);
          \draw[fill = black] (9,1) circle (4mm);

          \draw[fill = black] (11,7) circle (4mm);

          \foreach \i in {2,4,6,8,10} {
            \draw[dotted] (\i, 0) -- (\i, 12);
            \draw[dotted] (0,\i) -- (12,\i);
          }
        \end{tikzpicture}
      \caption{The peg permutation $\tro = \pl3 \mn1 \dt2 \pl4$ inflated by the
      vector $\vec{i} = (2,3,1,0)$ is the permutation $563214$.}
      \label{polyclass:fig:inflation}
      \end{figure}
      
      \begin{definition} \label{polyclass:def:inflations}
        For a peg permutation $\tro$, denote by $\I(\tro)$ the set of all valid
        inflations of $\tro$. Similarly, for a set $\tS$ of peg permutations, let 
        $$ \I(\tS) = \cup_{\tro \in \tS} \I(\tro).$$
        It follows that for a permutation $\pi \in \I(\tro)$, there exists some
        partition $P$ of the entries of $\pi$ into monotone intervals which are
        compatible with $\tro$. This partition is referred to as a
        $\tro$-partition of $\pi$. 
      \end{definition}

      It can be easily shown that, for a peg permutation $\tro$ of length $n$, if
      $\vec v = (v_1, v_2 \dots v_n) \in \Zgeq^n$ and $\vec w = (w_1, w_2, \dots
      w_n) \in \Zgeq^n$ are two vectors such that $v_i \leq w_i$ for all $i \in [n]$,
      then $\tro(\vec v) \prec \tro(\vec w)$ as permutations. Also, note that
      $\I(\tro)$ forms a permutation class, and in fact, as we shall soon see, a
      polynomial class. 

      \begin{theorem}[\cite{SophieVince, GridClasses}] \label{polyclass:thm:tfae}
        For a permutation class $\C$, the following are equivalent. 
        \begin{enumerate}[1)]
          \item $\C$ is a polynomial class,
          \item $\C_n < f_n$ for some $n$, where $f_n$ is the $n$th Fibonacci
          number,
          \item $\C$ does not contain arbitrarily long patterns of the forms
          shown in Figure~\ref{polyclass:fig:alternations},
          \item $\C = \I(\tS)$ for some set $\tS$ of peg permutations. 
        \end{enumerate}
      \end{theorem}

      \begin{figure}[t] \centering
        \begin{tikzpicture}[scale=.2]

          % loops over the 3 boxes
          \foreach \i in {1, 4, 7}{
            \draw[very thick, color=lightgray] 
            (\i,\i) -- (\i, \i+3) -- (\i+3, \i+3) -- (\i+3,\i) -- cycle;
            \draw[fill = black] (\i+2, \i+1) circle (12pt);
            \draw[fill = black] (\i+1, \i+2) circle (12pt);
            }

          % had to draw the ellipses by hand...
          \draw[fill = black] 
                (10.5,10.5) circle (2pt)
                (11,11) circle (2pt)
                (11.5,11.5) circle (2pt);
        \end{tikzpicture}
          \hspace{2em}
        \begin{tikzpicture}[scale=.2]
          \foreach \i in {1, 4, 7}{
            \draw[very thick, color=lightgray] 
            (\i,12-\i) -- (\i, 12-\i-3) -- (\i+3,12-\i-3) -- (\i+3,12-\i) -- cycle;
            \draw[fill = black] (\i+2, 12-\i-1) circle (12pt);
            \draw[fill = black] (\i+1, 12-\i-2) circle (12pt);
            }

          \draw[fill = black] 
                (10.5,1.5) circle (2pt)
                (11,1) circle (2pt)
                (11.5,0.5) circle (2pt);
        \end{tikzpicture}
          \hspace{2em}
        \begin{tikzpicture}[scale = .2]
          % draws the boundaries
          \draw[very thick, color = lightgray]
            (6,0) -- (6,12) -- (12,12) -- (12,0) -- (0,0) -- (0,12) -- (6,12);

          % draws the permutation
          \foreach \y [count = \x] in {2,4,6,8}
            \draw[fill = black] (\x, 1.1*\y) circle (12pt);
          \foreach \y [count = \x] in {1,3,5,7}
            \draw[fill = black] (\x + 6, 1.1*\y) circle (12pt);

          % draws dots
          \foreach \x/\y in {4.5/9, 4.75/9.5, 5/10}
          \draw[fill = black]
            (\x, 1.1*\y) circle (2pt);

          \foreach \x/\y in {4.5/9, 4.75/9.5, 5/10}
          \draw[fill = black]
            (\x + 6, 1.1*\y - 1) circle (2pt);
        \end{tikzpicture}
          \hspace{2em}
        \begin{tikzpicture}[scale = .2]
          % draws the boundaries
          \draw[very thick, color = lightgray]
            (6,0) -- (6,12) -- (12,12) -- (12,0) -- (0,0) -- (0,12) -- (6,12);

          % draws the permutation
          \foreach \y [count = \x] in {2,4,6,8}
            \draw[fill = black] (\x, 1.1*\y) circle (12pt);
          \foreach \y [count = \x] in {1,3,5,7}
            \draw[fill = black] (12-\x, 1.1*\y) circle (12pt);

          % draws dots
          \foreach \x/\y in {4.5/9, 4.75/9.5, 5/10}
          \draw[fill = black]
            (\x, 1.1*\y) circle (2pt);

          \foreach \x/\y in {4.5/9, 4.75/9.5, 5/10}
          \draw[fill = black]
            (12- \x, 1.1*\y - 1) circle (2pt);
        \end{tikzpicture}

      \caption{If a class contains arbitrarily long patterns of any of these forms,
      it is not a polynomial class.}
      \label{polyclass:fig:alternations}
      \end{figure}
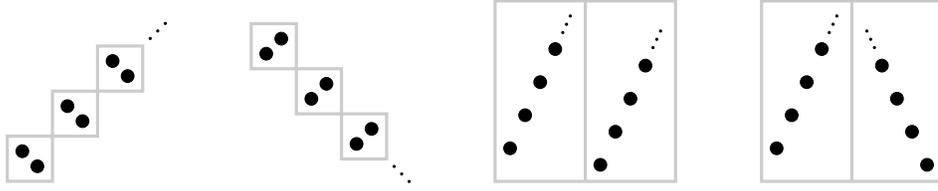

    \subsection{Peg Patterns}

      Analogous to the permutation pattern ordering, we can define an ordering on
      peg permutations. Essentially, we say that a peg permutation is contained in
      another if it can be obtained by deleting entries and changing signs to dots.

      \begin{definition} \label{polyclass:def:pegpattern}
        Let $\tro = \tro_1 \tro_2 \dots \tro_n$ and $\ttau = \ttau_1 \ttau_2 \dots
        \ttau_k$ be peg permutations. Say that $\ttau$ is contained within $\tro$ if
        there is a subsequence $\tro_{i_1} \tro_{i_2} \dots \tro_{i_k}$, whose
        entries lie in the same relative order as those of $\ttau$ and whose
        decorations are \emph{compatible}, meaning that $\tro_{i_j}$ either have
        the same decoration or $\ttau_j$ is decorated with a dot. 
      \end{definition}

      It follows from the definitions that if $\ttau \prec \tro$, then $\I(\ttau)
      \subset \I(\tro)$. However, the converse is not true. For example, letting
      $\ttau = \dt 1 \dt 2$ and $\tro = \pl 1$, we see that $\I(\ttau) = \{1, 12\}
      \subset \I(\tro)$, but $\ttau \not \prec \tro$. The core idea of the algorithm
      is the partition all permutations of the class according to peg permutation,
      and then enumerate these by enumerating integer vectors. 

      \begin{definition} \label{polyclass:def:fillingperm}
        For a peg permutation $\tro$ and a permutation $\pi$, say that $\pi$
        \emph{fills} $\tro$ if $\pi = \tro(\vec v)$ such that $\vec v_i = 1$
        whenever $\tro_i$ is decorated with a dot, and $\vec v_i \geq 2$
        otherwise. Every peg permutation $\tro$ has a unique minimal filling
        permutation, denoted $m_{\tro}$. 
      \end{definition}

    \subsection{Integer Vectors}

      Peg permutations provide a way of translating between integer vectors and
      permutations. The underlying idea of the algorithm is to formalize this
      correspondence in a way which preserves the ordering, converting permutation
      posets into posets of integer vectors. We will now establish some machinery
      for working with and enumerating integer vector posets.

      Downsets in the integer vector poset are easier to work with than permutation
      classes in part because of Higman's Theorem~\cite{HigmansThm}, which
      implies that every downset has a finite basis. The union and intersection
      of these downsets is easy to compute as well. 

      \begin{definition} \label{polyclass:def:vecposet}
        For two vectors $\vec v, \vec w \in \Zgeq^n$, say that $\vec v \prec \vec
        w$ if $v_i \leq w_i$ for each $i \in [n]$.  For a downset $\cV \subset
        \Zgeq^n$, denote by $\cB_\cV$ the set of minimal vectors in the
        complement of $\cV$. It follows then that $\cV$ can be described as
        precisely those vectors which \emph{avoid} the vectors of $\cBV$, that
        is, 

        $$ \cV := \{ \vec v \in \Zgeq^n : \vec b_i \not \prec \vec v, \frall
            \vec b_i \in \cBV \}. $$

        For two vectors $\vec v, \vec w \in \Zgeq^n$, denote by $\vec v \vee \vec
        w$ the minimal vector for which $\vec v \prec \vec v \vee \vec w$ and $\vec
        w \prec \vec v \vee \vec w$. It follows that $ (\vec v \vee \vec w)_i =
        \max(\vec v_i, \vec w_i)$ for each $i \in [n]$. 
      \end{definition}

      \begin{proposition} \label{polyclass:prop:vectorunions}
        Let $\cV, \cW$ be downsets in $\Zgeq^n$ with corresponding downsets $\cBV,
        \cBW$.  Letting $\cBM$ be the minimal vectors of the set $\{\vec v \vee
        \vec w : \vec v \in \cV, \vec w \in \cW \}$, 
        $\cBU$ the minimal vectors of the union $\cBV \cup \cBW$, and $\cM$
        and $\cU$ the downsets which avoid $\cBM$ and $\cBU$, respectively. We
        have that 
        $$ \cV \cap \cW = \cU, $$
        $$ \cV \cup \cW = \cM .$$
      \end{proposition}
      \begin{proof}
        Clearly, any vector in $\cV \cap \cW$ must avoid all basis elements of
        both $\cBV$ and $\cBW$, and so the basis for $\cV \cap \cW $ is the set
        of minimal elements of the set $\cBV \cup \cBW$. For
        unions, we proceed using De Morgan's laws:
        $$ \begin{aligned}
          \cV \cup \cW
          &= \left( \bigcap_{\vec v \in \cBV} \{\vec v \mbox{-avoiding
                vectors}\}\right) \bigcup
             \left( \bigcap_{\vec w \in \cBW} \{\vec w \mbox{-avoiding
                vectors}\}\right) \\
          &= \bigcap_{\substack{ \vec v \in \cBV \\ \vec w \in \cBW }} 
            \Big( \{\vec v \mbox{-avoiding vectors}\} 
              \cup \{\vec w \mbox{-avoiding vectors}\}\Big) \\
          &= \bigcap_{\substack{ \vec v \in cBV \\ \vec w \in \cBW }} 
              \{\vec v \vee \vec w \mbox{-avoiding vectors}\}.
          \end{aligned} $$

        Therefore the basis for $\cV \cup \cW$ consists of the set $\cBM$,
        completing the proof.
        \end{proof}

        Proposition~\ref{polyclass:prop:vectorunions} can also be used to enumerate 
        downsets of integer vector classes, using inclusion exclusion. It will be
        useful to consider these downsets as collections of point-sets on an integer
        lattice, and to enumerate the classes based on the number of $n$-element
        sets they contain. We formalize this below. 

      We define the \emph{weight} of a vector as the sum of its entries. A peg
      permutation, inflated by a vector of weight $k$, produces a permutation of
      length $k$. Counting integer vectors according to weight is relatively
      simple, and is equivalent to counting ordered compositions. Letting
      $a_{n,k}$ denote the number of $k$-weight vectors in $\Zgeq^n$, we have 
      $$ \sum_{k \geq 0} a_{n,k} z^k = \frac{1}{(1 - z)^n}.$$

      Similarly, the generating function for the number of permutations which
      contain a given vector $\vec v \in \Zgeq^n$ is given by 
      $$ \frac{z^{\wt(\vec v)}}{(1 - z)^n} .$$
      It follows from this and Proposition~\ref{polyclass:prop:vectorunions} that
      that downsets can be enumerated by adding and subtracting generating
      functions of this form. This leads to the following lemma. 

      \begin{lemma} \label{polyclass:lem:fillinggfcn}
        Let $\tro$ be a peg permutation, and let $s$ be the number of signs in
        the decoration of $\tro$, and $d$ the number of dots. Then the generating
        function for the filling permutations of $\tro$ is given by 
        $$ \frac{z^{d + 2s}}{(1 - z)^s}.$$ 
      \end{lemma}
      
      Lemma~\ref{polyclass:lem:fillinggfcn} will ultimately be our enumeration
      scheme for these classes. The main barrier is partitioning the class
      into categories based on which peg permutation they fill. The bulk of the
      algorithm, described in the next section, will be performing this
      partitioning.

  % =========================================================================== %
  \section{The Algorithm}
  \label{polyclass:sec:algo}
  % =========================================================================== %

    This section gives an overview of the enumeration algorithm, given a set of
    peg permutations as an input, and outputting a disjoint set of integer vector
    downsets, which can then be enumerated. The algorithm consists of three
    parts. First the set is \emph{completed}, then \emph{compacted}, and finally
    \emph{cleaned}, at which point we have a set of peg permutations which
    partition the class. Letting $\tS$ be a set of peg permutations, we describe
    each part in detail below, with the goal of enumerating the class $\I(\tS)$.
    A pseudocode overview of the algorithm is shown in
    Figure~\ref{polyclass:fig:pseudocode}.

    \subsection{Completing the Set}
      Say that a set $\tS$ is \emph{complete} if every permutation $\pi \in
      \I(\tS)$ fills at least one element $\tro \in tS$. For example, the set
      $\{\pl 2 \pl 1\}$ is not complete, because $1\ 2\ 3 \in \I(\pl 2 \pl 1)$
      (since $1\ 2\ 3 = \pl 2 \pl1 (3,0)$), but doesn't fill $\pl 2 \pl 1$. It
      follows from the definition of peg patterns, however, that every
      permutation in $\I(\tS)$ must fill some pattern within an element in $\tS$. 

      Therefore, the downset of any peg pattern is a complete set.  The first
      step of the algorithm completes the set $\tS$ by, 
      for each $\tro \in \tS$, we add all patterns of $\tro$ into the set $\tS$.
      After this step, the set $\tS$ is complete. 

    \subsection{Compacting the Set}

      The next obstacle in the enumeration is ensuring that every permutation in
      the class fills a unique peg permutation in the set.  Given a permutation,
      we can divide its entries up into monotone intervals in a number of ways.
      The following lemma will help to ensure uniqueness, and allow for
      enumeration. 

      \begin{lemma} \label{polyclass:lem:intervals}
        If two monotone intervals intersect, then their union and intersection
        are also monotone intervals.
      \end{lemma}
      \begin{proof}
        Suppose we have two monotone intervals with a non-empty intersection.
        Without loss of generality, suppose that one of them is increasing, and
        so their intersection is either increasing or consists of a single
        element. Since each interval consists of contiguous entries, the second
        entry must also be increasing, and so the union of the two is an
        increasing interval. 
      \end{proof}
      
      Lemma~\ref{polyclass:lem:intervals} implies that by greedily choosing the
      largest possible intervals, we can ensure that for each permutation $\pi$,
      there is a unique smallest peg permutation $\tro$ for which $\pi$ is in
      $\I(\tro)$, but not in $\I(\ttau)$ for any $\ttau \prec \tro$. However, not
      all peg permutations are able to fulfill this role. 

      Say that a peg permutation $\tro$ is \emph{compact} if, for all $\ttau \prec
      \tro$, we have that $\I(\ttau) \not = \I(\tro)$. For example, $\dt 2 \mn 1$
      is not compact, since $\I(\dt 2 \mn 1) = \I(\mn 1)$. The following
      lemma and proposition characterizes these peg permutations.

      \begin{proposition}\label{polyclass:prop:compact}
        For a peg permutation $\tro$, the following are equivalent:
        \begin{enumerate}[1)]
        \item $\tro$ is compact, 
        \item $\tro$ does not contain the patterns $\pl1 \pl2, \pl1 \dt2, \dt1
          \pl2$ or, symmetrically, $\mn2 \mn1, \mn2 \dt1$ or $\dt2 \mn 1$, 
        \item every permutation $\pi$ which fills $\tro$ has a unique vector
          $\vec v$ for which $\tro(\vec v) = \pi$. 
        \end{enumerate}
      \end{proposition}
      \begin{proof}
        First we show that (1) and (2) are equivalent. Clearly (1) implies (2),
        so to show the reverse implication, let $\tro$ be a noncompact peg
        permutation. By definition, there exists some $\ttau \prec \tro$ such that
        $\I(\ttau) = \I(\tro)$. Let $\pi$ be a permutation which fills $\tro$,
        with $P$ the $\tro$-partition and $P'$ the $\ttau$ partition. Because
        $\ttau$ is shorter than $\tro$, it follows that there must be some part
        of $P'$ which intersects two parts of $P'$. By
        Lemma~\ref{polyclass:lem:intervals} these two form a monotone
        interval, and so must be of one of the forms listed in (2). 

        Now, we show that (2) and (3) are equivalent. If a peg permutation
        $\tro$ contains one of the patterns specifies in (2), it is clear that a
        permutation can fill $\tro$ in at least two different ways, so (3)
        implies (2). Suppose that the permutation $\pi$ fills $\tro$ with two
        different $\tro$-partitions $P$ and $P'$. It follows then that
        a block of one partition must intersect two blocks of the other. However,
        this implies (Lemma~\ref{polyclass:lem:intervals}) that intersection
        and unions are also monotone, and so must be of one of the forms given in
        (2). 
      \end{proof}

      By simply removing each of the peg permutations which contain one of the
      intervals listed in Proposition~\ref{polyclass:prop:compact}, our set of
      peg permutations becomes a set of compact peg permutations. Further, since
      our set is a full and complete downset, the definition of compact implies
      that the new set will still be complete. 

    \subsection{Cleaning the Set}

      The final step in the algorithm is bijecting our complete and compact set
      of peg permutations to a set of downsets of integer vectors. Our final
      obstacle in this bijection will be peg permutations which have intervals of
      dotted entries. For example, the peg permutation $\dt 1 \dt 2 \dt 3 \dt 4$
      produces a class which is strictly contained in $\pl 1$, but there is no
      containment at the level of peg permutations. We remedy this by using
      forbidden vectors: the peg permutation $\dt 1 \dt2 \dt 3 \dt4$ is mapped to
      the inflations of $\pl 1$ which \emph{avoid} the vector $\vect{5}$. 

      \begin{definition} \label{polyclass:def:clean}
        Say that a peg permutation $\tro$ is \emph{clean} if $\I(\tro) \not
        \subset \I(\ttau)$ for any shorter permutation $\ttau$. 
      \end{definition}

      \begin{proposition}\label{polyclass:prop:clean}
        The compact peg permutation $\tro$ is clean if and only if it does not
        contain an interval order isomorphic to $\dt 1 \dt 2$ or $\dt 2 \dt 1$. 
      \end{proposition}
      \begin{proof}
        If $\tro$ contains one of the specified intervals, then letting $\ttau$ be
        the shorter peg permutation obtained by contracting these two entries
        into a single entry with the appropriate sign, we find that $\I(\tro)
        \subseteq \I(\ttau)$. 

        For the other direction, suppose that $\I(\tro) \subseteq \I(\ttau)$ for
        some shorter peg permutation $\ttau$. Let $\pi$ be any permutation which
        fills $\tro$. In any $\ttau$-partition of $\pi$ there must be a monotone
        interval formed from entries in different parts of any $\tro$ partition.
        Because $\tro$ is compact, it follows (from
        Proposition~\ref{polyclass:prop:compact}) that $\tro$ must contain either
        $\dt 1 \dt 2$ or $\dt 2 \dt 1$, completing the proof.
      \end{proof}

      Given a complete and compact set $\tS$ of peg permutations, it is not
      possible in general to find a clean set which inflates to the same class.
      To see this, let $\tro = \dt 1 \dt 2 \dt 3$. Then there is no clean set
      whose inflation is equal to $\I(\tro)$. However, we can put the set $\tS$
      in bijection with a clean set \emph{together} with a set of allowable
      inflation vectors. We formalize this below. 

      \begin{definition}\label{polyclass:def:restricted-inflations}
        For a peg permutation $\tro$ and a set $\cV$ of vectors of the same
        length, let $\I(\tro;\cV)$ denote the inflations of $\tro$ using vectors
        from the set $\cV$. 
      \end{definition}

      \begin{lemma}\label{polyclass:lem:clean}
        For each peg permutation $\tro$, there exists a clean permutation $\ttau$
        and a vector set $\cV$ such that the set of all inflations of $\tro$ is equal
        to $\I(\ttau, \cV)$. 
      \end{lemma}
      \begin{proof}
        To construct $\ttau$, simply contract all of the intervals of dotted
        entries in $\tro$ into signed entries. To construct $\cV$, build a vector
        $\vec v$ such that, if the entry $\ttau_i$ arose from a dotted interval
        of length $k$, let $\vec v_i = k+1$, and take $\cV$ to be the set of
        vectors avoiding $\vec v$.  This ensures that this entry will never be
        inflated by a run longer than the original sequence of dotted entries.
      \end{proof}

      The final step of the algorithm can be described as follows. First, let $\fV$ be
      an empty set, which will be the output. For each peg permutation $\tro \in \tS$, 
      compute the pair $(\ttau, \vec W)$ as described in
      Lemma~\ref{polyclass:lem:clean}, and let $\cV$ be the vector downset with
      basis $\cBV = \{\vec v\}$. 
      If there is no pair $(\ttau, \cW)$ in the
      set $\fV$, add $(\ttau, \cV)$ to $\fV$. 
      Otherwise, replace $(\ttau, \cW)$ with $(\ttau, \cW \cup \cV)$. 

      Since every permutation in the class fills a unique clean and compact
      peg permutation, and since each permutation which fills a compact
      permutation has a unique partition, it follows that the polynomial class is
      in bijection with the set 
      $$ \biguplus_{(\tro,\cV) \in \fV} \I(\tro, \cV).$$

      Letting $\vec m_{\tro}$ be the vector defined by $(\vec m_{\tro})_i
      = 1$ if $\tro_i$ is decorated with a dot, and $(\vec m_{\tro})_i = 2$
      otherwise, and let $s(\tro)$ denote the number of signed (non-dotted)
      entries of $\tro$. The generating function for $\I(\tro)$ is then given by
      inclusion exclusion in conjunction with
      Proposition~\ref{polyclass:prop:vectorunions}, and allows us to efficiently
      enumerate these classes.  
      $$ \sum_{B \subseteq \cBV} (-1)^{|B|} 
          \frac{ z^{\wt \left( \vec m_{\tro} \vee (\bigvee B)\right)}}{%
          {(1-z)}^{ s( \tro ) } }.  $$

    \SetAlFnt{\small\ttfamily}
    \begin{algorithm}
      \DontPrintSemicolon
      \KwIn{Set $\tS$ of Peg Permutations}
      \KwOut{Integer vectors in bijection with the class}

      \tcp{Complete $\tS$}
      \For{$\tro \ \in \tS$}{
        Add to $\tS$ all peg permutations which can be realized by deleting entries of
        $\tro$, or changing a signs to dots\;
      }

      \tcp{Remove all non-compact elements from $\tS$}
      \For{$\tro\ \in \tS$}{
        \If{\ttfamily $\tro$ contains any of the consecutive permutations 
        $\pl{1}\pl{2}, \dt{1}\pl{2}, \pl{1}\dt{2}$ or their symmetries
        }{ 
        Remove $\tro$ from $\tS$\; }
      }

      \tcp{Clean $\tS$ and construct a vector set}
      Initialize the set $\mathfrak{V}$, which will contain pairs $(\tro,
      \cV)$, where $\tro$ is a peg permutation and $\cV$ is a set of
      integer vectors of the same length as $\tro$\;
      \For{$\tro \  \in $ $\tS$}{
        \eIf{\ttfamily $\tro$ contains intervals of the form $\dt{1}\dt{2}$ or
        $\dt{2}\dt{1}$}
        { Let $\ttau$ denote the cleaned $\tro$, and $\cV$ the set of integer
        vectors for which $\{\ttau[\vec{v}] : \vec{v} \in \cV\}  = \{\tro[\vec{v}]
        : \vec{v} \in \mathcal{F}_{\tro}\} $\; Let $(\tro',\cV')
        \leftarrow (\ttau, \cV)$\;
        }
        { 
          Let $(\tro', \cV') \leftarrow  (\tro, \mathcal{F}_{\tro})$\;
        }
        \eIf{\ttfamily $(\tro', \cW) \in \mathfrak{V}$ for some $\cW$}{
          Replace the element $(\tro',\cW)$ with $(\tro', \cW \cup \cV')$\;
        }
        {
          Add $(\tro', \cV')$ to $\mathfrak{V}$\;
        }
      }
      The permutation class is now in bijection with the disjoint union
      $\displaystyle\biguplus_{(\tro,\cV) \in \mathfrak{V}} \{
        \tro[\vec{v}] : \vec{v} \in \cV \}$.\; 
    \SetAlFnt{\normalsize\sffamily}
    \caption{A pseudocode overview of the algorithm.}
    \label{polyclass:fig:pseudocode}
    \end{algorithm}

  % =========================================================================== %
  \section{Genomics}
  \label{polyclass:sec:data}
  % =========================================================================== %

    The field of computational biology is a new and rapidly developing field.
    The vast quantities of sequencing data produced by modern geneticists
    necessitate the use of complex mathematical techniques for analysis.
    A common problem, given two related genetic sequences, is to determine the
    most recent evolutionary ancestor. This is generally solved by determining
    the number of mutations required to rearrange one sequence into the other,
    allowing a researcher to determine the midpoint between the two. Determining
    this distance, however, is computationally difficult, but the work presented
    in this chapter can be used to effectively and efficiently perform these and
    other computations. 
    
    This section applies the theory of polynomial classes to the problem of
    evolutionary distance. 
    While the focus is on the combinatorial aspects of genome
    rearrangement, we begin with a rough overview of the biological mechanics.
    For a more complete introduction, see the surveys \cite{CompBio} or
    \cite{GenomeBook}.

    \subsection{Chromosomes and Mutation}
        
      Every living organism encodes its hereditary information in molecules
      called \emph{chromosomes}, the set of which is known as the organism's
      \emph{genome}. The information carried in the genome is passed down from
      organism to organism, and undergoes mutations which can cause both subtle
      and dramatic change between generations. 
    
      Each chromosome is composed of double strands of deoxyribonucleic acid
      (DNA), each strand of which is in turn composed of a sequence of
      \emph{nucleotides}. Nucleotides come in four types (A, C, G, and T), and
      the two strands, arranged in a double helix, are complementary, i.e., a A's
      are always coupled with a T, and G's with C. It follows that DNA can be
      defined as a single sequence - a \emph{word} on the alphabet
      $\{\text{A,C,G,T}\}$. A DNA \emph{sequence} is some consecutive piece of
      this word, while \emph{genes} are the smallest sequences which have some
      independent biological function.
      
      The genome is made up of chromosomes, which are in turn made up of coiled
      DNA strands, which can be broken down into genes sequences, which
      themselves are simply sequences of nucleotides. This complexity leads to
      a variety of errors which can be introduced during replication, and these
      inaccuracies are the basis for genetic evolution. Many of these
      mutations can be viewed as rearranging sequences of genes, and can be
      effectively modelled using permutations. 

      The physical properties of chromosomes lead to a variety of rearrangement
      operations, but they share a common theme: some contiguous segment of the
      gene sequence is removed, reversed and/or relocated, then replaced back in
      the sequence. While there are other mutations possible at both the larger
      and smaller scales, these so called \emph{genome-rearrangements} have
      received much attention in recent research and, most importantly, fall
      under the purview of polynomial classes.

    \subsection{Block Transformations}
    
      Permutations are apt models for rearrangement, and can be used to study
      genetic mutations. Mutations happen in various ways, and a variety of
      permutation transformations have been studied. These operations are known
      collectively as \emph{block transformations}, as each of them acts on
      contiguous subsequences of permutations, henceforth referred to as
      \emph{blocks}. Each of these operations can be viewed as a set of allowable
      moves which transform one permutation into another. 
      
      Treating block transformations as mutations, the basic problem is as
      follows: given two permutations, what is the shortest sequence of moves
      which can transform one into the other? By relabelling the entries, we can
      assume, without any loss of generality, that the target permutation is the
      identity permutation. In this light, the question becomes a \emph{sorting}
      sorting problem, and asks how quickly a sequence can be sorted.  We present
      here some of the more commonly studied operations, but note that other
      varieties and models are biologically significant.  

      \begin{definition} \label{polyclass:def:block}
        Let $\pi = \pi_1 \pi_2 \dots \pi_n$ be a permutation written in one-line
        notation. A \emph{block} of $\pi$ is some contiguous string of entries 
        $\pi_{i} \pi_{i + 1} \dots \pi_{i+k}$. 
        A \emph{prefix} is a block which starts at $\pi_1$. 
      \end{definition}

      Blocks of permutations are models for gene sequences, and each of the block
      permutations below differ only in their treatment of blocks. We define each
      type of sorting by defining a single allowable operation. 

      \begin{definition}[Block Reversal]\label{polyclass:def:blockrev}
        A \emph{block reversal} operation consists of reversing the entries of
        any block of the permutation. This operation was first studied by
        Watterson, Ewens, Hall, and Morgan~\cite{Watterson1982} and further
        investigated by Alpar-Vajk~\cite{Alpar2009}.
      \end{definition}
      
      \begin{definition}[Block Transposition]\label{polyclass:def:blocktrans}
        A \emph{block transposition} operation consists of moving one block
        from its current position to any other location in the permutation. 
        This operation was first studied by Bafna and Pevzner~\cite{Bafna1998}. 
      \end{definition}
    
      \begin{definition}[Block Interchange]\label{polyclass:def:blockint}
        A \emph{block interchange} operation consists of selecting two
        non-intersecting blocks of the permutation and interchanging them. 
        This operation was first studied by Christie~\cite{Christie1996}, and
        further investigated by B\'ona and Flynn~\cite{Bona2009}. 
      \end{definition}
    
      \begin{definition}[Prefix Transposition]\label{polyclass:def:prefixtrans}
        A \emph{prefix transposition} operation consists of moving a prefix of
        the permutation to any other location in the permutation. 
        This operation was first studied by Dias and Meidanis~\cite{Dias2002}.
      \end{definition}
        
      \begin{definition}[Prefix Reversal]\label{polyclass:def:prefixrev}
        A \emph{prefix reversal} operation consists of reversing the entries of a
        prefix of the permutation. This is sometimes referred to as the `pancake
        flipping operation', and was first studied by ``Harry Dweighter''
        (actually, Jacob E. Goodman) as a \emph{Monthly}
        problem~\cite{Dweighter} (and was also studied by Gates~\cite{BillGates}). 
      \end{definition}
    
      \begin{definition}[Cut-Paste Sorting]\label{polyclass:def:cutpaste}
        A \emph{cut-paste} operation consists of moving a block of the
        permutation, with the option to reverse its entries. 
        This operation was first studied by Cranston, Sudborough, and
        West~\cite{Cranston2007}. 
      \end{definition}
    
      For a given block transformation, we refer to the \emph{distance} between
      two permutations $\pi$ and $\sg$ as the minimum number of operations needed
      to transform one into the other. Finding the maximal distance between two
      permutations of a given length is equivalent to finding the maximal
      distance from the identity to any permutation. Further, since each of these
      operations is reversible --- if $\pi$ can be transformed into $\sg$, then
      $\sg$ can be transformed into $\pi$ --- this is equivalent to finding the
      distance from the identity to any permutation.

      Biologically, two permutations with a small distance represent closely
      related organisms, as each transformation represents a mutation which can
      occur from one generation to the next. Understanding the sets of
      permutations at each fixed distance from the identity can help to
      understand how different genomes are related. For any $k \in \Zgeq$, the
      set of permutations which are at distance $\leq k$ from the identity forms
      a polynomial class, and thus can be enumerated by our algorithm. 

      \begin{theorem} \label{polyclass:def:operation-polyclass}
        For each of the operations presented above and for a positive integer
        $k$, the set of permutations with distance at most $k$ from the identity
        forms a polynomial class. 
      \end{theorem}
      \begin{proof}
        The class of identity permutations is the inflations of the peg
        permutation $\pl 1$, which can be represented geometrically as a diagonal
        line parallel to $y =x $. Each block transformation can be viewed as
        taking some piece of this line and moving or reversing it. Such an array
        of lines can be translated back into a peg permutation, and it follows
        that the set of distance $\leq k$ permutation can be represented as the
        union of all peg permutations obtained in this way.  See
        Figures~\ref{polyclass:fig:blockrev} and~\ref{polyclass:fig:twoblockrevs}
        for graphical examples. 
      \end{proof}

      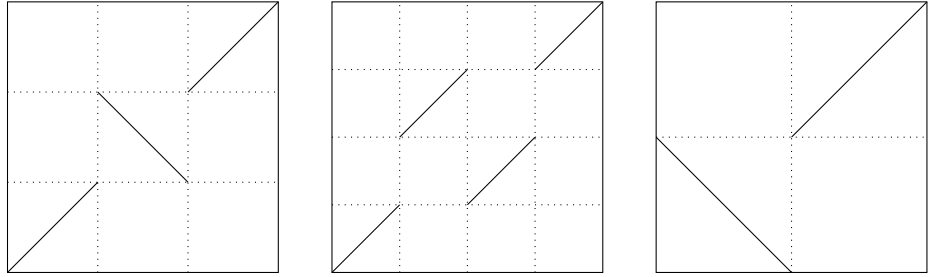
\begin{figure}[t] \centering
        \begin{tikzpicture}[scale=.3]
          \draw[black] (0,0) -- (12,0) -- (12,12) -- (0,12) -- cycle;
          \foreach \i in {4,8}{
            \draw[dotted] (0, \i) -- (12, \i);
            \draw[dotted] (\i, 0) -- (\i, 12);
          }
          \draw (0,0) -- (4,4);
          \draw (4,8) -- (8,4);
          \draw (8,8) -- (12,12);
        \end{tikzpicture}
        \hspace{1pc}
        \begin{tikzpicture}[scale=.3]
          \draw[black] (0,0) -- (12,0) -- (12,12) -- (0,12) -- cycle;
          \foreach \i in {3,6,9}{
            \draw[dotted] (0, \i) -- (12, \i);
            \draw[dotted] (\i, 0) -- (\i, 12);
          }
          \draw (0,0) -- (3,3);
          \draw (3,6) -- (6,9);
          \draw (6,3) -- (9,6);
          \draw (9,9) -- (12,12);
        \end{tikzpicture}
        \hspace{1pc}
        \begin{tikzpicture}[scale=.3]
          \draw[black] (0,0) -- (12,0) -- (12,12) -- (0,12) -- cycle;
          \foreach \i in {6}{
            \draw[dotted] (0, \i) -- (12, \i);
            \draw[dotted] (\i, 0) -- (\i, 12);
          }
          \draw (0,6) -- (6,0);
          \draw (6,6) -- (12,12);

        \end{tikzpicture}
        \caption[The classes of permutations which are at most one block 
        reversal]{The classes of permutations which are at most one block
        reversal, block transposition, and prefix reversal away from the identity
        are given by $\I(\pl 1 \mn 2 \pl 3)$, $\I(\pl 1 \pl 3 \pl 2 \pl 4)$, and 
        $\I(\mn 1 \pl 2)$, respectively.}
        \label{polyclass:fig:blockrev}
      \end{figure}

      \begin{figure}[t] \centering
        \begin{tikzpicture}[scale=.2]
          \draw[black] (0,0) -- (15,0) -- (15,15) -- (0,15) -- cycle;
          \foreach \i in {3,6,9,12}{
            \draw[dotted] (0, \i) -- (15, \i);
            \draw[dotted] (\i, 0) -- (\i, 15);
          }
          \draw (0,0) -- (3,3);
          \draw (12,12) -- (15,15);
          \draw (3, 12) -- (6,9);
          \draw (6,6) -- (9,9);
          \draw (9,6) -- (12,3);

        \end{tikzpicture}
        \hspace{4pt}
        \begin{tikzpicture}[scale=.2]
          \draw[black] (0,0) -- (15,0) -- (15,15) -- (0,15) -- cycle;
          \foreach \i in {3,6,9,12}{
            \draw[dotted] (0, \i) -- (15, \i);
            \draw[dotted] (\i, 0) -- (\i, 15);
          }
          \draw (0,0) -- (3,3);
          \draw (12,12) -- (15,15);
          \draw (3, 6) -- (6,3);
          \draw (6,6) -- (9,9);
          \draw (9,12) -- (12,9);

        \end{tikzpicture}
        \hspace{4pt}
        \begin{tikzpicture}[scale=.2]
          \draw[black] (0,0) -- (15,0) -- (15,15) -- (0,15) -- cycle;
          \foreach \i in {3,6,9,12}{
            \draw[dotted] (0, \i) -- (15, \i);
            \draw[dotted] (\i, 0) -- (\i, 15);
          }
          \draw (0,0) -- (3,3);
          \draw (12,12) -- (15,15);
          \draw (0,0) -- (3,3);
          \draw (12,12) -- (15,15);
          \draw (3, 9) -- (6,12);
          \draw (6,6) -- (9,3);
          \draw (9,9) -- (12,6);

        \end{tikzpicture}
        \hspace{4pt}
        \begin{tikzpicture}[scale=.2]
          \draw[black] (0,0) -- (15,0) -- (15,15) -- (0,15) -- cycle;
          \foreach \i in {3,6,9,12}{
            \draw[dotted] (0, \i) -- (15, \i);
            \draw[dotted] (\i, 0) -- (\i, 15);
          }
          \draw (0,0) -- (3,3);
          \draw (12,12) -- (15,15);
          \draw (3,9) -- (6,6);
          \draw (6,12) -- (9,9);
          \draw (9,3) -- (12,6);

        \end{tikzpicture}
        \caption[The class of permutations which are at most two block reversals]{
        The class of permutations which are at most two block reversals
        from the identity is given by inflations of the four peg permutations
        $\pl 1 \mn 4 \pl 3 \mn 2 \pl 5$, $\pl 1 \mn 2 \pl 3 \mn 4 \pl 5$, 
        $\pl 1 \pl 4 \mn 2 \mn 3 \pl 5$, and $\pl 1 \mn 3 \mn 4 \pl 2 \pl 5$.}
        \label{polyclass:fig:twoblockrevs}
      \end{figure}
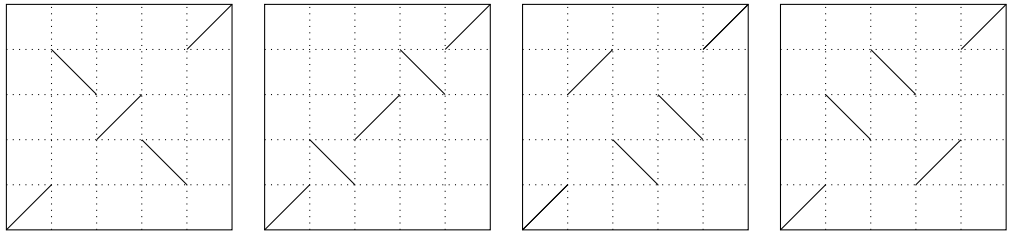

    \subsection{Data}
    
      Calculating the number of permutations of length $n$ which are at most $k$ operations
      away from the identity helps to understand how these block transformations
      differ, and how accurately they model biological mutation. The following
      tables show the numbers of these permutation in each radii from the
      identity, and build on the data presented in~\cite{GenomeBook}. The
      polynomials (in the variable $n$) enumerating these classes have integer
      coefficients when presented with the basis $\{\binom{n}{k}\}_{k \geq 0}$
      (as implied by~\cite{Klazar2003}). These enumerations are presented in the
      tables below.

      \begin{table}[t]
      \caption{Number of permutations of length $n$ within $k$ block transpositions of the
                identity.}
      \begin{footnotesize}
      $$
      \begin{array}{rrrrrrrrrrrc} 
      k &\ 1\ &\ 2\ &\ 3\ &\ 4\ &\ 5\ &\
      6\ &\ 7\ &\ 8\ &\ 9\ &10&\text{\OEISref}\\\hline
      1&1&2&5&11&21&36&57&85&121&166&\text{\OEISlink{A000292}}\\[0.5ex]
      &\multicolumn{10}{c}{{n\choose 0}+{n\choose 2}+{n\choose 3}}&\\[1ex]
      2&1&2&6&23&89&295&827&2017&4405&8812&\text{\OEISlink{A228392}}\\[0.5ex]
      &\multicolumn{10}{c}{{n\choose 0}+{n\choose 2}+2{n\choose 3}+8{n\choose
      4}+18{n\choose 5}+11{n\choose 6}}&\\[1ex]
      3&1&2&6&24&120&675&3527&15484&56917&179719&\text{\OEISlink{A228393}}\\[0.5ex]
      &\multicolumn{10}{c}{\text{\scriptsize $\nc0 + \nc2 + 2\nc3 + 9\nc4 +
      44\nc5 + 220\nc6 + 656\nc7 + 841\nc8 + 369\nc9$} }&\\[1ex]
      \end{array}
      $$
      \end{footnotesize}
      \end{table}

      \begin{table}[t]
      \caption{Number of permutations of length $n$ within $k$ prefix transpositions of the 
      identity.}
      \begin{footnotesize}
      $$
      \begin{array}{rrrrrrrrrrrc}
      k &\ 1\ &\ 2\ &\ 3\ &\ 4\ &\ 5\ &\ 6\ &\ 7\ &\ 8\ &\ 9\
      &10&\text{\OEISref}\\\hline 1& 1& 2& 4& 7& 11& 16& 22& 29& 37&
      46& \text{\OEISlink{A000124}}\\[0.5ex]
      &\multicolumn{10}{c}{\nc0 + \nc2  }&\\[1ex]
      2& 1& 2& 6& 21& 61& 146& 302& 561& 961& 1546& \text{\OEISlink{A228394}}\\[0.5ex]
      &\multicolumn{10}{c}{\nc0 + \nc2  + 2\nc3 + 6\nc4 }&\\[1ex]
      3& 1& 2& 6& 24& 116& 521& 1877& 5531& 13939&
      31156&\text{\OEISlink{A228395}}\\[0.5ex]
      &\multicolumn{10}{c}{\nc0 + \nc2 + 2\nc3 + 9\nc4 + 40\nc5 + 90\nc6
      }&\\[1ex]
      \end{array}
      $$
      \end{footnotesize}
      \end{table}
    
      \begin{table}[t]
      \caption{Number of permutations of length $n$ within $k$ block reversals of the
      identity.}
      \begin{footnotesize}
      $$
      \begin{array}{rrrrrrrrrrrc}
      k &\ 1\ &\ 2\ &\ 3\ &\ 4\ &\ 5\ &\ 6\ &\ 7\ &\ 8\ &\ 9\
      &10&\text{\OEISref}\\\hline
      1&1&2&4&7&11&16&22&29&37&46&\text{\OEISlink{A000124}}\\[0.5ex]
      &\multicolumn{10}{c}{{n\choose 0}+{n\choose 2}}&\\[1ex]
      2&1&2&6&22&63&145&288&516&857&1343&\text{\OEISlink{A228396}}\\[0.5ex]
      &\multicolumn{10}{c}{8{n\choose 0}-3{n\choose 1}+{n\choose 2}+4{n\choose 3}}&\\[1ex]
      3&1&2&6&24&118&534&1851&5158&12264&25943&\text{\OEISlink{A228397}}\\[0.5ex]
      &\multicolumn{10}{c}{\text{ \scriptsize $318\nc0 -214\nc1 +131\nc2 -61\nc3
      +20\nc4 +70\nc5 +35\nc6$  }}&\\[1ex]\end{array}
      $$
      \end{footnotesize}
      \end{table}
    
      \begin{table}[t]
      \caption{Number of permutations of length $n$ within $k$ prefix reversals of the
      identity.}
      \begin{footnotesize}
      $$
      \begin{array}{rrrrrrrrrrrc}
      k &1&2&3&4&5&6&7&8&9&10&\text{\OEISref}\\\hline
      1&1&2&3&4&5&6&7&8&9&10&\text{\OEISlink{A000027}}\\[0.5ex]
      &\multicolumn{10}{c}{{n\choose 1}}&\\[1ex]
      2&1&2&5&10&17&26&37&50&65&82&\text{\OEISlink{A002522}}\\[0.5ex]
      &\multicolumn{10}{c}{2 \nc0 -1 \nc1 + 2\nc2  }&\\[1ex]
      3&1&2&6&21&52&105&186&301&456&657&\text{\OEISlink{A228398}}\\[0.5ex]
      &\multicolumn{10}{c}{ -3 \nc0 + 3 \nc1 - 2\nc2 + 6\nc3}&\\[1ex]
      \end{array}
      $$
      \end{footnotesize}
      \end{table}

      \begin{table}[t]
      \caption{Number of permutations of length $n$ within $k$ cut-paste moves of the
      identity.}
      \begin{footnotesize}
      $$
      \begin{array}{rrrrrrrrrrrc}
      k &1&2&3&4&5&6&7&8&9&10&\text{\OEISref}\\\hline
      1&1&2&6&16&35&66&112&176&261&370&\text{\OEISlink{A060354}}\\[0.5ex]
      &\multicolumn{10}{c}{\nc1 + 3\nc3 }&\\[1ex]
      2&1&2&6&24&120&577&2208&6768&17469&39603&\text{\OEISlink{A228399}}\\[0.5ex]
      &\multicolumn{10}{c}{-18\nc0 + 45\nc1 - 61\nc2 + 70\nc3 - 53\nc4 + 88\nc5
      + 107\nc6}&\\[1ex]
      3&1&2&6&24&120&720&5040&36757&223898&1055479&\text{\OEISlink{A228400}}\\[0.5ex]
      &\multicolumn{10}{c}{508264\nc0 - 280036\nc1 + 140012\nc2 - 57622\nc3 +
      13839\nc4}&\\[1ex]
      &\multicolumn{10}{c}{+ 4136\nc5-5368\nc6 + 531\nc7 + 21125\nc8 +
      12615\nc9} \\ 
      \end{array}
      $$
      \end{footnotesize}
      \end{table}

      \begin{table}[t]
      \caption{Number of permutations of length $n$ within $k$ block interchanges of the
      identity.}
      \begin{footnotesize}
      $$
      \begin{array}{rrrrrrrrrrrc}
      k &1&2&3&4&5&6&7&8&9&10&\text{\OEISref}\\\hline
      1&1&2&6&16&36&71&127&211&331&496&\text{\OEISlink{A145126}}\\[0.5ex]
      &\multicolumn{10}{c}{{n\choose 0}+{n\choose 2}+2{n\choose 3}+{n\choose
      4}}&\\[1ex]
      2&1&2&6&24&120&540&1996&6196&16732&40459&\text{\OEISlink{A228401}}\\[0.5ex]
      &\multicolumn{10}{c}{{n\choose 0}+{n\choose 2}+2{n\choose 3}+9{n\choose
      4}+44{n\choose 5}+85{n\choose 6}+70{n\choose 7}+21{n\choose 8}}&\\
      \end{array}
      $$
      \end{footnotesize}
      \end{table}

\cleardoublepage
\typeout{******************}
\typeout{**  Chapter 5   **}
\typeout{******************}

  \chapter{Fixed-Length Patterns}
  \label{chap:fixpat}

    The set of all permutations, equipped with the pattern ordering, forms an
    infinite graded poset. While much research in this area (and within this
    dissertation) focuses on infinite downsets of this poset, this chapter focuses
    on finite subsets. In particular, we examine the downset induced by a single
    permutation, and investigate the number of distinct patterns.

    In 2003, Herb Wilf raised the question of finding the maximum number of
    distinct patterns which could be contained within a single permutation of
    length $n$, and classifying those permutations which maximize this number. 
    In~\cite{Flynn2007}, the authors showed that the maximum number of patterns
    for a length $n$ pattern is asymptotic to $2^n$, and provided a construction
    which achieves this number. 

    In this chapter we examine the number of distinct patterns of a
    \emph{specified length} which can be contained within a permutation. 
    In the language of posets, Wilf's question asks to find which permutations
    which maximize the size of their downset, while here we seek to maximize the
    \emph{width} of the downset. This chapter can be divided into two parts: in
    the first, we examine the number of $(n-1)$-patterns contained in a random
    permutation of length $n$, and obtain the expectation and variance for this statistic
    by extending a 1945 result of Kaplansky and Wolfowitz~\cite{Kaplansky,
    Wolfowitz}. In the second part, we examine the number of patterns of any
    fixed size within a permutation, and provide a construction which maximizes
    this number. This chapter is based partly on~\cite{me-fixpat}.

  % =========================================================================== %
  \section{Large Patterns}

    The set of all permutations, equipped with the pattern ordering, forms an
    infinite partially ordered set (see Figure~\ref{prelim:fig:poset}). We focus
    here on the local properties of this poset, namely the number of patterns
    containing and contained in a given pattern. The more general topology of
    this poset was studied by McNamara and
    Steingr\'imsson~\cite{Steingrimsson2013}.

    The set of patterns contained within any fixed permutation forms a partially
    ordered set, in fact a finite downset of the full pattern poset.  To examine
    these downsets, we use a top-down approach: deleting entries one at a time
    from the permutation to obtain the full set of patterns.
    Figure~\ref{fixpat:fig:downsets} shows several examples of these
    downsets\index{downset}. 

    \begin{figure}[t] 
      \centering
      \begin{tikzpicture}
      [scale = .4]
        \node (a1) at (-8,9) {$1234$};
        \node (a2) at (-8,6) {$123$};
        \node (a3) at (-8,3) {$12$};
        \node (a4) at (-8,0) {$1$};
        \draw (a1) -- (a2) -- (a3) -- (a4);

        \node (b1) at (0,9) {$1243$};
        \node (b21) at (-2,6) {$132$};
        \node (b22) at (2,6) {$123$};
        \node (b31) at (-2,3) {$21$};
        \node (b32) at (2,3) {$12$};
        \node (b4) at (0,0) {$1$};
        \draw (b1) -- (b21) -- (b31) -- (b4);
        \draw (b1) -- (b22) -- (b32) -- (b4);
        \draw (b21) -- (b32);

        \node (c1) at (12,9) {$2413$};
        \node (c21) at (7,6) {$312$};
        \node (c22) at (10,6) {$213$};
        \node (c23) at (14,6) {$132$};
        \node (c24) at (17,6) {$231$};
        \node (c31) at (10,3) {$21$};
        \node (c32) at (14,3) {$12$};
        \node (c4) at (12,0) {$1$};
        \draw (c1) -- (c21) -- (c31) -- (c4);
        \draw (c1) -- (c22) -- (c31) -- (c4);
        \draw (c1) -- (c23) -- (c32) -- (c4);
        \draw (c1) -- (c24) -- (c32) -- (c4);
        \draw (c21) -- (c32) --(c22);
        \draw (c23) -- (c31) --(c24);
      \end{tikzpicture}
    \caption{Downsets of 1234, 1243, and 2413} \label{fixpat:fig:downsets}
    \end{figure}
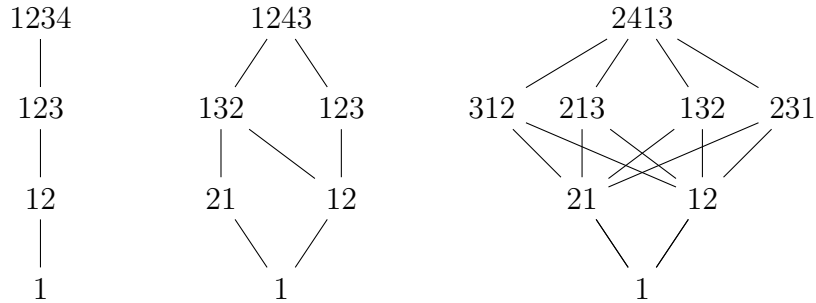

  \subsection{Definitions and Notation}

    It will be convenient to establish some machinery for dealing with large
    patterns. Fix $n \geq 2$, let $\pi$ be a permutation of length $n$, and let $\sg$ be an
    $(n-1)$-permutation. If $\sg$ is contained as a pattern within $\pi$, then it
    follows that $1$ can be obtained by deleting one entry from $\pi$, and
    relabelling with respect to order. Similarly, it follows that $\pi$ can be
    obtained by inserting an appropriate entry into $\sg$. We formalize these ideas
    with the following pair of definitions.

    \begin{definition} \label{fixpat:def:del}
      For any permutation $\pi \in \S_n$, define the function 
      $\del_\pi: [n] \ra \S_{n-1}$, where $\del_\pi(i)$ is the permutation
      obtained by deleting the $i$th entry of $\pi$, and standardizing the
      remaining entries. Let $\del_\pi = \{ \del_\pi(i) : i \in [n]\}$ denote the
      image of $\del_\pi$. 
    \end{definition}

    \begin{definition} \label{fixpat:def:ins}
      For any permutation $\sg \in \S_{n-1}$, define the function
      $\ins_\sg: [n] \times [n] \ra \S_n$, where $\ins_\sg(i,j)$ is the permutation
      obtained by inserting the entry $j - 1/2$ immediately to the left of the
      $i$th entry of $\sg$, and then standardizing the entries. 
      Let $\ins_\sg = \{\ins_\sg(i,j) : i, j \in [n]\}$ denote the image of
      $\ins_\sg$. 
    \end{definition}

    Letting $\pi$ and $\sg$ be an $n$- and $(n-1)$-permutation, respectively, 
    it follows from their definitions that these functions that $\del_\pi$ is the
    set of all $n-1$-patterns contained in $\pi$, and $\ins_\sg$ is the set of
    permutations of length $n$ which contain $\sg$. In addition, these functions satisfy the
    following inverse relationship:
    $$ \del_{\ins_\sg(i,j)}(i) = \sg \text{ and } \ins_{\del_{\pi}(i)}(i, \pi_i).$$

  % =========================================================================== %
  \section{Plentiful Permutations}
    
      Fix $n \in \Zp$, and let $\pi$ be a permutation of length $n$. Since every pattern
      within $\pi$ can be obtained by deleting elements of $\pi$ one by one,
      the relationship between $\del_\pi$ and $\pi$ can be applied iteratively to
      understand the full downset of $\pi$. It follows directly from the definition
      that $|\del_\pi| \leq n$, and that $|\del_\pi| = n$ if and only if
      $\del_\pi$ is a one-to-one function, i.e., $\del_\pi(i) = \del_\pi(j)$ if
      and only if $i=j$. Before investigating further, we introduce another pair
      of definitions. 

      \begin{definition} \index{permutation!plentiful}
      \label{fixpat:def:plentiful}
        Let $\pi$ be a permutation of length $n$. Say that $\pi$ is \emph{plentiful} if it
        contains $n$ distinct $(n-1)$-patterns. Equivalently, $\pi$ is plentiful
        if and only if $\delpi$ is a one-to-one function. 
      \end{definition}
      
      \begin{definition} \index{bond}
        Let $\pi = \pi_1p_2 \dots \pi_n$ be a permutation, and let $i \in [n-1]$. Say
        that the pair $(\pi_i, \pi_{i+1})$ is a \emph{bond}, of entries of $\pi$ if
        $\pi_i - \pi_{i+1} = \pm 1$. We say that the sequence $(\pi_i, \pi_{i+1}, \dots
        \pi_{i+k-1})$ is a \emph{run} of length $k$ if, for $1 \leq j \leq k-2$,
        the pair $(\pi_{i+j},\pi_{i+j+1})$ is a bond. Denote by $\beta(\pi)$ the
        number of bonds in $\pi$. 
      \end{definition}
    
      Note that runs are necessarily either increasing or decreasing, and that a run
      of length $k$ contains $k-1$ bonds. We can now establish a fundamental
      relationship between bonds and $(n-1)$-patterns. 

      \begin{lemma} \label{fixpat:lem:bonds}
        Let $\pi = \pi_1 \pi_2 \dots \pi_n$. For any $j,k \in [n]$ with $j \neq
        k$, $\del_\pi(i) = \del_\pi(j)$ if and only if $\pi_j$ and $\pi_k$ are
        part of the same run. 
      \end{lemma}
      \begin{proof}
        The forward direction is clear, since removing any element of a run
        simply results in a shorter run. 

        The reverse implication takes a bit more work. Suppose that there exist
        $j,k$ with $1 \leq j < k \leq n$ and $\del_\pi(j) = \del_\pi(k)$. We
        proceed by induction on $k-j$. 

        For the base case, suppose that $k =j+ 1$. Assume first that $\pi_j <
        \pi_{j+1}$, and consider the $j$th entry of $\delpi(j)=\delpi(j+1)$. By
        the definition of $\del$, the $j$th entry of $\delpi(j)$ is
        $\pi_{j+1}-1$, and the same entry in $\delpi(j+1)$ is $\pi_j$. Therefore,
        we see that $\pi_{j+1}-1=\pi_j$, which means that $(\pi_j,\pi_{j+1})$ is
        a bond. Again, the case where $\pi_{j+1}<\pi_j$ follows similarly. 

        Now assume by way of induction that the statement holds when $k=j+m-1$,
        and suppose there exists $1\leq j < k \leq n$ such that $k-j = m$ and
        $\delpi(j)=\delpi(k)$. Assume first that $\pi_j < \pi_{k}$.
        $\delpi(j)=\delpi(k)$ implies, in particular, that the $(k-1)$st entries
        on both sides of the equality are equal.  By definition, the $k-1$ entry
        of $\delpi(j)$ is $\pi_k-1$, while the $k-1$ entry of $\delpi(k)$ is
        either $\pi_{k-1}$ or $\pi_{k-1}-1$. The latter case would imply that
        $\pi_{k-1} = \pi_k$, a contradiction, and so it follows that $\pi_{k-1} =
        \pi_k$. 

        By what has already been proved, $\delpi(k-1) = \delpi(k)$ since these
        entries form a bond. But then $\delpi(j) = \delpi(k) = \delpi(k-1)$, and
        so by the induction hypothesis the entries $(\pi_j \pi_{j+1} \dots
        \pi_{k-1})$ form a run. Finally, $\pi_k - 1 = \pi_{k-1}$ implies that
        $(\pi_j \pi_{j+1} \dots \pi_{k-1} \pi_k)$ is a length $m$ run.  Once
        more, the case where $\pi_j > \pi_k$ follows similarly, and the lemma is
        proved.
      \end{proof}

      The size of the set $\delpi$ then depends entirely on $\beta(\pi)$, since
      each bond decreases by one the number of distinct $(n-1)$-patterns
      contained in $\pi$. This leads to the following theorem, and its immediate
      corollary. 

      \begin{theorem} \label{fixpat:thm:bonds}
        Let $\pi \in \S_n$. Then $|\delpi| = n - \beta(\pi)$. 
      \end{theorem}

      \begin{corollary} \label{fixpat:cor:alldistinct}
        A permutation is plentiful if and only if it contains no bonds. 
      \end{corollary}

      Theorem~\ref{fixpat:thm:bonds} also provides a simple proof of the
      following local property of the permutation pattern poset \index{pattern
      poset}. 

      \begin{corollary} \label{fixpat:cor:n2+1} 
        If $\sg \in \S_{n-1}$, then $|\ins_\sg| = n^2 - 2n + 2 = (n-1)^2 + 1$. In
        other words, every permutation of length $n$ is contained in exactly $n^2 +1$
        $(n+1)$-permutations. 
      \end{corollary}
      \begin{proof}

        By definition, the set $\ins_\sg = \{ins_\sg(j,k) : 1 \leq j , k \leq
        n\}$, so we see that $|\ins_\sg| \leq n^2$.

        Now, a permutation $\pi \in \S_n$ is contained in $\ins_\sg$ more than once
        exactly when $\sg$ can be obtained in more than one way by deleting a entry
        of $\pi$. It follows that $\sg$ is contained in a permutation $\pi \in \S_n$
        more than once exactly when $\ins_\sg(j,k) = \ins_\sg(j\p,k\p)$ where
        $(j,k) \neq (j\p,k\p)$. By the lemma, this happens exactly when the $j$th
        entry of
        $\ins_\sg(j,k)$ is a part of the same run as the $j\p$ entry of
        $\ins_\sg(j\p,k\p)$. We can prevent this from occurring by never inserting
        an element just to the right and directly above or below an existing
        element of $\sg$, as this ensures that any new bonds can be created in
        exactly one way. 

        This eliminates exactly $2(n-1)$ choices for inserting an entry into $\sg$,
        and so therefore $|\ins_\sg| = n^2 - 2(n-1) = (n-1)^2 +1$, and the proof
        is complete.  
      \end{proof}

  % =========================================================================== %
  \section{Distribution of the Number of Patterns}
    \label{fixpat:sec:delpi}

    We now consider let $\pi$ be a (uniformly) randomly \index{random} chosen
    permutation of length $n$, and examine the distribution of the statistic
    $|\delpi|$. The correlation presented in Theorem~\ref{fixpat:thm:bonds}
    allows us to investigate this distribution by analyzing the distribution of
    bonds. This distribution has been examined previously in other contexts, most
    notably by Kaplansky and Wolfowitz~\cite{Kaplansky, Wolfowitz}. In this
    section we extend their asymptotic results by finding \emph{exact} values for
    the expectation and variance of $\beta(\pi)$, and therefore of $|\delpi|$. 

    Throughout this section, fix $n$ and let $\dtan$ and $\btan$ be random
    variables denoting the number of distinct $(n-1)$-patterns and the number of
    bonds in a random permutation of length $n$, respectively.  Our primary tool
    in this investigation will be multivariate generating functions, but first we
    note that $\Ex{\dta}$ can be obtained directly using results from the
    previous section. \index{generating function!multivariate}

    \begin{proposition} \label{fixpat:prop:easyexpectation}
      The expectation of $\dta$ is equal to $n - \frac{2(n-1)}{n}$, which
      approaches $n-2$ as $n$ increases. 
    \end{proposition}
    \begin{proof}
      By the definition of expectation, we have 
      $$ \Ex{\dta} = \frac{\sum_{\pi \in \S_n} |\delpi| }{n!}. $$
      The proposition then follows immediately from
      Corollary~\ref{fixpat:cor:n2+1} and the identity
      $$ (n^2 - 2n + 2) (n-1)! = \left(n - \frac{2(n-1)}{n}\right) n!.$$
    \end{proof}

  \subsection{Generating Functions}

    Generating functions allow us to go several steps further, and obtain
    higher moments for the distributions of these variables. It follows from
    Theorem~\ref{fixpat:thm:bonds} and the linearity of expectation
    \index{expectation!linearity} that 
    $$ \Ex{\dta} = n - \Ex{\bta}.$$
    Therefore we can translate the distribution of $\bta$ to that of $\dta$. 
    We start by building a multivariate generating function which keeps track
    of the distribution of bonds throughout all permutations. We use a method
    similar to the cluster method of Goulden and Jackson~\cite{gouldenjackson1,
    gouldenjackson2}, described by Noonan and Zeilberger~\cite{Noonan1999}. Note
    that this generating function converges nowhere, but still yields useful
    algebraic information. 

    \begin{theorem} \label{fixpat:thm:genfcn}
      Let $a_{n,k}$ be the number of permutations of length $n$ which contain
      exactly $k$ bonds, and let $a_{0,0} = 1$. Then the numbers $a_{n,k}$ have
      the following generating function
      $$ \sum_{n \geq 0} \sum_{k \geq 0} a_{n,k} z^n u^k = 
        \sum_{m \geq 0} m! \left(z + \frac{2z^2(u - 1)}{1 - z(u-1)}\right)^m.$$
      Denote this function by $f(z,u)$. 
    \end{theorem}
    \begin{proof}
      First we construct a related generating function, then translate it into
      ours using the technique of inclusion-exclusion. 
      Say that a bond in a permutation can be arbitrarily \emph{marked}, and then a  
      \emph{marked permutation} is one in which each bond is either marked or
      unmarked.  Let $b_{n,k}$ be the number of permutations of length $n$ which contain
      exactly $k$ marked bonds. For example, $b_{n,0} = n!$, since every
      permutation can be written with no bonds marked, and no permutation is
      counted more than once. Similarly, $b_{n,n-1} = 1$, since the only
      marked permutation with $n-1$ marked bonds is the decreasing permutation
      with all of its bonds marked. 

      Let 
      $$ g(z,u) := \sum_{n \geq 0}\sum_{k \geq 0} b_{n,k}z^n u^k. $$ 
      This generating
      function is easier to construct, as we can build a permutation of length $n$ with
      $k$ marked bonds by first specifying our marked runs, then permuting
      these runs with the remaining entries. The benefit to this method is that
      we don't have to worry about bonds forms between these runs, as we have
      already specified which ones are marked. A marked run of length $j$
      can be either ascending or descending, and contains $j-1$ bonds. 
      It follows that 

      $$ g(z,v) = \sum_{m \geq 0}
          m! \left(z + \frac{2z^2v}{1 - zv} \right)^m.$$

      Now, we can use this generating function to obtain $f(z,u)$. The variable
      $v$ keeps track of marked bonds, while $u$ keeps track of all bonds.
      Since every bond can either be marked or unmarked, it follows that by
      substituting $u$ for $v + 1$ we can translate $f(z,u)$ to $g(z,v)$.
      Therefore, we have the relation $f(z,v+1) = g(z,v)$, from which we see
      that 
      $$ f(z,u) = g(z,u-1) = 
        \sum_{m \geq 0} m! \left(z + \frac{2z^2(u-1)}{1 - z(u-1)} \right)^m.$$
    \end{proof}

    The following corollary is immediate, and follows from the relationship
    between $\dta$ and $\bta$. 
      
    \begin{corollary} \label{fixpat:cor:gfn2}
      Let $d_{n,k}$ denote the number of permutations of length $n$ containing
      exactly $k$ distinct $(n-1)$ patterns, and let $d_{0,0} = 1$. Then 
      $$ h(z,u) :=  \sum_{n \geq 0}\sum_{ k \geq 0} h_{n,k} z^n u^k = 
        \sum_{m \geq 0}  m! \left( zu + 
        \frac{2zu^2 (1/u - 1)}{1 - zu(1/u - 1)}\right)^m.$$
    \end{corollary}
    \begin{proof}
      Since $\dta = n - \bta$, it follows that $h(z,u) = f(zu, 1/u)$. 
    \end{proof}

    The remainder of this section will consist of the analysis of the function
    $F(z,u)$, and the translation of this analysis into facts about
    permutations. 
    First, we compute the number of permutations which have no bonds (and
    are therefore plentiful). 

    \begin{proposition} \label{fixpat:cor:chessboard}
      Let $b_n$ be the number of permutations of length $n$ with no bonds. Then 
      $$ \begin{aligned} 
       \sum_{n \geq 0} b_n z^n 
         &= \sum_{m \geq 0} m! z^m \frac{(1 - z)^m}{(1 + z)^m} \\
         &= 1 + z + 2z^4 + 14z^5 + 90z^6 + 646z^7 + 5242z^8 + \dots 
      \end{aligned} $$
    \end{proposition}
    \begin{proof}
      This follows immediately by setting $u = 0$ in $f(z,u)$. 
    \end{proof}

    The numbers $b_n$ in Corollary~\ref{fixpat:cor:chessboard} are 
    \OEIS{A002464}. These numbers are also equal to the number of ways of
    placing $n$ non-attacking kings on an $n \times n$ chessboard with one king
    per each row and column, as can be seen by plotting the permutations. It
    was shown in~\cite{Tauraso2006} that this sequence is asymptotic to
    $n!/e^2$, and so Corollary~\ref{fixpat:cor:alldistinct} implies the
    following corollary. 

    \begin{corollary}
      The probability that a randomly selected $n$ permutation is plentiful tends
      to $1/e^2$ as $n$ tends to infinity. 
    \end{corollary}

    In addition to exact results, we can use the function $f(z,u)$ to determine
    the \emph{expected} \index{expectation} number of bonds within a randomly
    selected permutation of length $n$, Using techniques described in
    Chapter~\ref{chap:prelim} and in \cite{flajolet}.

    \begin{theorem} \label{fixpat:thm:exvar}
      The expectation and variance of the random variable $\btan$ are as
      follows:
      $$ \begin{aligned}
          \Ex{\btan} &= 2\frac{(n-1)}{n} \\
          \Var{\btan} &= 4\frac{(n-2)^2}{n(n-1)} + 2\frac{n-1}{n} -
          4\frac{(n-1)^2}{n^2} . 
          \end{aligned} $$
    \end{theorem}
    \begin{proof}
      The expectation is obtained by taking the partial derivative with respect
      to $u$, then plugging in $u = 0$ as shown below. 

      $$ \begin{aligned}
        \sum_{n\geq 0}\Ex{\btan}z^n 
          &= \frac{\partial_u f(z,u) \big|_{u=0}}{n!} \\
          &= \sum_{n \geq 0} 2 (n-1)! \cdot (n-1) z^n. 
        \end{aligned}$$
      
      The second factorial moment $\Ex{\btan(\btan - 1)}$ can be computed from
      the generating function as follows:
      
      $$ \sum_{n \geq 0} \Ex{\btan(\btan - 1)}z^n 
        = \frac{\partial_u^2 f(z,u) \uisone}{n!}.$$
      
      The variance can then be computed using linearity of expectation:
      \index{expectation!linearity}
      $$ \Var{\btan^2} = \Ex{\btan^2} - \Ex{\btan}^2 
        = \Ex{\btan(\btan - 1)} + \Ex{\btan} - \Ex{\btan}^2.$$
      
      From here, a tedious and technical computation finishes the proof. 
    \end{proof}

    Higher moments can be computed iteratively. 
    The relationship between the variables $\btan$ and $\dtan$ immediately
    provides the corresponding expectation and variance for $\dtan$. Taking the
    limit as $n \ra \infty$ gives asymptotic values for this distribution,
    which leads to the results found in~\cite{Kaplansky, Wolfowitz}. We
    summarize these ideas in the following corollaries. 

    \begin{corollary}\label{fixpat:cor:exvar}
      The expectation and variance for the variable $\dtan$ are as follows:
      $$ \begin{aligned} 
         \Ex{\btan} &= n - \frac{2(n-1)}{n} \\
          \Var{\btan} &= 4\frac{(n-2)^2}{n(n-1)} + 2\frac{n-1}{n} -
          4\frac{(n-1)^2}{n^2}. 
          \end{aligned} $$
    \end{corollary}

    \begin{corollary}
      For large $n$, we have that 
      $$ \Ex{\btan} \sim n - 2 \text{ \quad and \quad }
        \Var{\btan} \sim 2.$$
    \end{corollary}

  \section{Patterns of Other Sizes}
    \label{fixpat:sec:othersizes}
    
    In this section, we examine the number of distinct $(n-k)$-patterns contained
    in a permutation of length $n$. For a given permutation of length $n$ $\pi$, $\delpi$ denotes
    the image of the function $\delpi$, which is exactly the set of
    $(n-1)$-patterns contained in $\pi$. The following definitions generalize
    the Definitions~\ref{fixpat:def:del} and~\ref{fixpat:def:plentiful}. 

    \begin{definition} \label{fixpat:def:delpik}
      Let $S = \{i_1, i_2, \dots i_k\} \subseteq [n]$, with $i_1 < i_2 < \dots
      < i_k$. We denote by $\delpi(S)$ the permutation obtained by deleting the
      entries in positions $i_1, \dots i_k$, and standardizing the remaining
      entries. Denote by $\delpik$ the set of all permutations which can be
      obtained by deleting $k$ entries from $\pi$ and standardizing. 
    \end{definition}

    \begin{definition} \label{fixpat:def:kplentiful}
      Say that a permutation of length $n$ $\pi$ is \emph{$k$-plentiful} if it has the maximal
      number of distinct $(n-k)$-patterns, i.e., if 
      $$|\delpik| = \binom{n}{k}.$$
    \end{definition}

  \subsection{Characterizing k-plentiful Permutations}

    We seek to characterize those permutations which are $k$-plentiful, for an
    arbitrary $k \in [n]$. In Section~\ref{fixpat:sec:delpi} we found that a
    permutation is plentiful if and only if it contains no bonds. By generalizing
    our notion of bonds, we obtain an analogous result here. 

    \begin{definition} \label{fixpat:def:gap}
      Let $\pi = \pi_1 \pi_2 \dots \pi_n \in \S_n$. For any two integers $i,j
      \in [n]$, define the \emph{distance} $d_\pi(i,j)$ between $i$ and $j$ to be 
      $$ d_\pi(i,j) = |i - j| + |\pi_i - \pi_j|.$$
      The \emph{minimum gap} of $\pi$, denoted by $\Gam(\pi)$, is defined to be
      the minimum distance between any two entries. Formally:
      $$ \Gam(\pi) = \min\{ d_\pi(i,j) : 1 \leq i, j \leq n \}.$$
    \end{definition}

    If we plot a permutation $\pi$, then the function $d_\pi$ is just the usual
    taxicab metric on $\{(i, \pi_i) : 1 \leq i \leq n\} \subset \mathbb{R}^2$. 
    It is easy to see that $(\pi_i, \pi_j)$ is a bond if and only if $d_\pi(i,j)
    = 2$. It follows then that $\pi$ is plentiful if and only if $\Gam(\pi) \geq 3$.
    This idea allows us to generalize Corollary~\ref{fixpat:cor:alldistinct}.  We
    start with one more definition, and a simple lemma which will prove useful. 

    \begin{definition} \label{fixpat:def:span}
      Let $\pi = \pi_1 \pi_2 \dots \pi_n \in \S_n$ and let $i,j \in [n]$ with $i
      < j$. The \emph{span} of the indices $i$ and $j$, denoted
      $\pspan_\pi(i,j)$, is defined as the set of indices corresponding to
      entries which are between $(i,\pi_i)$ and $(j,\pi_j)$ either horizontally
      and vertically. Formally, when $\pi_i < \pi_j$ we have
      $$ 
      \pspan_\pi(i,j) = \{k : i < k < j\} \cup \{k : \pi_i < \pi_k < \pi_j \}.
      $$
      The case when $\pi_i > \pi_j$ is defined analogously. 
    \end{definition}

    \begin{lemma} \label{fixpat:lem:span}
      Let $\pi \in \S_n$ be such that $\Gam(\pi) = m$, and let $i,j$ be such that
      $d_\pi(i,j) = m$. Then $|\pspan_\pi(i,j)| = m-2$. Further, deleting one
      entry can reduce the minimum gap by at most one, i.e., $\Gam(\delpi(k))
      \geq k-1$ for all $k \in [n]$. 
    \end{lemma}
    \begin{proof}
      Clearly $|\pspan_\pi(i,j)| \leq m-2$, since otherwise this would contradict
      $\Gam(\pi) = m$. The only way in which $|\pspan_\pi(i,j)|$ could be less
      than $k-2$ is if there exists an entry $\pi_k$ which lies between
      $(i,\pi_i)$ and $(j,\pi_j)$ both vertically and horizontally. However, this
      would imply that $d_\pi(i,m) < d_\pi(i,j) = m-2$, which contradicts the
      minimality of $d_\pi(i,j)$. Therefore, $|\pspan_\pi(i,j)| = m-2$. 

      For the second part, note that the only way that deleting a single entry
      could reduce the minimum gap by more than one is if that entry lies between
      two minimally separated entries. However, we have just seen that no such
      entry exists. 
    \end{proof}

    We are now able to give a partial characterization of the $k$-plentiful
    permutations in the following generalization of
    Corollary~\ref{fixpat:cor:alldistinct}. 

    \begin{theorem} \label{fixpat:thm:kplentiful}
      A permutation $\pi$ is $k$-plentiful if and only if $\Gam(\pi) \geq k+2$. 
    \end{theorem}
    \begin{proof}
      First let $\pi = \pi_1 \pi_2 \dots \pi_n$ be a $k$-plentiful permutation, and
      assume by way of contradiction that $\Gam(\pi) = m < k+2$. Let $i<j$ be
      such that $d_\pi(i,j) = m$. By Lemma~\ref{fixpat:lem:span}, we have that
      $\pspan_\pi(i,j) = \{s_1, s_2, \dots s_{m-2}\}$. Let $\sg =
      \delpi(\pspan_\pi(i,j) \in \S_{n-m+2}$, the permutation obtained by
      removing the entries with indices $s_i$ and standardizing the remaining
      entries. If follows then that $\Gam(\sg) = 2$ and so $\sg$ has a bond
      $(\sg_i, \sg_j)$ and is therefore not plentiful. It follows then that
      $\del_\sg(i) = \del_\sg(j)$, and so there are two sets of indices $S$ and
      $S\p$ for which $\del_\pi(S) = \del_\pi(S\p)$. Therefore $|\delpik| <
      \binom{n}{k}$, contradicting the plentifulness of $\pi$. 

      For the other direction, we proceed using induction.
      We have already shown that the theorem holds when $k=1$
      (Corollary~\ref{fixpat:cor:alldistinct}), so let $k>1$ and assume that the
      statement holds for all positive integers less than $k$.
      Let $\pi \in \S_n$ be such that $\Gam(\pi) \geq k+2$. We know by induction
      that this permutation is $m$-plentiful for all $1 \leq m < k$. 

      Suppose by way of contradiction that $\sg \in \S_{n - k}$ can be obtained
      by deleting two different sets of entries from $\pi$. That is, suppose that
      there exist $A = \{a_1, a_2, \dots a_k\} \neq B = \{b_1, b_2, \dots
      b_k\}$, with $a_i < a_j$ and $b_i < b_j$ for $i < j$, such that $\delpi(A)
      = \delpi(B) = \sg$. Claim that $A \cap B = \emptyset$. To see this, suppose
      that $a_i = b_j$, and note that since $A - \{a_i\} \neq B - \{b_j\}$, 
      $\sg$ is contained in $\delpi(a_i)$ in two different ways. However, by
      Lemma~\ref{fixpat:lem:span}, $\Gam(\delpi(a_i)) \geq k + 1$, and so by
      induction $\delpi(a_i)$ is $(k-1)$-plentiful, a contradiction. Therefore $A$ and
      $B$ must be disjoint. 

      Assume without loss of generality that $a_1 < b_1$. Let $j \in [n]$ be the
      smallest integer such that $j > a_1$ but $j \notin A$. Since $\delpi(A) =
      \delpi(B) = \sg = \sg_1 \sg_2 \dots \sg_{n-k}$, it follows that the
      entries $p_{a_1}$ will move to fulfill the role of $\sg_{a_1}$ once the $B$
      entries are deleted. However, the entry $a_j$ will also move to fulfill
      this role once the $A$ entries are deleted. However, this implies that
      every entry in the span of $\pi_{a_1}$ and $\pi_{j}$ must be deleted, but
      there must be at least $k$ such entries by Lemma~\ref{fixpat:lem:span}.
      Therefore, $A$ must contain $a_1$ and $k$ additional entries, contradicting
      $|A| = k$ and proving the theorem. 
    \end{proof}

  \subsection{Constructing k-plentiful Permutations}

    It is not immediately obvious that there exist permutations with arbitrarily
    large minimum gaps. In~\cite{Flynn2007}, the authors constructed a
    permutation of length $(k-1)^2$ which has a minimum gap equal to $k$. 
    We conclude this section with a construction that gives a slightly smaller
    permutation which achieves the same gap size, and prove that this
    construction is the best possible. 

    \begin{figure}[t]
      \centering
      \begin{tikzpicture}
        [scale = .3, line width = .8pt]
        \draw (0,15) -- (0,0) -- (15,0);
        \foreach \num in {2, 4,..., 14}
        {
          \draw (\num, 0) -- (\num, -.3);
          \draw (0, \num) -- (-.3, \num);
        }
        \foreach \y [count = \x] in {3,6,1,4,7,2,5}
          \draw[fill = black] (2*\x,2*\y) circle (4mm);
      \end{tikzpicture} \hspace{4pc}
      \begin{tikzpicture}
        [scale = .3, line width = .8pt]
        \draw (0,15) -- (0,0) -- (15,0);
        \foreach \num in {1,..., 14}
        {
          \draw (\num, 0) -- (\num, -.3);
          \draw (0, \num) -- (-.3, \num);
        }
        \foreach \y [count = \x] in {4,8,12,1,5,9,13,2,6,10,14,3,7,11}
          \draw[fill = black] (\x,\y) circle (2mm);
      \end{tikzpicture}
      \caption{The plots of the permutations $\Theta^{(4)}$ and $\Theta^{(5)}$.}
      \label{fixpat:fig:thetan}
    \end{figure}
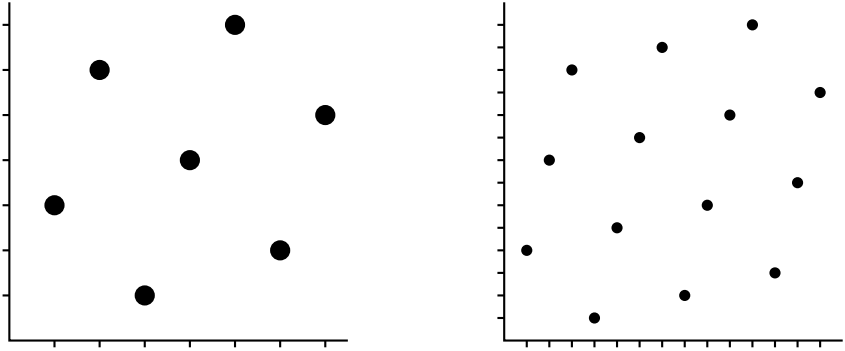

    \begin{definition}
      Let $\pi \in \S_{(k-1)^2}$ be defined by 
      $$\pi_{i(k-1) + j + 1} = i + j(k-1) + 1, \quad 0 \leq i,j \leq k-2.$$
      Then let $\Theta^{(k)} \in \S_{(k-1)^2 - 2}$ be defined by removing the first
      and last entries of $\pi$. 
    \end{definition}

    The permutations $\Theta^{(4)}$ and $\Theta^{(5)}$ are shown in
    Figure~\ref{fixpat:fig:thetan}. It is clear from the figure, and can be
    shown from the definition (with some tedious but simple calculation) that
    $\Gam(\Theta^{(k)}) = k$. It also follows that $\Theta^{(k)}$ is an
    involution, and its reverse is equal to its complement, so its orbit
    under the automorphism group of the pattern poset consists of only two
    elements. 

    % \begin{figure}[t]
    %   \centering
    %   \begin{tikzpicture}
    %     [scale = .3, line width = .8pt]
    %     \draw (0,15) -- (0,0) -- (15,0);
    %     \foreach \num in {2, 4,..., 14}
    %     {
    %       \draw (\num, 0) -- (\num, -.3);
    %       \draw (0, \num) -- (-.3, \num);
    %     }
    %     \foreach \y [count = \x] in {3,6,1,4,7,2,5}
    %       {\draw[fill = black] (2*\x,2*\y) circle (4mm);
    %       \draw (2*\x + 4, 2*\y - 1.33) -- (2*\x +2, 2*\y +3.33) -- 
    %             (2*\x - 4,  2*\y + 1.33) -- (2*\x - 2 ,2*\y -3.33) -- cycle;}
    %   \end{tikzpicture} \hspace{4pc}
    %   \caption{The tiling associated to $\Theta^{(4)}$.}
    %   \label{fixpat:fig:tiling}
    % \end{figure}

    By embedding a permutation $\pi$ into the plane, the function $d_\pi$ can be
    extended to the usual taxicab metric $d_1$ on $\mathbb{R}^2$. If $\pi$ has a
    minimum gap size of $k$, then $\pi$ defines a tiling of the plane with angled
    bricks of uniform size and centered on the points of $\mathbb{Z}^2$. It
    is clear that a minimal such permutation will correspond to a maximal tiling
    of this form, with the property that no two centers lie on the same
    horizontal or vertical line. There are exactly two such tilings, corresponding to
    the permutation $\Theta^{(k)}$ and its reverse.
    % (see Figure~\ref{fixpat:fig:tiling}). 
    We summarize this in the following theorem. 

    \begin{theorem} \label{fixpat:thm:tiling}
      The permutation $\Theta^{(k)}$ and its reverse are the shortest
      permutations with minimum gap size equal to $k$. 
    \end{theorem}

    We end this chapter with one last theorem, generalizing
    Theorem~\ref{fixpat:thm:bonds}.

    \begin{theorem} \label{fixpat:thm:gappairs}
      Let $\pi \in \S_{n}$ have $\Gam(\pi) = k + 1$, and let $p_k$ be the number
      of pairs $(i,j)$ such that $d_\pi(i,j) = k$. Then 
      $$ | \delpik | = \binom{n}{k} - p_k.$$
    \end{theorem}
    \begin{proof}
      Let $\pi \in \S_n$ be such that $\Gam(\pi) = k+1$, and let $i,j \in [n]$ be
      such that $d_\pi(i,j) = k+1$ (i.e., $|\pspan_\pi(i,j)| = k-1$). If we let
      $S = \pspan_\pi \cup i$ and $S\p = \pspan_\pi \cup j$, we see that
      $\delpi(S) = \delpi(S\p)$, and so 
      $$\delpik \leq \binom{n}{k} - p_k. $$
      To show equality, let $A = \{a_1, a_2, \dots a_k\} \neq B = \{b_1, b_2,
      \dots b_k\}$, with $a_i < a_j$ and $b_i < b_j$ when $i < j$, and suppose
      that $\delpi(A) = \delpi(B)$. 
      
      Claim that $|A \cap B| = k-1$, i.e., that the two sets differ by exactly
      one element. .  Suppose first that $a_1 \neq b_1$, and let $s$ be the
      smallest integer greater than $a_1$ such that $s \notin A$. Then, as in the
      proof of Theorem~\ref{fixpat:thm:kplentiful}, we have $d_\pi(a_1, s) = k+1$, and
      $A - {a_1} = B - {b_1} = \pspan_\pi(a_1, s)$.  In the case where $a_1 =
      b_1$, let $\pi\p = \delpi(a_1)$, $A\p = A - \{a_1\}$, and $B\p = B -
      \{b_1\}$. Since $\del_{\pi\p} (A\p) = \del_{\pi\p}(B\p)$, by
      Lemma~\ref{fixpat:lem:span} and Theorem~\ref{fixpat:thm:kplentiful} imply that that
      $\Gam(\pi\p) = k$. We now find that either $a_2 = b_2$ or $A\p - \{a_2\} =
      B\p - \{b_2\}$. Iterating this argument shows that the two sets differ by
      at most one element. 

      Finally, let $i,j$ be such that $a_i \in A - B$ and $b_j \in B - A$. It
      follows then that $ A-\{a_i\} = B - \{b_j\} - \{\pspan_\pi(i,j) \}$. But
      since their span has size $k-1$, their distance must be equal to $k+1$,
      an element in between them both horizontally and vertically would
      contradict the size of the minimum gap. Thus, each pair $i,j$ for which
      $d_\pi(i,j) = k+1$ reduces the number of $(n-k)$-patterns by exactly one,
      which completes the proof. 
    \end{proof}

% ====================================================================== %

% \addtocontents{toc}{\protect\addvspace{10pt}%
  % \noindent{\large \scshape Appendices}\protect\hfill\par\noindent\hrulefill}{}
% \addtocontents{toc}{\scshape \Large \noindent Appendices\hfill\par}{}

% \UnderlinedTocSection{Appendix}
% 
% \appendix
% 
% % \let\oldthepage\thepage
% % 
% % % ====================================================================== %
% \cleardoublepage
% % \pagenumbering{arabic}
% % \renewcommand*{\thepage}{A\arabic{page}}
% \include{app-dyck}
% 
% % ====================================================================== %
% \cleardoublepage
% % \pagenumbering{arabic}
% % \renewcommand*{\thepage}{B\arabic{page}}
% \include{app-probability}
% 
% % ====================================================================== %
% \cleardoublepage
% % \pagenumbering{arabic}
% % \renewcommand*{\thepage}{C\arabic{page}}
% \include{app-complex}
% 
% % ====================================================================== %
% % \renewcommand*{\thepage}{\oldthepage}
% % \pagenumbering{roman}
% \cleardoublepage
\backmatter
\bibliographystyle{acm}
\bibliography{library}

\begin{thebibliography}{10}

\bibitem{PermLab}
{\sc Albert, M.~H.}
\newblock Perm{L}ab: Software for permutation patterns, 2012.
\newblock Available online at \url{http://www.cs.otago.ac.nz/PermLab}.

\bibitem{Atkinson2005}
{\sc Albert, M.~H., and Atkinson, M.~D.}
\newblock Simple permutations and pattern restricted permutations.
\newblock {\em Discrete Math. 300}, 1-3 (2005), 15 pp.

\bibitem{Albert2010}
{\sc Albert, M.~H., Atkinson, M.~D., Brignall, R., Ru\v{s}kuc, N., Smith, R.,
  and West, J.}
\newblock Growth rates for subclasses of {A}v(321).
\newblock {\em Electron. J. Comb. 17}, 1 (2010), Research Paper 141, 16.

\bibitem{Claesson2013}
{\sc Albert, M.~H., Atkinson, M.~D., and Claesson, A.}
\newblock Isomorphisms between pattern classes.
\newblock arXiv:1308.3262 [math.CO], 11 pp.

\bibitem{Albert2012}
{\sc Albert, M.~H., Atkinson, M.~D., and Vatter, V.}
\newblock Inflations of geometric grid classes: Three case studies.
\newblock {\em Australas. J. Comb. 58}, 1 (2014), pp. 27--47.

\bibitem{GridClasses}
{\sc Albert, M.~H., Atkinson, M.~D., Vatter, V., Ru\v{s}kuc, N., and Bouvel,
  M.}
\newblock Geometric grid classes of permutations.
\newblock {\em Trans. Am. Math. Soc. 365\/} (2012), 5859--5881.

\bibitem{Flynn2007}
{\sc Albert, M.~H., Coleman, M., Flynn, R., and Leader, I.}
\newblock Permutations containing many patterns.
\newblock {\em Ann. Comb. 11}, 3-4 (2007), 265--270.

\bibitem{1324LowerBound}
{\sc Albert, M.~H., Elder, M., Rechnitzer, A., Westcott, P., and Zabrocki, M.}
\newblock On the {S}tanley-{W}ilf limit of 4231-avoiding permutations and a
  conjecture of {A}rratia.
\newblock {\em Adv. Appl. Math. 36}, 2 (2006), 96--105.

\bibitem{Linton2005}
{\sc Albert, M.~H., Linton, S., and Ru\v{s}kuc, N.}
\newblock The insertion encoding of permutations.
\newblock {\em Electron. J. Comb. 12\/} (2005), Paper 47, 31 pp.

\bibitem{Vatter2013}
{\sc Albert, M.~H., and Vatter, V.}
\newblock Generating and enumerating 321-avoiding and skew-merged simple
  permutations.
\newblock {\em Electron. J. Comb. 20}, 2 (2013), Paper 44, 11 pp.

\bibitem{Alpar2009}
{\sc Alpar-Vajk, K.}
\newblock A bound for the reversal distance of genome rearrangements.
\newblock {\em J. Math. Chem.\/} (2009), 941--945.

\bibitem{Arratia1999}
{\sc Arratia, R.}
\newblock On the {S}tanley-{W}ilf conjecture for the number of permutations
  avoiding a given pattern.
\newblock {\em Electron. J. Comb. 6\/} (1999), Note 1, 4 pp.

\bibitem{Atkinson2011}
{\sc Atkinson, M.~D., Ru\v{s}kuc, N., and Smith, R.}
\newblock Substitution-closed pattern classes.
\newblock {\em J. Comb. Theory, Ser. A 118}, 2 (2011), 317--340.

\bibitem{BabsonWest}
{\sc Babson, E., and West, J.}
\newblock The permutations {$123p_4\cdots p_m$} and {$321p_4\cdots p_m$} are
  {W}ilf-equivalent.
\newblock {\em Graphs Comb. 16}, 4 (2000), 373--380.

\bibitem{Backelin2007}
{\sc Backelin, J., West, J., and Xin, G.}
\newblock Wilf-equivalence for singleton classes.
\newblock {\em Adv. Appl. Math. 38}, 2 (2007), 133--148.

\bibitem{Bafna1998}
{\sc Bafna, V., and Pevzner, P.}
\newblock Sorting by transpositions.
\newblock {\em SIAM J. Discret. Math. 11}, 2 (1998), 224--240.

\bibitem{bloomvince}
{\sc Bloom, J., and Vatter, V.}
\newblock Two vignettes on full rook placements.
\newblock arXiv:1310.6073 [math.CO], 9 pp.

\bibitem{Bona1997}
{\sc B\'{o}na, M.}
\newblock Exact enumeration of {1342}-avoiding permutations: a close link with
  labeled trees and planar maps.
\newblock {\em J. Comb. Theory, Ser. A 80}, 2 (1997), 257--272.

\bibitem{Bona2007}
{\sc B\'{o}na, M.}
\newblock New records in {S}tanley–{W}ilf limits.
\newblock {\em Eur. J. Comb. 28}, 1 (Jan. 2007), 75--85.

\bibitem{Bona2010}
{\sc B\'{o}na, M.}
\newblock The absence of a pattern and the occurrences of another.
\newblock {\em Discret. Math. Theor. Comput. Sci. 13}, 2 (2010), 89--102.

\bibitem{BonaWalk}
{\sc B\'{o}na, M.}
\newblock {\em A walk through combinatorics}, third~ed.
\newblock World Scientific Publishing Co. Pte. Ltd., Hackensack, NJ, 2011.
\newblock An introduction to enumeration and graph theory, With a foreword by
  Richard Stanley.

\bibitem{BonaPerm}
{\sc B\'{o}na, M.}
\newblock {\em Combinatorics of permutations}, second~ed.
\newblock Discrete Mathematics and its Applications (Boca Raton). CRC Press,
  Boca Raton, FL, 2012.
\newblock With a foreword by Richard Stanley.

\bibitem{1324UpperBound}
{\sc B\'{o}na, M.}
\newblock A new upper bound for 1324-avoiding permutations.
\newblock arXiv:1207:2379 [math.CO], 6 pp.

\bibitem{Bona2012}
{\sc B\'{o}na, M.}
\newblock Surprising symmetries in objects counted by catalan numbers.
\newblock {\em Electron. J. Comb. 19}, 1 (2012), Paper 62, 12 pp.

\bibitem{Bona2009}
{\sc B\'{o}na, M., and Flynn, R.}
\newblock The average number of block interchanges needed to sort a permutation
  and a recent result of stanley.
\newblock {\em Inf. Process. Lett. 109}, 16 (July 2009), 927--931.

\bibitem{me-involutions}
{\sc B\'{o}na, M., Homberger, C., Pantone, J., and Vatter, V.}
\newblock Pattern avoiding involutions: Exact and asymptotic enumeration.
\newblock arXiv:1310.7003 [math.CO], 26 pp.

\bibitem{Brignall2007}
{\sc Brignall, R.}
\newblock Simplicity in relational structures and its application to
  permutation classes.
\newblock {\em PhD Thesis, Univ. St Andrews\/} (2007).

\bibitem{Brignall2008}
{\sc Brignall, R., Huczynska, S., and Vatter, V.}
\newblock Simple permutations and algebraic generating functions.
\newblock {\em J. Comb. Theory, Ser. A 115}, 3 (2008), 423--441.

\bibitem{Elizalde2013}
{\sc Burstein, A., and Elizalde, S.}
\newblock Total occurrence statistics on restricted permutations.
\newblock {\em Pure Math Appl. to appear\/}.

\bibitem{pantone2014}
{\sc Burstein, A., and Pantone, J.}
\newblock Two examples of unbalanced wilf-equivalence.
\newblock arXiv:1402.3842 [math.CO], 8 pp.

\bibitem{Chapman1999}
{\sc Chapman, R.}
\newblock Moments of dyck paths.
\newblock {\em Discrete Math. 204\/} (1999), 113--117.

\bibitem{Cheng2007}
{\sc Cheng, S.-E., Eu, S.-P., and Fu, T.-S.}
\newblock Area of catalan paths on a checkerboard.
\newblock {\em Eur. J. Comb. 28}, 4 (May 2007), 1331--1344.

\bibitem{Christie1996}
{\sc Christie, D.~A.}
\newblock Sorting permutations by block-interchanges.
\newblock {\em Inf. Process. Lett. 60}, 4 (1996), 165--169.

\bibitem{Claesson2012}
{\sc Claesson, A., Jel\'{\i}nek, V., and Steingr\'{\i}msson, E.}
\newblock Upper bounds for the {S}tanley-{W}ilf limit of 1324 and other layered
  patterns.
\newblock {\em J. Comb. Theory, Ser. A 119}, 8 (2012), 1680--1691.

\bibitem{Claesson2008}
{\sc Claesson, A., and Kitaev, S.}
\newblock Classification of bijections between 321- and 132-avoiding
  permutations.
\newblock {\em 20th Annu. Int. Conf. Form. Power Ser. Algebr. Comb. (FPSAC
  2008)}, 060005012 (2008), 495--506.

\bibitem{Cranston2007}
{\sc Cranston, D.~W., Sudborough, I.~H., and West, D.~B.}
\newblock Short proofs for cut-and-paste sorting of permutations.
\newblock {\em Discrete Math. 307}, 22 (2007), 2866--2870.

\bibitem{Denise1995}
{\sc Denise, A., and Simion, R.}
\newblock Two combinatorial statistics on dyck paths.
\newblock {\em Discrete Math. 137}, 1-3 (1995), 155--176.

\bibitem{Deutsch1999}
{\sc Deutsch, E.}
\newblock Dyck path enumeration.
\newblock {\em Discrete Math. 204}, 1-3 (June 1999), 167--202.

\bibitem{Dias2002}
{\sc Dias, Z., and Meidanis, J.~a.}
\newblock Sorting by prefix transpositions.
\newblock In {\em String Process. Inf. Retr.}, A.~Laender and A.~Oliveira,
  Eds., vol.~2476 of {\em Lecture Notes in Computer Science}. Springer Berlin
  Heidelberg, 2002, pp.~65--76.

\bibitem{Elder2005}
{\sc Elder, M., and Vatter, V.}
\newblock Problems and conjectures presented at the third international
  conference on permutation patterns, university of florida, march 7-11, 2005.
\newblock arXiv:0505504v1 [math.CO], 8 pp.

\bibitem{sergithesis}
{\sc Elizalde, S.}
\newblock Statistics on pattern-avoiding permutations.
\newblock {\em PhD Thesis, Massachusetts Inst. Technol.\/} (2004).

\bibitem{GenomeBook}
{\sc Fertin, G., Labarre, A., Rusu, I., Tannier, E., and Vialette, S.}
\newblock {\em Combinatorics of genome rearrangements}.
\newblock The MIT Press, June 2009.

\bibitem{flajolet}
{\sc Flajolet, P., and Sedgewick, R.}
\newblock {\em Analytic combinatorics}.
\newblock Cambridge University Press, Cambridge, 2009.

\bibitem{Fox}
{\sc Fox, J.}
\newblock {S}tanley-{W}ilf limits are typically exponential.
\newblock arXiv:1310.8378 [math.CO], 13 pp.

\bibitem{BillGates}
{\sc Gates, W.~H., and Papadimitriou, C.~H.}
\newblock Bounds for sorting by prefix reversal.
\newblock {\em Discrete Math. 27}, 1 (1979), 47--57.

\bibitem{Gessel1990}
{\sc Gessel, I.~M.}
\newblock Symmetric functions and {P}-recursiveness.
\newblock {\em J. Comb. Theory, Ser. A 53}, 2 (1990), 257--285.

\bibitem{gouldenjackson1}
{\sc Goulden, I.~P., and Jackson, D.~M.}
\newblock An inversion theorem for cluster decompositions of sequences with
  distinguished subsequences.
\newblock {\em J. London Math. Soc. 20}, 3 (1979), 567--576.

\bibitem{gouldenjackson2}
{\sc Goulden, I.~P., and Jackson, D.~M.}
\newblock {\em Combinatorial enumeration}.
\newblock Dover Publications, Inc., Mineola, NY, 2004.
\newblock With a foreword by Gian-Carlo Rota, Reprint of the 1983 original.

\bibitem{Guibert2001}
{\sc Guibert, O., Pergola, E., and Pinzani, R.}
\newblock Vexillary involutions are enumerated by {M}otzkin numbers.
\newblock {\em Ann. Comb. 5}, 2 (2001), 153--174.

\bibitem{HigmansThm}
{\sc Higman, G.}
\newblock Ordering by divisibility in abstract algebras.
\newblock {\em Proc. London Math. Soc. 2\/} (1952), 326--336.

\bibitem{me-fixpat}
{\sc Homberger, C.}
\newblock Counting fixed-length permutation patterns.
\newblock {\em Online J. Anal. Comb. 7\/} (2012), 12 pp.

\bibitem{me-expat}
{\sc Homberger, C.}
\newblock Expected patterns in permutation classes.
\newblock {\em Electron. J. Comb. 19}, 3 (2012), Paper 43, 12 pp.

\bibitem{me-polyclass}
{\sc Homberger, C., and Vatter, V.}
\newblock On the effective and automatic enumeration of polynomial permutation
  classes.
\newblock arXiv:1308.4946 [math.CO], 10 pp.

\bibitem{polyclass-algo}
{\sc Homberger, C., and Vatter, V.}
\newblock Poly{C}lass algorithm, 2013.
\newblock Published online at \url{http://github.com/cheyneh/polyclass}.

\bibitem{SophieVince}
{\sc Huczynska, S., and Vatter, V.}
\newblock Grid classes and the {F}ibonacci dichotomy for restricted
  permutations.
\newblock {\em Electron. J. Comb. 13}, 1 (2006), Research Paper 54, 14 pp.
  (electronic).

\bibitem{Jaggard2002}
{\sc Jaggard, A.~D.}
\newblock Prefix exchanging and pattern avoidance by involutions.
\newblock {\em Electron. J. Comb. 9}, 2 (2003), Research paper 16, 24.
\newblock Permutation patterns (Otago, 2003).

\bibitem{Janson2014}
{\sc Janson, S., Nakamura, B., and Zeilberger, D.}
\newblock On the asymptotic statistics of the number of occurrences of multiple
  permutation patterns.
\newblock arXiv:1312.3955 [math.CO], 18 pp.

\bibitem{Klazar2003}
{\sc Kaiser, T., and Klazar, M.}
\newblock On growth rates of closed permutation classes.
\newblock {\em Electron. J. Comb. 9}, 2 (Oct. 2002), 20 pp.

\bibitem{Kaplansky}
{\sc Kaplansky, I.}
\newblock The asymptotic distribution of runs of consecutive elements.
\newblock {\em Ann. Math. Stat. 16\/} (1945), 200--203.

\bibitem{Dweighter}
{\sc Kleitman, D.~J., Kramer, E., Conway, J.~H., Bell, S., and Dweighter, H.}
\newblock Problems and {S}olutions: {E}lementary {P}roblems: {E}2564-{E}2569.
\newblock {\em Am. Math. Mon. 82}, 10 (1975), 1009--1010.

\bibitem{Knuth}
{\sc Knuth, D.~E.}
\newblock {\em The Art of Computer Programming. {V}ol 1: {F}undamental
  Algorithms}.
\newblock Addison-Wesley Publishing Co., Reading, Mass., 1969.
\newblock Sorting and searching, Addison-Wesley Series in Computer Science and
  Information Processing.

\bibitem{Krattenthaler2001}
{\sc Krattenthaler, C.}
\newblock Permutations with restricted patterns and dyck paths.
\newblock {\em Adv. Appl. Math. 27}, 2-3 (2001), 17 pp.

\bibitem{Kremer2000}
{\sc Kremer, D.}
\newblock Permutations with forbidden subsequences and a generalized
  {S}chr\"{o}der number.
\newblock {\em Discrete Math. 218}, 1-3 (2000), 121--130.

\bibitem{KremerPS}
{\sc Kremer, D.}
\newblock Postscript: ``{P}ermutations with forbidden subsequences and a
  generalized {S}chr\"{o}der number'' [{D}iscrete {M}ath.\ {218} (2000), no.\
  1-3, 121--130; {MR}1754331 (2001a:05005)].
\newblock {\em Discrete Math. 270}, 1-3 (2003), 333--334.

\bibitem{PercyBook}
{\sc MacMahon, P.~A.}
\newblock {\em Combinatory analysis}.
\newblock Two volumes (bound as one). Chelsea Publishing Co., New York, 1960.

\bibitem{MarcusTardos}
{\sc Marcus, A., and Tardos, G.}
\newblock Excluded permutation matrices and the {S}tanley-{W}ilf conjecture.
\newblock {\em J. Comb. Theory, Ser. A 107}, 1 (2004), 153--160.

\bibitem{SubsDecomp}
{\sc M\"{o}hring, R.~H., and Radermacher, F.~J.}
\newblock Substitution decomposition for discrete structures and connections
  with combinatorial optimization.
\newblock In {\em Algebr. Comb. Methods Oper. Res. Proc. Work. Algebr. Struct.
  Oper. Res.}, R.~A. C.-G. {R.E. Burkard} and U.~Zimmermann, Eds., vol.~95 of
  {\em North-Holland Mathematics Studies}. North-Holland, 1984, pp.~257--355.

\bibitem{MurphyVatter}
{\sc Murphy, M.~M., and Vatter, V.}
\newblock Profile classes and partial well-order for permutations.
\newblock {\em Electron. J. Comb. 9}, 2 (2002), Research paper 17, 30 pp.
  (electronic).
\newblock Permutation patterns (Otago, 2003).

\bibitem{Noonan1999}
{\sc Noonan, J., and Zeilberger, D.}
\newblock The goulden-jackson cluster method: extensions, applications and
  implementations.
\newblock {\em J. Differ. Equations Appl.\/} (1999), 1--17.

\bibitem{pantone2013}
{\sc Pantone, J.}
\newblock The enumeration of permutations avoiding 3124 and 4312.
\newblock arXiv:1309.0832 [math.CO], 21 pp.

\bibitem{pemantle}
{\sc Pemantle, R., and Wilson, M.~C.}
\newblock {\em Analytic Combinatorics in Several Variables}.
\newblock Cambridge University Press, New York, NY, USA, 2013.

\bibitem{A=B}
{\sc Petkov\v{s}ek, M., Wilf, H.~S., and Zeilberger, D.}
\newblock {\em A=B}.
\newblock A K Peters Ltd., Wellesley, MA, 1996.

\bibitem{CompBio}
{\sc Pevzner, P.~A.}
\newblock {\em Computational molecular biology}.
\newblock Computational Molecular Biology. MIT Press, Cambridge, MA, 2000.
\newblock An algorithmic approach, A Bradford Book.

\bibitem{Regev}
{\sc Regev, A.}
\newblock Asymptotic values for degrees associated with strips of {Y}oung
  diagrams.
\newblock {\em Adv. Math. (N. Y). 41}, 2 (1981), 115--136.

\bibitem{Rudolph2013}
{\sc Rudolph, K.}
\newblock Pattern popularity in 132-avoiding permutations.
\newblock {\em Electron. J. Comb. 20}, 1 (2013), Paper 8, 15 pp.

\bibitem{Simion1985}
{\sc Simion, R., and Schmidt, F.~W.}
\newblock Restricted permutations.
\newblock {\em Eur. J. Comb. 6}, 4 (1985), 383--406.

\bibitem{Smith2006}
{\sc Smith, R.}
\newblock Permutation reconstruction.
\newblock {\em Electron. J. Comb. 13}, 1 (2006), Note 11, 8.

\bibitem{Stankova1994}
{\sc Stankova, Z.~E.}
\newblock Forbidden subsequences.
\newblock {\em Discrete Math. 132}, 1-3 (1994), 291--316.

\bibitem{Stanley2}
{\sc Stanley, R.~P.}
\newblock {\em Enumerative combinatorics. {V}ol. 2}, vol.~62 of {\em Cambridge
  Studies in Advanced Mathematics}.
\newblock Cambridge University Press, Cambridge, 1999.
\newblock With a foreword by Gian-Carlo Rota and appendix 1 by Sergey Fomin.

\bibitem{Stanley1}
{\sc Stanley, R.~P.}
\newblock {\em Enumerative combinatorics. {V}olume 1}, second~ed., vol.~49 of
  {\em Cambridge Studies in Advanced Mathematics}.
\newblock Cambridge University Press, Cambridge, 2012.

\bibitem{Steingrimsson2013}
{\sc Steingr\'{\i}msson, E., and McNamara, P. R.~W.}
\newblock On the topology of the permutation pattern poset.
\newblock arXiv:1305.5569 [math.CO], 27 pp.

\bibitem{SteveWaton2007}
{\sc {Steve Waton}}.
\newblock On permutation classes generated by token passing networks, gridding
  matrices and pictures: Three flavours of involvement.
\newblock {\em PhD Thesis, Univ. St Andrews\/} (2007).

\bibitem{Tauraso2006}
{\sc Tauraso, R.}
\newblock The dinner table problem: the rectangular case.
\newblock {\em Integers 6\/} (2006), A11, 13.

\bibitem{oeis}
The {O}n-{L}ine {E}ncyclopedia of {I}nteger {S}equences.
\newblock Published Electronically at \url{http://oeis.org}., 2010.

\bibitem{Vatter2010}
{\sc Vatter, V.}
\newblock Permutation classes of every growth rate above 2.48188.
\newblock {\em Mathematika 56}, 1 (2010), 182--192.

\bibitem{Vatter2011}
{\sc Vatter, V.}
\newblock Small permutation classes.
\newblock {\em Proc. London Math. Soc.\/} (2011), 38 pp.

\bibitem{RegInsEnc}
{\sc Vatter, V.}
\newblock Finding regular insertion encodings for permutation classes.
\newblock {\em J. Symb. Comput. 47}, 3 (2012), 259--265.

\bibitem{Watterson1982}
{\sc Watterson, G.~A., Ewens, W.~J., Hall, T.~E., and Morgan, A.}
\newblock The chromosome inversion problem.
\newblock {\em J. Theor. Biol. 99}, 1 (1982), 1--7.

\bibitem{WestDiss}
{\sc West, J.}
\newblock Permutations with forbidden subsequences and stack-sortable
  permutations.
\newblock {\em PhD Thesis, Massachusetts Inst. Technol.\/} (1990).

\bibitem{wilfbook}
{\sc Wilf, H.~S.}
\newblock {\em generatingfunctionology}, third~ed.
\newblock A K Peters, Ltd., Wellesley, MA, 2006.

\bibitem{Wolfowitz}
{\sc Wolfowitz, J.}
\newblock Note on runs of consecutive elements.
\newblock {\em Ann. Math. Stat. 15\/} (1944), 97--98.

\end{thebibliography}

% \printindex

\end{document}